\RequirePackage{tikz}
\documentclass[pdflatex,sn-mathphys]{sn-jnl}
\usepackage{amsmath}
\usepackage{dutchcal}
\usepackage{mathtools}
\usepackage{hyperref}
\usepackage{float}
\usepackage{graphicx}
\usepackage[all]{xy}
\usepackage{url}

\usepackage{yfonts}
\usepackage{etoolbox}
\usepackage{nomencl}
\makenomenclature

\usepackage{tikz-cd}
\usepackage{pgfplots}
\usepackage{thm-restate}

\definecolor{darkdgmcolor}{rgb}{0.0, 0.0, 0.8}
\definecolor{darkgreen}{rgb}{0.0, 0.8, 0.0}
\definecolor{purple}{RGB}{153,50,204}
\definecolor{dgmcolor}{RGB}{255,20,147}
\definecolor{barccolor}{RGB}{20,147,255}

\newcommand{\linterval}{[}
\newcommand{\rinterval}{]}

\newcommand{\cupmodule}{\Phi} 
\newcommand{\invariant}{\mathbf{I}} 

\newcommand{\N}{\mathbb{N}}
\newcommand{\bbR}{\mathbb{R}}
\newcommand{\bbZ}{\mathbb{Z}}
\newcommand{\bbS}{\mathbb{S}} 
\newcommand{\bbT}{\mathbb{T}} 
\newcommand{\barc}{\mathbf{barc}}

\newcommand{\Afunc}{A_\bullet}
\newcommand{\Bfunc}{B_\bullet}
\newcommand{\Zfunc}{Z_\bullet}
\newcommand{\Xfunc}{X_{\bullet}}
\newcommand{\Yfunc}{Y_\bullet}
\newcommand{\Ffunc}{F_\bullet}
\newcommand{\Gfunc}{G_\bullet}
\newcommand{\Mfunc}{M_\bullet}
\newcommand{\Rfunc}{R_\bullet}
\newcommand{\Nfunc}{N_\bullet}

\newcommand{\bfJ}{\mathbf{J}}
\newcommand{\bfH}{\mathbf{H}}

\newcommand{\VR}{\operatorname{VR}}

\newcommand{\Vflag}{V_\star}
\newcommand{\Wflag}{W_\star}

\newcommand{\pflag}{V_{\star,\bullet}}

\newcommand{\ob}{\operatorname{Ob}}
\newcommand{\mor}{\operatorname{Mor}}
\newcommand{\id}{\operatorname{id}}
\newcommand{\diam}{\operatorname{diam}}
\newcommand{\image}{\operatorname{im}}

\newcommand{\fdim}{\mathbf{dim}}
\newcommand{\rank}{\mathbf{rk}}
\newcommand{\len}{\mathbf{len}}
\newcommand{\cupprod}{\mathbf{cup}}
\newcommand{\dgm}{\mathbf{dgm}}
\newcommand{\cat}{\mathbf{cat}}
\newcommand{\dimension}{\mathbf{dim}}

\newcommand{\catC}{\mathbcal{C}}
\newcommand{\catN}{\mathbcal{N}}
\newcommand{\Set}{\mathbcal{S\!e\!t}}
\newcommand{\Vect}{\mathbcal{V\!e\!c}}
\newcommand{\Flag}{\mathbcal{F\!l\!a\!g}}
\newcommand{\grFlag}{\mathbcal{GF\!l\!a\!g}}
\newcommand{\GVect}{\mathbcal{GV\!e\!c}}
\newcommand{\Ring}{\mathbcal{GR\!i\!n\!g}}
\newcommand{\Alg}{\mathbcal{GA\!l\!\!g}}
\newcommand{\Top}{\mathbcal{T\!\!o\!p}}
\newcommand{\Int}{\mathbcal{I\!nt}}

\newcommand{\field}{K}

\newcommand{\di}{d_{\mathrm{I}}}
\newcommand{\db}{d_{\mathrm{B}}}
\newcommand{\dhi}{d_{\mathrm{HI}}}
\newcommand{\de}{d_{\mathrm{E}}}
\newcommand{\dgh}{d_{\mathrm{GH}}}

\newcommand{\bsigma}{\boldsymbol{\sigma}}
\newcommand{\sigmadgmX}{\dgm\left(\cupprod(\Xfunc),\bsigma\right)}

\newcommand{\Coordinate}[2]
{ \coordinate (#1) at (#2);
}

\jyear{2021}

\theoremstyle{thmstyleone}

\newtheorem{theorem}{Theorem}
\newtheorem{proposition}{Proposition}[section]
\newtheorem{lemma}[proposition]{Lemma}
\newtheorem{corollary}[proposition]{Corollary}
\newtheorem{definition}[proposition]{Definition}
\newtheorem{example}[proposition]{Example}

\newtheorem{remark}[proposition]{Remark}

\pgfplotsset{compat=1.18}
\begin{document}

\title[Persistent Cup Product Structures and Related Invariants]{Persistent Cup Product Structures and Related Invariants}

\author[1,2]{\fnm{Facundo} \sur{M\'emoli}}\email{memoli@math.osu.edu}
\equalcont{These authors contributed equally to this work.}

\author*[3]{\fnm{Anastasios} \sur{Stefanou}}\email{stefanou@uni-bremen.de}

\author[1]{\fnm{Ling} \sur{Zhou}}\email{zhou.2568@osu.edu}
\equalcont{These authors contributed equally to this work.}

\affil[1]{\orgdiv{Department of Mathematics}, \orgname{The Ohio State University}, \orgaddress{
\city{ Columbus}, \postcode{43210}, \state{Ohio}, \country{USA}}}

\affil[2]{\orgdiv{Department of Computer Science and Engineering}, \orgname{The Ohio State University}, \orgaddress{ 
\city{ Columbus}, \postcode{43210}, \state{Ohio}, \country{USA}}}

\affil[3]{\orgdiv{Department of Mathematics and Computer Science}, \orgname{University of Bremen}, \orgaddress{ 
\city{Bremen}, \postcode{28359}, \country{Germany}}}

\abstract{One-dimensional persistent homology is arguably the most important and heavily used computational tool in topological data analysis. Additional information can be extracted from datasets by studying multi-dimensional persistence modules and by utilizing cohomological ideas, e.g.~the cohomological cup product. \\

In this work, given a single parameter filtration, we investigate a certain 2-dimensional persistence module structure associated with persistent cohomology, where one parameter is the cup-length $\ell\geq0$ and the other is the filtration parameter. This new persistence structure, called the \emph{persistent cup module}, is induced by the cohomological cup product and adapted to the persistence setting.
Furthermore, we show that this persistence structure is stable.
By fixing the cup-length parameter $\ell$, we obtain a 1-dimensional persistence module, called the persistent $\ell$-cup module, and again  show it is stable in the interleaving distance sense, and study their associated generalized persistence diagrams.\\

In addition, we consider a generalized notion of a \emph{persistent invariant},  which  extends both the \emph{rank invariant} (also referred to as \emph{persistent Betti number}), Puuska's rank invariant induced by epi-mono-preserving invariants of abelian categories, and the recently-defined \emph{persistent cup-length invariant}, and  we establish their stability. This generalized notion of persistent invariant also enables us to lift the Lyusternik-Schnirelmann (LS) category of topological spaces to a novel stable persistent invariant of filtrations, called the \emph{persistent LS-category invariant}. 
}


\maketitle

\tableofcontents

\medskip
\noindent\textbf{Acknowledgements.} 
FM and LZ were partially supported by the NSF through grants RI-1901360, CCF-1740761, and CCF-1526513, and DMS-1723003. AS was supported by NSF through grants CCF-1740761,  
DMS-1440386, 
RI-1901360 and the Dioscuri program initiated by the Max Planck Society, jointly managed
with the National Science Centre (Poland), and mutually funded by the
Polish Ministry of Science and Higher Education and the German Federal
Ministry of Education and Research. We thank Marco Contessoto for discussion of materials described in \textsection \ref{sec:Persistent-cup-length}.

\section{Introduction}

\medskip
\noindent\textbf{Persistent Homology in TDA.} 
In \textit{Topological Data Analysis (TDA)}, one studies the evolution of homology across a filtration of spaces, called \textit{persistent homology} \citep{frosini1990distance,frosini1992measuring,robins1999towards,zomorodian2005computing,cohen2007stability,edelsbrunner2008Persistent,carlsson2009topology,carlsson2020persistent}. 
Persistent homology is able to extract both the time when a topological feature (e.g.~a component, loop, cavity) is `born' and the time when it `dies'.
The collection of these birth-death pairs (real intervals) constitute the \textit{barcode}, also called the \textit{persistence diagram}, of the filtration (depending on the manner in which they are visualized).

\medskip
\noindent\textbf{Cohomology Rings in TDA.}
In the case of \textit{cohomology}, which is dual to the case of homology for a given field $K$, one studies linear functions 
on (homology) chains, known as \textit{cochains}. Cohomology has a graded ring structure, inherited from the \textit{cup product} operation on cochains, denoted by $\smile:\bfH^p( X) \times\bfH^q( X)\to \bfH^{p+q}( X)$ for a space $ X$ and dimensions $p,q\geq 0$; see \citep[Sec.~~48 and Sec.~~68]{munkres1984elements} and \citep[Ch. 3, Sec.~3.D]{hatcher2000}. This makes the \textit{cohomology ring} a richer structure than homology (vector spaces). 

\textit{Persistent cohomology} has been studied in \citep{de2011dualities,de2011persistent,dlotko2014simplification,bauer2021ripser,kang2021evaluating},
without exploiting the ring structure induced by the cup product. Works which do attempt to exploit this ring structure include \citep{gonzalez2003hha,kaczynski2010computing} in the standard case and 
\citep{huang2005cup,yarmola2010persistence,aubrey2011persistent,herscovich2018higher,belchi2021a,polanco2021,lupo2022persistence,contessoto_et_al:LIPIcs.SoCG.2022.31} at the persistent level.

In \citep{huang2005cup}, the author applies the persistence algorithm toward calculating a set of invariants related to the cup products in the cohomology ring of a
space. In \citep{yarmola2010persistence}, the author studies an algebraic substructure of the cohomology ring. 
In \citep{polanco2021}, the authors study a persistence based approach for differentiating quasi-periodic and periodic signals which is inherently based on cup products.

In \citep{aubrey2011persistent}, the author develops a setting for persistent characteristic classes and constructs algorithms for (i) finding the Poincare Dual to a homology class, (ii) decomposing cohomology classes and (iii) deciding when a cohomology class is a Steenrod square. In \citep{lupo2022persistence} the authors establish the notion of persistent Steenrod modules by incorporating the Steenrod square operation into the persistence computational pipeline, and they implement an algorithm to compute the barcode of persistent Steenrod modules \citep{steenroder}.

Assuming an embedding of a simplicial set into $\bbR^n$, the author of \citep{herscovich2018higher} studies a notion of barcodes (together with a suitable extension of the bottleneck distance) which absorb information from a certain $A_\infty$-algebra structure on persistent cohomology. In \citep{belchi2021a}, the authors study the structure and stability of a family of barcodes that absorb information from an $A_\infty$-coalgebra structure on persistent homology. See also \citep{belchi2015a}. 

In \citep{ginot2019multiplicative}, the authors study several interleaving-type distances on persistent cohomology by considering different algebraic structures (including the natural ring structure) and study the stability of the persistent cohomology for filtrations. 

In our previous joint work with Contessoto \citep{contessoto_et_al:LIPIcs.SoCG.2022.31}, we tackled the question of quantifying the \emph{evolution} of the cup product structure across a filtration through introducing a \emph{polynomial-time computable} invariant which is induced from the notion of \emph{cup-length}: the maximal number of cocycles (in dimensions 1 and above) having non-zero cup product. We call this invariant the \textit{persistent cup-length invariant}, and we identify a tool - the \textit{persistent cup-length diagram} (associated to a family of representative cocycles $\bsigma$ of the barcode) as well as a polynomial-time algorithm to compute it. In Sec.~\ref{sec:Persistent-cup-length}, we recall and provide more details for the mathematical results in \citep{contessoto_et_al:LIPIcs.SoCG.2022.31}. 
Readers interested in the algorithmic part should still refer to the original paper \citep{contessoto_et_al:LIPIcs.SoCG.2022.31}. 

The goal of this paper is to develop more general notions of persistent invariants that can extract additional information from the cup product operation than just the persistent cup-length invariant, including the persistent LS-category (see Sec.~\ref{sec:ls-cat}) and the persistent cup modules (see Sec.~\ref{sec:cup module}). 

\medskip\noindent\textbf{Some invariants related to the cup product.} 
An \textit{invariant} in standard topology is a quantity assigned to a given topological space that remains invariant under a certain class of maps. This invariance helps in discovering, studying, and classifying properties of spaces when the class of maps is that of homotopy equivalences. Beyond \emph{Betti numbers}, examples of classical invariants are: the \emph{Lyusternik-Schnirelmann category (LS-category)} of a space $ X$, defined as the minimal integer $k\geq 1$ such that there is an open cover $\{U_i\}_{i=1}^k$ of $ X$ such that each inclusion map $U_i\hookrightarrow  X$ is null-homotopic, and the \emph{cup-length invariant},  which is the maximum number of positive-degree cocycles having non-zero cup product.
While being relatively more informative, the LS-category is difficult to compute \citep{cornea2003lusternik}, and the rational LS-category\footnote{The rational LS-category of a space $X$ is the smallest LS-category of spaces that are rational homotopy equivalent to $X$\citep{felix1982rational}. And two spaces are rational homotopy equivalent if there is a map between them that induces isomorphism between homotopy groups of the two spaces.} is known to be NP-hard to compute \citep{lechuga2000complexity}. 
The cup-length invariant, as a lower bound of the LS-category \citep{rudyak1999analytical,Rudyak1999ONCW}, serves as a computable estimate for the LS-category.
Another well known invariant which can be estimated through the cup-length is the so-called \textit{topological complexity} \citep{SMALE198781,farber2003topological,sarin2017cup}.

\subsection*{Our contributions.} 

Let $\Top$ denote the category of (compactly generated weak Hausdorff) topological spaces.\footnote{We are following the convention from \citep{blumberg2017universality}.} Throughout the paper, by a (topological) space we refer to an object in $\Top$, and by a \emph{persistent space} we mean a functor from the poset category $(\bbR,\leq )$ to $\Top$. A filtration (of spaces) is an example of a persistent space where the transition maps are given by inclusions. This paper considers only persistent spaces with a discrete set of critical values. In addition, all (co)homology groups are assumed to be taken over a field $\field$.
We denote by $\Int_\omega$ the set of intervals of type $\omega$, where $\omega$ can be 
any one of the four types: open-open, open-closed, closed-open and closed-closed. 
Results in this paper apply to all four situations, so for simplicity of notation, we state our results only for closed-closed intervals and omit $\omega$ unless otherwise stipulated.

Let $(\catN,\leq )$ be a poset category (e.g., $\catN=\N,\N^\infty$ or $\N^{\infty,\infty}$) with a partial order $\leq$. Let $(\catN,\leq )^{\mathrm{op}}$ be the opposite category\footnote{The opposite category $\catC^{\mathrm{op}}$ of a category $\catC$ is the category whose objects are the same as $\catC$, but whose arrows are the arrows of $\catC$ with the reverse direction.} of $(\catN,\leq )$, i.e. a poset category on $\catN$ equipped with the converse (or dual) relation $\geq$. 
In Sec.~\ref{sec:pers-inv}, for any given category $\catC$, we define the \emph{$\catN$-valued categorical invariants} to be maps $\invariant :\ob(\catC)\sqcup\mor(\catC)\to\catN$ assigning values to both objects and morphisms in $\catC$, such that $\invariant (\id _X)=\invariant (X)$ for all $X\in \ob(\catC)$ and $\invariant (g\circ f)\leq \min\{\invariant (f),\invariant (g)\}$ for any $f:X\to Y$, $g:Y\to Z$ in $\catC$; see Defn.~\ref{def:categorical invariant} and Prop.~\ref{prop:eqvi def for cat inv}. Compared with classical invariants, which are usually only defined on the objects of the underlying category, categorical invariants are also defined on the morphisms. Notice that categorical invariants are always invariant under isomorphisms (see Rmk.~\ref{rmk:invariant preserves iso}).

Here we are abusing the name `invariant', and the standard notion of invariant is more closely related to what we call epi-mono invariant, i.e. invariants that are non-increasing under regular epimorphisms and non-decreasing under monomorphism.

The categorical invariants from Defn.~\ref{def:categorical invariant} can be seen as a generalization of the notion of epi-mono invariant mentioned in Ex.~\ref{ex:surjective-injective-inv} and of other invariants that appeared in TDA literature (see Sec.~\ref{sec:comparison-long} for a detailed comparison):
\begin{itemize}
    \item For $\catC$ an abelian category, the notion of \emph{epi-mono-respecting pre-orders on $\catC$} of Puuska \citep[Defn.~3.2]{puuska2017erosion} is equivalent to the restriction of our notion of epi-mono invariant to abelian categories.
     \item For $\catC$ any category, the notion of \emph{categorical persistence function} of Bergomi et al.~\citep[Defn.~3.2]{bergomi2019rank} is a categorical invariant satisfying an additional inequality.
   \item   
    For $\catC$ a regular category, an epi-mono invariant 
    is a special case of a categorical invariant (see \ref{sec:comparison-long} for the details). Examples of epi-mono invariants, include \emph{rank functions} of Bergomi et al.~\citep[Defn.~2.1]{bergomi2019rank}  and \emph{amplitudes} of Giunti et al.~\citep[Defn.~3.1]{giunti2021amplitudes}.
\end{itemize}  

The persistent LS-category invariant, which we introduce in this work, cannot be realized as an invariant of the above types, making our notion of categorical invariant a non-trivial generalization.

Given any \emph{persistent object}, i.e. a functor $\Ffunc:(\bbR,\leq)\to\catC$, a categorical invariant $\invariant $ gives rise to a \emph{persistent (categorical) invariant}\footnote{In this paper, by `persistent invariant' we always mean one such invariant lifted from a categorical invariant.} defined as the functor $\invariant (\Ffunc):(\Int,\subseteq)\to (\catN,\leq)^{\mathrm{op}}$ 
sending each interval $[a,b]$ to the $\invariant $-invariant of the transition map $f_a^b$, cf. Defn.~\ref{def:p_invariant}. 
For example, the well-known \emph{rank invariant} 
\citep[Defn.~11]{carlsson2007theory}
of a persistent module is a persistent invariant induced by the dimension map $\dimension:\ob(\Vect)\sqcup\mor(\Vect)\to\N$ defined by sending each vector space to its dimension and each linear map to the dimension of its image. Here $\Vect$ denotes the category of finite-dimensional vector spaces over a given field $K$.

In Sec.~\ref{sec:Persistent-cup-length}, we realize the cup-length invariant as a categorical invariant by defining the cup-length of a map to be the cup-length of its image. We then lift the cup-length invariant to a persistent invariant: for a persistent space $\Xfunc:(\bbR,\leq)\to \Top$ with $t\mapsto  X_t$, the \textbf{persistent cup-length invariant} $\cupprod(\Xfunc):\Int\to\N$ of $\Xfunc$, see Defn.~\ref{def:cup_l_func},
is defined as the functor 
from $(\Int,\subseteq)$ to $(\N,\geq )$ of non-negative integers, which assigns to each interval $[a,b]$ the cup-length of the image ring $\image \big(\bfH^*( X_b) \to\bfH^*( X_a) \big)$.
See also \citep[Sec.~2]{contessoto_et_al:LIPIcs.SoCG.2022.31} for details. 

In Sec.~\ref{sec:ls-cat}, we recall the notion of the LS-category of a map first introduced in \citep{fox1941lusternik} and more carefully studied in \citep[Defn.~1.1]{berstein1962category}, 
and see that the LS-category is a categorical invariant of topological spaces. We define the \textbf{persistent LS-category invariant} of a persistent space $\Xfunc$ to be the function
$\cat(\Xfunc):\Int\to\N$ of $\Xfunc$ assigning to each interval $[a,b]$ the LS-category of the transition map $ X_a\to  X_b
$; see Defn.~\ref{def:pers-ls}.
In Prop.~\ref{prop:cup<cat}, we prove that in analogy with the standard fact that cup-length is a lower bound for the LS-category their persistent versions also satisfy that inequality: for any interval $[a,b]$,
$$\cupprod(\Xfunc)([a,b])\leq \cat(\Xfunc)([a,b]).$$ 

See Fig.~\ref{fig:torus_sphere_VR_flag} for examples of the persistent cup-length invariant and the persistent LS-category invariant. Although the latter invariant is pointwisely bounded below by the former, the latter is not necessarily stronger in terms of distinguishing topological filtrations; see Ex.~\ref{ex:cat v.s. cup p-version}.

\begin{figure}[H]
\centering
     \begin{tikzpicture}[scale=0.6]
    \begin{axis} [ 
    ticklabel style = {font=\large},
    axis y line=middle, 
    axis x line=middle,
    ytick={0.5,0.67,0.95},
    yticklabels={$\tfrac{\pi}{2}$,$\tfrac{2\pi}{3}$,$\pi$},
    xtick={0.5,0.58,0.95},
    xticklabels={$\tfrac{\pi}{2}$,$\zeta_3$, $\pi$},
    xmin=0, xmax=1.1,
    ymin=0, ymax=1.1,]
    \addplot [mark=none] coordinates {(0,0) (1,1)};
    \addplot [thick,color=dgmcolor!40!white,fill=dgmcolor!40!white, 
                    fill opacity=0.45]coordinates {
            (0,0.58)
            (0,0)
            (0.58,0.58)
            (0,0.58)};
    \addplot [thick,color=black!10!white,fill=black!10!white, 
                    fill opacity=0.4]coordinates {
            (0.58,0.95)
            (0.58,0.58)
            (0.95,0.95)
            (0.58,0.95)};
    \node[mark=none] at (axis cs:.25,.45){\textsf{2}};
    \end{axis}
    \end{tikzpicture}
    \hspace{1.5cm}
    \begin{tikzpicture}[scale=0.6]
    \begin{axis} [ 
    ticklabel style = {font=\large},
    axis y line=middle, 
    axis x line=middle,
    ytick={0.5,0.67,0.95},
    yticklabels={$\tfrac{\pi}{2}$,$\tfrac{2\pi}{3}$,$\pi$},
    xtick={0.5,0.6,0.95},
    xticklabels={$\tfrac{\pi}{2}$,$\zeta_2$, $\pi$},
    xmin=0, xmax=1.1,
    ymin=0, ymax=1.1,]
    \addplot [mark=none] coordinates {(0,0) (1,1)};
    \addplot [thick,color=dgmcolor!40!white,fill=dgmcolor!40!white, 
                    fill opacity=0.45]coordinates {
            (0,0.6)
            (0,0)
            (0.6,0.6)
            (0,0.6)};
    \addplot [thick,color=black!10!white,fill=black!10!white, 
                    fill opacity=0.4]coordinates {
            (0.6,0.95)
            (0.6,0.6)
            (0.95,0.95)
            (0.6,0.95)};
    \node[mark=none] at (axis cs:.25,.45){\textsf{2}};
    \end{axis}
    \end{tikzpicture}
\caption{The persistent invariants $\invariant(\VR\left(\bbT^2\vee\bbS^3\right))$ (left) and $\invariant(\VR\left((\bbS^1\times\bbS^2)\vee\bbS^1\right))$ (right), respectively, for $\invariant=\cupprod$ or $\cat$. Here, $\zeta_2=\arccos(-\tfrac{1}3)\approx 0.61\pi$, $\zeta_3=\arccos(-\tfrac{1}{4})\approx 0.58\pi$ and gray regions means undetermined values. See Ex.~\ref{sec:flag>cup}.} 
\label{fig:torus_sphere_VR_flag}
\end{figure}

In Sec.~\ref{sec:all-stability}, we establish stability results for persistent invariants. 
We first prove that the erosion distance $\de$ between persistent invariants is bounded above by the interleaving distance $\di$ between the underlying persistent objects (see Sec.~\ref{sec:cat-stab-per-inv}):
\begin{restatable}[$d_{\mathrm I}$-stability of persistent invariants]{theorem}{distab}
\label{thm:stab-per-inv}
Let $\catC$ be a category, and let $\invariant :\ob(\catC)\sqcup\mor(\catC)\to\catN$ be a categorical invariant of $\catC$. 
The persistence $\invariant $-invariant is $1$-Lipschitz stable: for any $\Ffunc,\Gfunc:(\bbR,\leq)\to\catC$,
$$d_{\mathrm{E}}(\invariant (\Ffunc),\invariant (\Gfunc))\leq d_{\mathrm I}(\Ffunc,\Gfunc).$$
\end{restatable}

In Sec.~\ref{sec:homo-stab-per-inv}, for the case of topological spaces, we consider categorical invariants that preserve weak homotopy equivalences, and we strengthen the above stability result by replacing $\di$ with 
the homotopy-interleaving distance $\dhi$ introduced by Blumberg and Lesnick \citep[Defn.~3.6]{blumberg2017universality}. Following the fact that $\dhi$ is stable under the Gromov-Hausdorff distance $\dgh$ between metric spaces (see Prop. \ref{prop:dh-dgh}), we also obtain stability of such categorical invariants in the Gromov-Hausdorff sense:

\begin{restatable}[Homotopical stability]{theorem}{homostab} 
\label{thm:main-stability} 
Let $\invariant $ be a categorical invariant of topological spaces satisfying the condition that for any maps 
$ X\xrightarrow{f}Y\xrightarrow{g}Z\xrightarrow{h}W$ where $g$ is a weak homotopy equivalence, $\invariant (g\circ f)=\invariant (f)$ and $\invariant (h\circ g)=\invariant (h)$. 
Then, for two persistent spaces $\Xfunc,\Yfunc:(\bbR,\leq)\to\Top$, we have
\begin{equation}\label{eq:dE-dHI}
    d_{\mathrm{E}}(\invariant (\Xfunc),\invariant (\Yfunc))\leq d_{\mathrm{HI}}(\Xfunc,\Yfunc).
\end{equation}
For the Vietoris-Rips filtrations $\VR_\bullet(X)$ and $\VR_\bullet(Y)$ of compact metric spaces $X$ and $Y$, we have
\begin{equation}\label{eq:dE-dGH}
d_{\mathrm{E}}\left(\invariant \left(\VR_\bullet (X)\right),\invariant \left(\VR_\bullet (Y)\right)\right)\leq 2\cdot d_{\mathrm{GH}}(X,Y). 
\end{equation} 
\end{restatable}

We apply the above theorem to show that the persistent cup-length invariant and the persistent LS-category are stable:

\begin{restatable}[Homotopical stability of $\cupprod(\cdot)$]{corollary}{homostabcup} 
\label{cor:stab-cup}
For persistent spaces $\Xfunc,\Yfunc:(\bbR,\leq)\to\Top$, the persistent cup-length invariant $\cupprod(\cdot)$ satisfies Eqn. (\ref{eq:dE-dHI}) and Eqn. (\ref{eq:dE-dGH}).
\end{restatable}

\begin{restatable}[Homotopical stability of $\cat(\cdot)$]{corollary}{homostabcat} 
\label{cor:stab-cat}
For persistent CW complexes $\Xfunc,\Yfunc:(\bbR,\leq)\to\Top$, the persistent LS-category $\cat(\cdot)$ satisfies Eqn. (\ref{eq:dE-dHI}) and Eqn. (\ref{eq:dE-dGH}).
\end{restatable}

Notice that the persistent cup-length invariant and persistent LS-category invariant are comparable in the sense that neither invariant is stronger than the other (see Ex.~\ref{ex:cat v.s. cup p-version}), similar to the static case (see Ex.~\ref{ex:cat v.s. cup}).

Through several examples, we show that the persistent cup-length (or LS-category) invariant helps in discriminating filtrations when  persistent homology fails to or has a relatively weak performance in doing so, e.g. \citep[Ex.~13]{contessoto_et_al:LIPIcs.SoCG.2022.31}. Also, in Ex.~\ref{ex:T2-wedge}, we specify suitable metrics on the torus $\bbT^2$ and on the wedge sum $\bbS^1 \vee \bbS^2 \vee \bbS^1$, and compute the erosion distance between their persistent cup-length (or LS-category) invariants and apply Thm.~\ref{thm:main-stability} to obtain a lower bound $\tfrac{\pi}{6}$ for the Gromov-Hausdorff distance between them $\bbT^2$ and $\bbS^1 \vee \bbS^2 \vee \bbS^1$ (see Prop.~\ref{prop:erosion-comp})\footnote{For the case of persistent cup-length invariant, this result was stated without proof in pg.~4 of the conference paper \citep{contessoto_et_al:LIPIcs.SoCG.2022.31}.}. 
We also verify that the interleaving distance between the persistent homology of these two spaces is at most $\tfrac{3}{5}$ of the bound obtained from persistent cup-length (or LS-category) invariants. See Rmk.~\ref{rmk:inter-torus-wedge}.

In Sec.~\ref{sec:cup module}, for a given persistent space $\Xfunc$ and any $\ell\in \N^+$, 
we study the $\ell$-fold product $(\bfH^{ + }(\Xfunc))^{\ell}$ of the persistent (positive-degree) cohomology ring, via the notion of flags of vector spaces. A \emph{flag}\footnote{In the literature, the term `flag' refers to a strictly increasing or decreasing filtration on a vector space. For the purpose of simplicity, we abuse this terminology and define a flag to be a non-increasing filtration.} (of vector spaces) means a non-increasing sequence of vector spaces connected by inclusions, e.g. $V_1\supseteq  V_2\supseteq \cdots$. A flag is said to have \emph{finite depth} if there is some $n$ such that $V_n=0$ (as a consequence, $V_k=0$ for all $k\geq n$). 
Similarly, we call a non-increasing sequence of \emph{graded} vector spaces connected by degree-wise inclusions to be a \emph{graded flag}:
$$\bigoplus_{p\geq 1} V^{p}_1 \supseteq  \bigoplus_{p\geq 1} V^{p}_2 \supseteq  \cdots \supseteq  \bigoplus_{p\geq 1} V^{p}_\ell\supseteq  \cdots.$$

For a topological space $ X$, we define $\cupmodule( X)$ to be the graded flag induced by $\ell$-fold product $(\bfH^{ + }( X))^\ell$ for all $\ell\in \N^+$: 
\[\cupmodule( X):=\, \bfH^{ + }( X)\supseteq  (\bfH^{ + }( X))^{ 2}\supseteq  (\bfH^{ + }( X))^{3}\supseteq \cdots.\]
Let $\Flag$ and $\grFlag$ be the category of finite-depth flags and finite-depth graded flags, respectively. For a persistent space $\Xfunc$, we have the \emph{persistent graded flag} $\cupmodule(\Xfunc):(\bbR,\leq)\to \grFlag$ with $t\mapsto \cupmodule( X_t)$, and we call it the \textbf{persistent cup module} of $\Xfunc$. Indeed, the persistent cup module can be described via the following commutative diagram: for any $t\leq t'$,
\[
\begin{tikzcd}[column sep=small]
\cupmodule( X_t):=
&
\bfH^{ + }( X_t)
\ar[r, phantom, sloped, "\supseteq "]
& (\bfH^{ + }( X_t))^{ 2} 
\ar[r, phantom, sloped, "\supseteq "]
& (\bfH^{ + }( X_t))^{ 3}
\ar[r, phantom, sloped, "\supseteq "]
& \cdots
\\
\cupmodule( X_{t'}):=
\ar[u]
&
\bfH^{ + }( X_{t'})
\ar[u]
\ar[r, phantom, sloped, "\supseteq "]
& (\bfH^{ + }( X_{t'}))^{ 2} 
\ar[u]
\ar[r, phantom, sloped, "\supseteq "]
& (\bfH^{ + }( X_{t'}))^{ 3}
\ar[u]
\ar[r, phantom, sloped, "\supseteq "]
& \cdots
\ar[u]
\end{tikzcd}
\]
The above diagram suggests that the persistent cup module $\cupmodule(\Xfunc)$ has the structure of a 2D persistence module, which we still denote as $\cupmodule(\Xfunc)$ but view as a functor 
$\cupmodule(\Xfunc):(\N^+,\leq)\times(\bbR,\leq) \to \GVect^{\mathrm{op}}\text{ with }(\ell,t)\mapsto (\bfH^{ + }\left( X_t)\right)^{\ell}.$
Two-dimensional persistent modules have wild types of indecomposables in most cases \citep{leszczynski1994representation,leszczynski2000tame,bauer2020cotorsion}, making them difficult to study (see Sec.~\ref{sec:2d cup module} for details). 
Therefore, in Sec.~\ref{sec:flag cup module}, we concentrate on studying $\cupmodule(\Xfunc)$ as a persistent graded flag, and taking the point of view of generalized persistent diagrams \citep[Defn.~7.1]{patel2018generalized}.

Flags can be completely characterized by a non-increasing sequence of integers, where each integer is the dimension of the corresponding vector space (see Prop.~\ref{prop:decomposition of Flag}). We call such non-increasing sequence of integers the \textbf{dimension of a flag}, and write it as
\[\fdim(V_1\supseteq  V_2\supseteq \cdots):=\big(\dimension(V_1),\dimension(V_2),\dots\big).\]
We define the \textbf{rank invariant} of flags as the map $\rank:\Flag\to \N^\infty$ sending each flag to its dimension and each flag morphism to the dimension of its image; see Defn.~\ref{def:rank}. This invariant is clearly a ($\N^\infty$-valued) categorical invariant and thus can be lifted to a persistent invariant. Similarly, we define the \textbf{dimension of a graded flag}, to be a matrix such that each row is the dimension of the flag in the corresponding degree: The dimension $\fdim\left(\bigoplus_{p\geq 1} V^{p}_1 \supseteq  \bigoplus_{p\geq 1} V^{p}_2 \supseteq  \cdots \supseteq \bigoplus_{p\geq 1} V^{p}_\ell \supseteq  \cdots\right)$ is defined as 
\[
\begin{pmatrix}
\dimension(V_1^1) & \dimension(V_2^1) & \cdots & \dimension(V_\ell^1) & \cdots \\
\dimension(V_1^2) & \dimension(V_2^2) & \cdots & \dimension(V_\ell^2) & \cdots \\
\cdots  & \cdots  & \cdots  & \cdots & \cdots \\
\dimension(V_1^p) & \dimension(V_2^p) & \cdots & \dimension(V_\ell^p) & \cdots \\
\cdots  & \cdots  & \cdots  & \cdots & \cdots 
\end{pmatrix},
\]
The \textbf{rank invariant} of graded flags is defined similar to the non-graded case and will also be denoted as $\rank$. See Fig.~\ref{fig:torus_sphere_VR_flag_rank}. 

\begin{figure}[H]
\centering
     \begin{tikzpicture}[scale=0.6]
    \begin{axis} [ 
    ticklabel style = {font=\large},
    axis y line=middle, 
    axis x line=middle,
    ytick={0.5,0.67,0.95},
    yticklabels={$\tfrac{\pi}{2}$,$\tfrac{2\pi}{3}$,$\pi$},
    xtick={0.5,0.58,0.95},
    xticklabels={$\tfrac{\pi}{2}$,$\zeta_3$, $\pi$},
    xmin=0, xmax=1.1,
    ymin=0, ymax=1.1,]
    \addplot [mark=none] coordinates {(0,0) (1,1)};
    \addplot [thick,color=dgmcolor!40!white,fill=dgmcolor!40!white, 
                    fill opacity=0.45]coordinates {
            (0,0.58)
            (0,0)
            (0.58,0.58)
            (0,0.58)};
    \addplot [thick,color=black!10!white,fill=black!10!white, 
                    fill opacity=0.4]coordinates {
            (0.58,0.95)
            (0.58,0.58)
            (0.95,0.95)
            (0.58,0.95)};
    \node[mark=none] at (axis cs:.25,.45){\scriptsize \textsf{$\begin{pmatrix}
    2 & 0\\
    1& 1\\
    1& 0
    \end{pmatrix}$}};
    \addplot [thick,color=dgmcolor!20!white,fill=dgmcolor!20!white, 
                    fill opacity=0.4]coordinates {
            (0,0.58)
            (0.58,0.58)
            (0.58,0.67)
            (0,0.67)
            (0,0.58)};
    \node[mark=none] at (axis cs:.25,.85){\scriptsize \textsf{$\begin{pmatrix}
    2 & 0\\
    1& 1\\
    0& 0
    \end{pmatrix}$}};
    \draw [->] (0.1,0.64) to [out=120,in=240] (0.1,0.8);
\node[mark=none] at (axis cs:0.68,0.21){$\rank(\Phi(\bbT^2\vee\bbS^3))$};
    \end{axis}
    \end{tikzpicture}
    \hspace{1.5cm}
    \begin{tikzpicture}[scale=0.6]
    \begin{axis} [ 
    ticklabel style = {font=\large},
    axis y line=middle, 
    axis x line=middle,
    ytick={0.5,0.67,0.95},
    yticklabels={$\tfrac{\pi}{2}$,$\tfrac{2\pi}{3}$,$\pi$},
    xtick={0.5,0.6,0.95},
    xticklabels={$\tfrac{\pi}{2}$,$\zeta_2$, $\pi$},
    xmin=0, xmax=1.1,
    ymin=0, ymax=1.1,]
    \addplot [mark=none] coordinates {(0,0) (1,1)};
    \addplot [thick,color=dgmcolor!40!white,fill=dgmcolor!40!white, 
                    fill opacity=0.45]coordinates {
            (0,0.6)
            (0,0)
            (0.6,0.6)
            (0,0.6)};
    \addplot [thick,color=black!10!white,fill=black!10!white, 
                    fill opacity=0.4]coordinates {
            (0.6,0.95)
            (0.6,0.6)
            (0.95,0.95)
            (0.6,0.95)};
    \node[mark=none] at (axis cs:.25,.45){\scriptsize \textsf{$\begin{pmatrix}
    2 & 0\\
    1& 0\\
    1& 1
    \end{pmatrix}$}};
    \addplot [thick,color=dgmcolor!20!white,fill=dgmcolor!20!white, 
                    fill opacity=0.4]coordinates {
            (0,0.6)
            (0.6,0.6)
            (0.6,0.67)
            (0,0.67)
            (0,0.6)};
    \node[mark=none] at (axis cs:.25,.85){\scriptsize \textsf{$\begin{pmatrix}
    2 & 0\\
    0& 0\\
    0& 0
    \end{pmatrix}$}};
    \draw [->] (0.1,0.64) to [out=120,in=240] (0.1,0.8);
\node[mark=none] at (axis cs:0.68,0.21){$\rank(\Phi((\bbS^1\times\bbS^2)\vee\bbS^1))$};
    \end{axis}
    \end{tikzpicture}
\caption{Rank invariants $\rank\left(\Phi(\VR\left(\bbT^2\vee\bbS^3\right))\right)$ (left) and $\rank\left(\Phi(\VR\left((\bbS^1\times\bbS^2)\vee\bbS^1\right))\right)$ (right) of persistent cup modules (up to degree $3$) arising from Vietoris-Rips filtrations of $\bbT^2\vee\bbS^3$ and $(\bbS^1\times\bbS^2)\vee\bbS^1$, respectively. See Ex.~\ref{sec:flag>cup}.} 
\label{fig:torus_sphere_VR_flag_rank}
\end{figure}

We show that the erosion distance $\de$ between persistent cup modules is stable under the homotopy-interleaving distance $\dhi$ between persistent spaces, which is a consequence of Thm.~\ref{thm:stab-per-inv}. In addition, the stability of persistent cup modules improves the stability of the persistent cup-length invariant. Recall from \citep{CSGO16} that a standard persistence module $\Mfunc:(\bbR,\leq )\to \Vect$ is \emph{$\mathrm{q}$-tame} if it satisfies the condition that $\rank(M_t\to M_{t'})<\infty$ whenever $t<t'$. \label{para:q-tame}

\begin{restatable}{theorem}{cupmodulestab} 
\label{thm:stability dE_dB_dGH} 
For persistent spaces $\Xfunc$ and $\Yfunc $ with $\mathrm{q}$-tame persistent (co)homology, we have
\[
d_{\mathrm{E}}(\cupprod(\Xfunc),\cupprod(\Yfunc))\leq \de\left( \rank( \cupmodule(\Xfunc)), \rank( \cupmodule(\Yfunc))\right)
\leq  \dhi \left(\Xfunc,\Yfunc \right).\]
For the Vietoris-Rips filtrations $\VR_\bullet(X)$ and $\VR_\bullet(Y)$ of two metric spaces $X$ and $Y$, all the above quantities are bounded above by $2\cdot\dgh \left(X,Y \right).$ 
\end{restatable}

For a fixed $\ell$, we call the functor \[\cupmodule^\ell(\Xfunc):(\bbR,\leq) \to \GVect^{\mathrm{op}}\text{ with }t\mapsto (\bfH^{ + }\left( X_t)\right)^{\ell}\] 
the \textbf{persistent $\ell$-cup module} of $\Xfunc$.
For any $p,\ell\geq 1$, 
we let $\barc\!\left(\deg_p\left(\cupmodule^\ell(\Xfunc)\right)\right)$ be the barcode of the degree-$p$ component of $\cupmodule^\ell(\Xfunc)$. We also show that the bottleneck distance $\db$ between  $\barc\!\left(\deg_p\left(\cupmodule^\ell(\cdot)\right)\right)$ is stable under $\dhi$ between persistent spaces: 

\begin{restatable}{proposition}{ellpbarcode} 
\label{prop:stab-l-cup module}
For persistent spaces $\Xfunc,\Yfunc $, we have
\[
 \max_{\ell,p}\db\left( \barc\!\left(\deg_p\left(\cupmodule^\ell(\Xfunc)\right)\right),\barc\!\left(\deg_p\left(\cupmodule^\ell(\Yfunc)\right)\right)\right)
 \leq \dhi \left(\Xfunc,\Yfunc \right).\]
 For the Vietoris-Rips filtrations $\VR_\bullet(X)$ and $\VR_\bullet(Y)$ of two metric spaces $X$ and $Y$, all the above quantities are bounded above by $2\cdot\dgh \left(X,Y \right).$ 
\end{restatable}

In Ex.~\ref{sec:flag>cup}, we use the spaces $\bbT^2\vee\bbS^3$ and $(\bbS^1\times\bbS^2)\vee\bbS^1$ to demonstrate that the rank invariant of persistent cup modules and the barcode of persistent $\ell$-cup modules provide better lower bounds for the Gromov-Hausdorff distance than those given by persistent cup-length (or LS-category) invariants and persistent homology. This shows the ability of persistent cup modules to distinguish between spaces and capture additional important topological information.\\

In Prop.~\ref{prop:dgm_p,l}, we prove that the persistent cup-length invariant can be obtained from the persistence diagrams of all $\cupmodule^\ell(\Xfunc)$. This is another piece of evidence that the persistent cup module is a richer structure than the persistent cup-length invariant.

\medskip\noindent\textbf{Organization of the paper.} 
In Sec.~\ref{sec:pers-theo}, we provide an overview of persistent theory and discuss the general concept of persistent objects. 
In Sec.~\ref{sec:cat-inv}, we define \textit{categorical invariants}, and see that every categorical invariant gives rise to a persistent invariant. 
In Sec.~\ref{sec:Persistent-cup-length}, we recall our previous work on the \textit{persistent cup-length invariant} of a topological filtration, including the graded ring structure of cohomology which is yielded by the \textit{cup product}, the notion of \textit{persistent cup-length diagram}, the idea of our proposed algorithm, as well as additional details and examples. In Sec.~\ref{sec:ls-cat}, we introduce the \textit{persistent LS-category invariant} and show that it is pointwisely bounded below by the persistent cup-length invariant. In Sec.~\ref{sec:mobuis}, we study the M\"obius inversion of persistent invariants. 
We show that for persistent cup-length invariant and persistent LS-category, their M\"obius inversion can return negative values.
In Sec.~\ref{sec:all-stability}, we establish the stability of persistent invariants, and prove Thm.~\ref{thm:stab-per-inv}, Thm.~\ref{thm:main-stability}, Cor.~\ref{cor:stab-cup}, and Cor.~\ref{cor:stab-cat}. 
In Sec.~\ref{sec:cup module}, we study the $\ell$-fold products of persistent cohomology algebras both as a persistent graded flag (see Sec.~\ref{sec:flag cup module}) and a 2D persistence module (see Sec.~\ref{sec:2d cup module}). In the former case, we identify a complete invariant for flags, and lift it to a persistent invariant which is stable and improves the stability of the persistent cup-length invariant. 

\renewcommand{\nomname}{List of Symbols}
\renewcommand{\nompreamble}{Throughout the paper we fix an arbitrary field $\field$ for the coefficients of (co)homology.}

\nomenclature[1]{\(\N,\N^+,\bbZ\)}{Set of non-negative integers, positive integers and integers, respectively}
\nomenclature[1]{\(\bbR\)}{Set of real numbers}

\nomenclature[2]{\( X, Y, Z,\dots \)}{Topological spaces}
\nomenclature[2]{\(R,S,\dots \)}{Graded rings}
\nomenclature[2]{\(A,B,\dots \)}{Graded algebras}

\nomenclature[3]{\(\catC\)}{A category}
\nomenclature[3]{\(\catC^{\mathrm{op}} \)}{Opposite category of $\catC$}
\nomenclature[3]{\(\ob(\catC)\)}{Objects of a Category $\catC$}
\nomenclature[3]{\(\mor(\catC)\)}{Morphisms of a Category $\catC$}

\nomenclature[4,1]{\(\catN\)}{A poset category}
\nomenclature[4,2]{$(\Int,\subseteq)$}{Poset of inclusions of  closed intervals of the real line.}

\nomenclature[4,8]{\(\Flag\)}{Category of finite-depth flags; see pg.~\pageref{def:flag}}
\nomenclature[4,9]{\(\grFlag\)}{Category of finite-depth graded flags; see pg.~\pageref{para:graded flag}}
\nomenclature[4,8]{\(\Flag_{\operatorname{fin}}\)}{Category of finite-depth flags of finite-dimensional vector spaces; see pg.~\pageref{def:flag}}
\nomenclature[4,9]{\(\grFlag_{\operatorname{fin}}\)}{Category of finite-depth flags of finite-dimensional graded vector spaces; see pg.~\pageref{para:graded flag}}

\nomenclature[4,3]{\(\Top\)}{Category of compactly generated weak Hausdorff topological spaces}
\nomenclature[4,6]{\(\Vect\)}{Category of vector spaces over $\field$}
\nomenclature[4,4]{\(\Ring\)}{Category of graded rings}
\nomenclature[4,5]{\(\Alg\)}{Category of graded algebras}
\nomenclature[4,7]{\(\GVect\)}{Category of graded vector spaces over $\field$}

\nomenclature[5,1]{\(\Ffunc,\Gfunc,\dots \)}{Persistent objects in a given category $\catC$, i.e. functors from a poset category to $\catC$} 
\nomenclature[5,3]{\(\Xfunc, \Yfunc,\dots\)}{Persistent spaces (e.g. a filtration of a given space)}
\nomenclature[5,4]{\(\Mfunc\)}{Standard persistent module, i.e. a persistent vector space}

\nomenclature[6,2]{\(\bfH^*( X)\)}{Graded cohomology ring $\bigoplus_{p\in\N}\bfH^p( X) $ of a space $ X$}
\nomenclature[6,2]{\(\bfH^+(X)\)}{Positive-degree cohomology $\bigoplus_{p\geq 1}\bfH^p(X)$ of a space $ X$}
\nomenclature[6,1]{\(\bfH^p( X) \)}{Degree-$p$ cohomology of a space $ X$}

\nomenclature[6,3]{\(\cupmodule(\Xfunc)\)}{Persistent cup module of $\Xfunc$; see Defn.~\ref{def:cup module}}

\nomenclature[6,4]{\(\cupmodule^\ell(\Xfunc)\)}{Persistent $\ell$-cup module of $\Xfunc$; see Defn.~\ref{def:l-cup module}}

\nomenclature[6,5]{\(\deg_p(\cdot)\)}{Degree-$p$ component of a graded vector space or persistent graded vector space, e.g. $\deg_p(\bfH^*( X))=\bfH^p( X)$}

\nomenclature[6,6]{\(\barc(\Mfunc)\)}{Barcode associated to a standard persistence module $\Mfunc:(\bbR,\leq)\to \Vect$. In this paper, we consider $\Mfunc=\bfH^p(\Xfunc),\bfH^*(\Xfunc),\bfH^+(\Xfunc),\deg_p\left(\cupmodule^\ell(\Xfunc)\right)$ or $\cupmodule^\ell(\Xfunc)$}

\nomenclature[7,1]{\(\invariant\)}{A categorical invariant, i.e. a map $\invariant:\ob(\catC)\sqcup\mor(\catC)\to\catN$ satisfying Defn.~\ref{def:categorical invariant}}
\nomenclature[7,2]{\(\dgm(\invariant(\Ffunc))\)}{Persistence $\invariant$-diagram of a persistent object $\Ffunc$, i.e. the M\"obius inversion of $\invariant(\Ffunc)$; see Defn.~\ref{def:I-dgm}}

\nomenclature[7,3]{\(\rank(\Mfunc)\)}{Rank invariant of a standard persistence module $\Mfunc$; see Ex.~\ref{ex:pers-rank}}
\nomenclature[7,4]{\(\dgm(\Mfunc)\)}{Persistence diagram of a standard persistence module $\Mfunc$, equivalently, $\dgm(\Mfunc)=\dgm(\rank(\Mfunc))$ the M\"obius inversion of $\rank(\Mfunc)$}
\nomenclature[7,5]{\(\len(R)\)}{Length of a graded ring $R$; see Defn.~\ref{def:cup-length}}
\nomenclature[7,6]{\(\cupprod( X)\)}{Cup-length $\len(\bfH^*( X))$ of a space $ X$; see Defn.~\ref{def:cup-length}}
\nomenclature[7,7]{\(\cupprod(\Xfunc) \)}{Persistent cup-length invariant of a persistent space $\Xfunc$, see Defn.~\ref{def:cup_l_func}}
\nomenclature[7,8]{\(\cat( X)\)}{LS-Category of $ X$; see Defn.~\ref{def:cat(X)}}
\nomenclature[7,9]{\(\cat(\Xfunc) \)}{Persistent LS-category of a persistent space $\Xfunc$, see Defn.~\ref{def:pers-ls}}

\nomenclature[8,1]{\(V_\star\)}{Flag of vector spaces $ V_1 \supseteq  V_2 \supseteq  \cdots$, where $\star$ represents the depth; see Defn.~\ref{def:flag}}
\nomenclature[8,2]{\(V^{\circ}_\star\)}{Flag of graded vector spaces $ \bigoplus_{p\geq 1} V^{p}_1 \supseteq  \bigoplus_{p\geq 1} V^{p}_2 \supseteq  \cdots$, where $\circ$ represents the degree; see pg.~\pageref{para:graded flag}}
\nomenclature[8,3]{\(V_{\star,\bullet},\,V^{\circ}_{\star,\bullet}\)}{Persistent flag and persistent graded flag, resp. Here $\bullet$ represents the filtration parameter}
\nomenclature[8,4]{\(\fdim(V_\star)\)}{Dimension of a flag $V_\star$; see Defn.~\ref{def:dim}}
\nomenclature[8,5]{\(\rank(V_{\star,\bullet})\)}{Rank invariant of a persistent flag; see Defn.~\ref{def:rank}}

\nomenclature[9,2]{\(d_{\mathrm{HI}}\)}{Homotopy interleaving distance; see Defn.~\ref{def:dhi}}
\nomenclature[9,1]{\(\di\)}{Interleaving distance; see Defn.~\ref{def:interleaving-dist}. When needed, we will write $\di^{\catC}$ to indicate the underlying category $\catC$}
\nomenclature[9,3]{\(d_{\mathrm{E}}\)}{Erosion distance; see Defn.~\ref{def:de}}
\nomenclature[9,4]{\(d_{\mathrm{GH}}\)}{Gromov-Hausdorff distance, cf. \citep{gromov2007metric}}
\nomenclature[9,5]{\(\db\)}{Bottleneck distance, cf. \citep{cohen2007stability}}

\printnomenclature

\section{Persistent invariants}
\label{sec:pers-inv}
In this section, we define the notions of \emph{invariants} and \emph{persistent invariants} in a general setting. 

In classical topology, an invariant is a numerical quantity associated to a given topological space that remains invariant under a homeomorphism. In linear algebra, an invariant is a numerical quantity that remains invariant under a linear isomorphism of vector spaces. Extending these notions to the general `persistence' setting from TDA, leads to the study of persistent invariants, which are designed to
extract and quantify important information about TDA structures, such as 
the \textit{rank invariant} for persistent vector spaces 
\citep[Defn.~11]{carlsson2007theory}.
We study two other persistent invariants: the \textit{persistent cup-length invariant} of persistent spaces \citep[Defn.~7]{contessoto_et_al:LIPIcs.SoCG.2022.31} (see also Sec.~\ref{sec:Persistent-cup-length}) and the \textit{persistent LS-category invariant} of persistent spaces (see Sec.~\ref{sec:ls-cat}) that we introduce.

\subsection{Persistence theory}\label{sec:pers-theo}
We recall the notions of \emph{persistent objects} and their \emph{morphisms} from \citep[Defn.~2.2]{patel2018generalized}.
For general definitions and results in category theory, we refer to \citep{awodey2010category,mac2013categories,leinster2014basic}. 

\begin{definition}\label{defi:pers-object}
Let $\catC$ be a category.
We call any functor $\Ffunc:(\bbR,\leq)\to\catC$ a \textbf{persistent object (in $\catC$)}.
Specifically, a persistent object $\Ffunc:(\bbR,\leq)\to\catC$ consists of 
\begin{itemize}
    \item for each $t\in\bbR$, an object $ F_t$ of $\catC$,
    \item for each inequality $t\leq s$ in $\bbR$, a morphism $f_t^s: F_t\to F_s$, such that 
    \begin{itemize}
        \item $f_t^t=\id _{ F_t}$
        \item $f_s^r\circ f_t^s=f_t^r$, for all $t\leq s\leq r$.
    \end{itemize}
\end{itemize}
\end{definition}

\begin{definition}
Let $\Ffunc,\Gfunc:(\bbR,\leq)\to\catC$ be two persistent objects in $\catC$.
A \textbf{natural transformation from $\Ffunc$ to $\Gfunc$}, denoted by $\varphi:\Ffunc\Rightarrow \Gfunc$, consists of an $\bbR$-indexed family $(\varphi_t: F_t\to G_t)_{t\in\bbR}$ of morphisms in $\catC$, such that the diagram
\[
\xymatrix
{
 F_t\ar[rr]^{\varphi_t}\ar[d]_{f_t^s}  &  &
 G_t \ar[d]^{g_t^s} \\
	   F_s\ar[rr]^{\varphi_s}& &  G_s\\
}
\]
commutes for all $t\leq s$.
\end{definition}

\begin{example}
\begin{itemize}
    \item Let $Z$ be a finite metric space and let $\VR_t(Z)$ denote the \textit{Vietoris-Rips complex of $Z$} at the scale parameter $t$, which is the simplicial complex defined as $\VR_t(Z):=\{\alpha\subseteq Z\mid\diam(\alpha)\leq t\}.$ Let us denote   
 \[X_t :=  
      \begin{cases}
      \lvert\VR_t(Z)\rvert , & \text { if $t\geq 0$} \\
      \lvert\VR_0(Z)\rvert , &  \text{ otherwise.}
     \end{cases} \]
For each inequality $t\leq s$ in $\bbR$, we have the inclusion $\iota_t^s : X_t\hookrightarrow X_s$ giving rise to a persistent space $\Xfunc:(\bbR,\leq)\to \Top$.  

\item Applying the $p$-th homology functor to a persistent topological space $\Xfunc$, for each $t\in\bbR$ we obtain the vector space $\bfH_p( X_t)$ and for each pair of parameters $t\leq s$ in $\bbR$, we have the linear map in (co)homology induced by the inclusion $ X_t\hookrightarrow X_s$. This is another example of a persistent object, namely a \textit{persistent vector space} $\bfH_p(\Xfunc):(\bbR,\leq)\to \Vect$. Dually, by applying the $p$-th cohomology functor, we obtain a persistent vector space $\bfH^p(\Xfunc):(\bbR,\leq)\to \Vect$ which is a contravariant functor.
\end{itemize}
\end{example}

In the literature, different types of invariants have been identified to study properties of persistent objects based on the category they lie in. For example:
\begin{example}
\begin{itemize}
    \item  For the category of finite sets, $\Set$, whose morphisms are functions between finite sets, we consider $\invariant :\ob(\Set)\to\N$ to be the \textbf{cardinality invariant}.
    \item For the category of vector spaces over $\field$, $\Vect$, whose morphisms are linear maps, we consider $\invariant :\ob(\Vect)\to\N$ to be the \textbf{dimension invariant}. 
    \item  For the category of topological spaces, $\Top$, whose morphisms are continuous maps, we consider $\invariant :\ob(\Top)\to\N$ to be the invariant that counts the \textbf{number of connected components}.
   \item  For the category of smooth manifolds, $\mathbcal{Man}$, whose morphisms are 
   smooth maps, we consider $J:\ob(\mathbcal{Man})\to\N$ to be the \textbf{genus invariant}. 
\end{itemize}
\end{example}

\noindent\textbf{Persistence modules, barcodes and persistence diagrams.} A persistent object in $\Vect$ is also called a \textbf{(standard) persistence module}. An \emph{interval module} associated to an interval $[a,b]$ is the persistence module, denoted by $K[a,b]$ such that 
\[K[a,b](t)=\begin{cases}
    K, &\mbox{$t\in [a,b]$}\\
    0, &\mbox{$t\notin [a,b]$}
\end{cases}\text{   and    }K[a,b](t\leq s)=\begin{cases}
    \id_{K}, &\mbox{$[t,s]\subseteq [a,b]$}\\
    0, &\mbox{otherwise.}
\end{cases}
\]

When $\Mfunc$ can be decomposed as a direct sum of interval modules (e.g. when $\Mfunc$ is $\mathrm{q}$-tame \citep[Defn.~1.12]
{oudot2015persistence}), 
say $\Mfunc\cong \bigoplus_{l\in L}K[a_l,b_l]$, 
the \textbf{barcode} of $\Mfunc$ is defined as the multiset 
$$\barc(\Mfunc):=\{ [a_l,b_l]:l\in L\},$$
where the elements $[a_l,b_l]$ are called \textit{bars}.
The \textbf{persistence diagram} of $\Mfunc$ is the map $\dgm(\Mfunc):\Int\to \N$ such that $\dgm(\Mfunc)([a,b])$ is the multiplicity of $[a,b]$ in $\barc(\Mfunc)$. It is clear that $\barc(\Mfunc)$ and $\dgm(\Mfunc)$ determine each other. Later in Ex.~\ref{ex:rk and dgm}, we recall that the persistence diagram is the M\"obius inversion of the rank invariant.
\medskip

In the following subsection, we study more general persistent objects and identify a general condition on invariants so that they can be used to study these persistent objects.

\subsection{Persistent \texorpdfstring{$\catN$}{caN}-valued categorical invariants}
\label{sec:cat-inv}

We introduce the notion of $\catN$-valued \emph{categorical invariants}, where $\catN$ is a poset category with a partial order $\geq$ (e.g., $\catN=\N$ or $\N^\infty$), and devise a method for lifting such invariants to persistent invariants.

\begin{definition}
\label{def:categorical invariant}
Let $\catC$ be any category and let $\mor(\catC)$ denote the collection of all morphisms of $\catC$.
A $\catN$-valued invariant $\invariant :\ob(\catC)\to\catN$ of $\catC$ is said to be a \textbf{$\catN$-valued categorical invariant of $\catC$}, and denoted by $\invariant :\ob(\catC)\sqcup\mor(\catC)\to\catN$, if $\invariant $ extends to 
\begin{align*}
    \invariant :\mor(\catC)&\to\catN\\
    f &\mapsto \invariant (f),
\end{align*}
a map on the class $\mor(\catC)$ of morphisms in $\catC$ such that 
\begin{itemize}
    \item [(i)] $\invariant (\id _X)=\invariant (X)$,~for all $X\in \ob(\catC)$, and
    \item [(ii)] for any commutative diagram of the following form:
     \[
         \begin{tikzcd}[column sep = 4em]
		Y \ar[r,"g"]
        &
        W
        \ar[d,"h" right]
        \\
		 X \ar[r,"h\circ g\circ f" below]
		 \ar[u, "f" left]
        &
        Z,
		\end{tikzcd}
     \]
	we have \[\invariant(h\circ g\circ f)\leq \invariant(g).\]
\end{itemize}
\end{definition}

\begin{remark}\label{rmk:invariant preserves iso}
A categorical invariant preserves isomorphisms in the underlying category. This follows immediately from Condition (ii) of Defn.~\ref{def:categorical invariant}: for any isomorphism $f:X\to Y$ in a given category $\catC$, 
\[\invariant(f)=\invariant(f\circ f^{-1}\circ f)\leq \invariant(f^{-1})\]
and similarly $\invariant(f^{-1})\leq \invariant(f).$
\end{remark}

Condition (ii) of Defn.~\ref{def:categorical invariant} also implies that for a persistent object $\Ffunc:(\mathbb{R},\leq)\to \catC$, we have 
$$[a,b]\subseteq [c,d]\Rightarrow \invariant \left(f_a^b\right)\geq \invariant \left(f_c^d\right).$$
Thus, we can associate a functor $(\Int,\subseteq)\to(\catN,\leq)^{\mathrm{op}}$ to each persistent object in $\catC$ as follows:

\begin{definition}\label{def:p_invariant}
Let $\catC$ be a category and let $\invariant $ be a $\catN$-valued categorical invariant. 
For any given persistent object $\Ffunc:(\bbR,\leq)\to\catC$, we associate the functor 
\begin{align*}
    \invariant (\Ffunc):(\Int,\subseteq)&\to(\catN,\leq)^{\mathrm{op}}\\
    [a,b]&\mapsto \invariant \left(f_a^b\right).
\end{align*}

We call $\invariant (\Ffunc)$ the \textbf{persistence $\invariant $-invariant associated to $\Ffunc$}.
\end{definition}

We establish an equivalent definition of Defn.~\ref{def:categorical invariant} (2), which is easier to use when checking whether an invariant is a categorical invariant.

\begin{proposition} [Equivalent definition of categorical invariant]
\label{prop:eqvi def for cat inv}
A $\catN$-valued invariant $\invariant $ is a categorical invariant, if and only if
\begin{itemize}
    \item [(i)] $\invariant (\id _X)=\invariant (X)$,~for all $X\in \ob(\catC)$, and
    \item [(ii')] for any $f:X\to Y$ and $g:Y\to Z$ in $\catC$, $\invariant (g\circ f)\leq \min\{ \invariant (f), \invariant (g)\}$.
\end{itemize}
\end{proposition}

\begin{proof} We first prove that Condition (ii) implies Condition (ii'). By Condition (ii), for any $f:X\to Y$ and $g:Y\to Z$, \[\invariant( g\circ f)=\invariant(\id_Z\circ g\circ f) \leq \invariant(g).\]
Similarly, we have $\invariant( g\circ f)=\invariant( g\circ f\circ\id_X) \leq \invariant(f).$

Conversely, for any $f:X\to Y, g:Y\to W$ and $h:W\to Z$, it follows from Condition (ii') that
\[\invariant( h\circ g\circ f)\leq \invariant(g\circ f)\leq \invariant(g).\]
\end{proof}

By its definition, a categorical invariant needs to assign values to both the objects and the morphisms in a category. Below, we consider one type of invariants that are originally defined only on objects but can be easily extended to a categorical invariant by sending each morphism to the invariant evaluated on its image.

\begin{example}[epi-mono invariant]\label{ex:surjective-injective-inv}
Let $\catC$ be any \emph{regular category} (e.g.~the category of rings or the category of vector spaces).
An $\catN$-valued \textbf{epi-mono invariant} in $\catC$ is any map $\invariant :\ob(\catC)\to\catN$ such that: 
\begin{itemize}
    \item if there is a regular epimorphism $ X\twoheadrightarrow Y$, then $\invariant (X)\geq \invariant (Y)$;    
    \item if there is a monomorphism $X\hookrightarrow Y$, then $\invariant (X)\leq \invariant (Y)$.

\end{itemize}
In a regular category $\catC$, the regular epimorphisms and monomorphisms form a factorization system, and thus $\catC$ is a category with images in particular. Hence, any epi-mono invariant $\invariant :\ob(\catC)\to\catN$ of a regular category $\catC$, yields a categorical invariant $\invariant :\ob(\catC)\sqcup\mor(\catC)\to\catN$, given by $\invariant (f):=\invariant (\image(f))$. Indeed, because $\image(g\circ f)\hookrightarrow \image(g)$ is a  monomorphism, we have
$\invariant (\image(g\circ f))\leq \invariant (\image(g))$; because $\image(f)\twoheadrightarrow \image(g\circ f)$ is a regular epimorphism, we have
$\invariant (\image(g\circ f))\leq \invariant (\image(f))$. 
\end{example}

\begin{example}[Rank invariant, {\citep[Defn.~11]{carlsson2007theory}}]
\label{ex:pers-rank}
Recall that $\Vect$ is the category of vector spaces over the field $\field$ whose morphisms are $\field$-linear maps. The dimension invariant $\dimension:\ob(\Vect)\to\N$, that assigns to each vector space its dimension, is an example of an $\N$-valued epi-mono invariant. According to Ex.~\ref{ex:surjective-injective-inv}, for any $\Ffunc:(\bbR,\leq)\to \Vect$, $\dimension$ gives rise to a persistent invariant such that $\dimension(\Ffunc):[a,b]\mapsto\dimension(\image(f_a^b))=\rank(f_a^b)$, which coincides with the well-known \textbf{rank invariant} 
\citep[Defn.~11]{carlsson2007theory}.
\end{example}

\subsubsection{Comparison to related notions of invariants}
\label{sec:comparison-long}

The notion of categorical invariant in Defn.~\ref{def:categorical invariant} can be seen as a generalization of both that of an epi-mono invariant as in Ex.~\ref{ex:surjective-injective-inv} and of related notions that have been considered  in the TDA literature. Below we provide the details. 
\begin{itemize}
\item For $\catC$ an abelian category (and thus a regular category in particular) then (a) the notion of \emph{epi-mono-respecting pre-orders on $\catC$} introduced by Puuska \citep[Defn.~3.2]{puuska2017erosion} is equivalent to (b) 
    the restriction of our notion of epi-mono invariant to abelian categories (and thus a special case of a categorical invariant), as follows. 
    \smallskip
    
    (b)$\Rightarrow$(a): Given any epi-mono invariant $\invariant:\ob(\catC)\to\catN$ on \textit{an abelian category} $\catC$ where $\catN$ is a poset, we can define a pre-order $\leq_\invariant$ on $\ob(\catC)$, induced by the invariant, given by: $X\leq_\invariant Y \Leftrightarrow \invariant(X)\leq \invariant(Y)$. 
    By the definition of an epi-mono invariant, $\invariant$ is non-decreasing on monomorphisms, and non-increasing on regular epimorphisms. 
    This implies that the pre-order $\leq_\invariant$ is epi-mono-respecting in the sense of Puuska. 
    \smallskip
    
    (a)$\Rightarrow$(b): Suppose that we have an epi-mono-respecting
        pre-order $\leq$ on $\ob(\catC)$ in the sense of Puuska. Then the ``1-skeleton" of that pre-order (viewed as a category whose objects are the equivalence classes associated with the equivalence $x\simeq y\Leftrightarrow \big(x\leq y$ and $y\leq x\big)$) will be a poset which we denote by $\catN:=(\ob(\catC)/\!\!\simeq,\leq)$. Then,  we obtain 
 the persistent invariant $\invariant^{\leq}:\ob(\catC)\to\catN$, $X\mapsto [X]_{\simeq}$.
   One can check that these two constructions ($\invariant\mapsto\leq_\invariant$ and $\leq\mapsto \invariant^{\leq}$) are inverses of each other, i.e. they induce a bijection.
   \medskip
   
     \item For $\catC$ any category, the notion of \emph{categorical persistence function} of Bergomi et al.~\citep[Defn.~3.2]{bergomi2019rank} is a lower bounded function $p:\mor(C)\to\N$ such that for any $X\xrightarrow{f}Y\xrightarrow{g}Z\xrightarrow{h}W$ in $\catC$: (i) $p(g\circ f)\leq p(g)$ and $p(h\circ g\circ f)\leq p(h\circ g)$, and (ii) $p(g)-p(g\circ f)\geq p(h\circ g)-p(h\circ g\circ f)$. The first condition is equivalent to our notion of a categorical invariant 
     (we consider that the categorical persistence function is defined on each object $X$ in $\catC$ as $p(X):=p(\mathbf{id}_X)$). The second condition is actually equivalent to the positivity of the persistence diagram (yielded as the M\"obius inversion as in Defn.~\ref{def:I-dgm}) of the categorical persistence function. However, our notion of categorical invariant in Defn.~\ref{def:categorical invariant} does not assume such positivity conditions, e.g.~both the persistent cup-length and persistent LS-category invariants sometimes can have negative persistence diagrams (see Ex.~\ref{ex:non-example}). 
     \medskip
     
   \item  
    For $\catC$ a regular category, the categorical invariant induced by an epi-mono invariant by definition (see Ex.~\ref{ex:surjective-injective-inv}) invokes  images of morphisms whereas general categorical invariants do not have to, making it a special case of a categorical invariant.
   For example: the LS-category invariant $\cat$ of a map is in general not equal to $\cat$ of the image of the map and $\cat$ is not epi-mono; see Rmk.~\ref{rmk:cat not epi-mono}.
    This illustrates that the notion of a categorical invariant is different and does not follow from the work of Puuska (i.e.~epi-mono-respecting pre-orders on abelian categories). 
   \smallskip
   
    For $\catC$ a regular category, the notion of \emph{rank function} of Bergomi et al.~\citep[Defn.~2.1]{bergomi2019rank} is an epi-mono invariant (as in  Ex.~\ref{ex:surjective-injective-inv}) that satisfies a positivity condition for the persistence diagram induced by the rank function (as in the case of the categorical persistence function). 
   \smallskip
   
    For $\catC$ an abelian category (and thus a regular category in particular) then the notion of an \emph{amplitude on $\catC$} introduced by Giunti et al.~\citep[Defn.~3.1]{giunti2021amplitudes} coincides with  an epi-mono invariant $\alpha:\ob(\catC)\to[0,\infty)$ satisfying the additional conditions that $\alpha(0)=0$ and that for any short exact sequence $0\to A\to B\to C\to 0$, $\alpha(B)\leq\alpha(A)+\alpha(C)$. 

\end{itemize}  
To summarize, our notion of a categorical invariant is a strict generalization of several concepts introduced in the TDA literature. 
In particular, the persistent LS-category invariant cannot be realized as an invariant of the above types.\\

In the remaining part of Sec.~\ref{sec:pers-inv}, we will concentrate on two other $\N$-valued categorical invariants and will omit the term `$\N$-valued' for conciseness.

In Sec.~\ref{sec:Persistent-cup-length}, we consider the \emph{cup-length}, a categorical invariant of topological spaces, which arises from the cohomology ring structure. Recall that the cohomology functor is contravariant. In general, a contravariant functor from $\catC$ to $\catN$ is equivalent to a covariant functor from the opposite category $\catC^{\mathrm{op}}$ of $\catC$ to $\catN$. 
It is clear that any categorical invariant $\invariant :\ob(\catC)\to\N$ of $\catC$ is also a categorical invariant in the opposite category $\catC^{\mathrm{op}}$ of $\catC$.

Later in Sec.~\ref{sec:ls-cat}, we study the persistent invariant arising from the LS-category of topological spaces, which admits the persistent cup-length invariant as a pointwise lower bound.

\subsection{Persistent cup-length invariant}
\label{sec:Persistent-cup-length}

In the standard setting of persistent homology, one considers a \textit{filtration} of spaces, i.e.~a collection of spaces $\Xfunc=\{ X_t\}_{t\in\bbR}$ such that $ X_t\subseteq  X_s$ for all $t\leq s$, and studies its \textit{$p$-th persistent homology} for any given dimension $p$. Persistent homology is defined as the functor $\bfH_p(\Xfunc):(\bbR,\leq)\to\Vect$ which sends each $t$ to the $p$-th homology $\bfH_p( X_t)$ of $ X_t$; see \citep{edelsbrunner2008Persistent,carlsson2009topology}. 
The barcode of the $p$-th persistent homology $\bfH_p(\Xfunc)$, also called 
\textit{the $p$-th barcode of $\Xfunc$}, encodes the lifespans of the degree-$p$ holes ($p$-cycles that are not $p$-boundaries) in $\Xfunc$. 
The \emph{$p$-th persistent cohomology $\bfH^p(\Xfunc)$} and its corresponding barcode $\barc\!\left(\bfH^p(\Xfunc)\right)$ are defined dually. Although persistent homology and persistent cohomology have the same barcode \citep[Prop.~2.3]{de2011dualities}, this paper mostly concerns cohomology so we will use the latter notion.
We call the barcode $\barc\!\left(\bfH^*(\Xfunc)\right)$ of $\bfH^*(\Xfunc)$, which is
the disjoint union $\sqcup_{p\in \N}\barc\!\left(\bfH^p(\Xfunc)\right)$, \emph{the total barcode of $\Xfunc$}. 

In Sec.~\ref{sec:cup and ring} we recall the notion of the cup product of cocycles, together with the notion and properties of the cup-length invariant of cohomology rings.  In Sec.~\ref{sec:cup_l_func} we lift the cup-length invariant to a persistent invariant, called the persistent cup-length invariant, and  examine some examples that highlight its strength.

Persistent cup-length invariant sometimes captures more  information than persistent (co)homology, cf. \citep[Ex.~13]{contessoto_et_al:LIPIcs.SoCG.2022.31}.  However, cup-length is not a complete invariant of graded rings. For instance, the spaces $\bbT^2 \vee \bbT^2$ and $\bbS^1 \vee \bbS^2 \vee \bbS^1 \vee \bbT^2$ have different ring structures, but have the same cup-length. 
For the purpose of extracting even more information from the cohomology ring structure, in Sec.~\ref{sec:cup module} we will study the (persistent) $\ell$-fold product of the persistent cohomology algebra, which provide a strengthening of the notion of cup-length.

\subsubsection{Cohomology ring and cup-length} 
\label{sec:cup and ring}

We recall the cup product operation in the setting of simplicial cohomology. 
Let $ X$ be a simplicial complex with an ordered vertex set $\{x_1<\dots<x_n\}$. 
For any non-negative integer $p$, we denote a $p$-simplex by $\alpha:=[\alpha_0,\dots,\alpha_{p}]$ where $\alpha_0<\dots<\alpha_{p}$ are ordered vertices in $ X$, and by $\alpha^*:C_p(  X)\to \field$, the dual of $\alpha$. Here $\field$ is the base field as before.
Let $\beta:=[\beta_0,\dots,\beta_{q}]$ be a $q$-simplex for some non-negative integer $q$. The \emph{cup product} $\alpha^*\smile \beta^*$ is defined as the linear map $C_{p+q}(  X)\to  \field$ such that for any $ (p+q)$-simplex $\tau=[\tau_{0},\dots,\tau_{p+q}]$, 
\[\alpha^{*}\smile\beta^{*}(\tau):=\alpha^*([\tau_{0},\dots,\tau_{p}])\cdot \beta^*([\tau_{p},\dots,\tau_{p+q}]).\] 
Equivalently, 
we have that 
$\alpha^{*}\smile\beta^{*}$ is $[\alpha_0,\dots,\alpha_{p},\beta_1,\dots,\beta_{q}]^*$
if $\alpha_{p}=\beta_0$, and $0$ otherwise.
By a \emph{$p$-cochain} we mean a finite linear sum $\sigma=\sum_{j=1}^h \lambda_j \alpha^{j*}$, where each $\alpha^j$ is a $p$-simplex in $ X$ and $\lambda_j\in \field$. The \emph{cup product} of a $p$-cochain $\sigma=\sum_{j=1}^h \lambda_j \alpha^{j*}$ and a $q$-cochain $\sigma'=\sum_{j'=1}^{h'}\mu_{j'}\beta^{j'*}$ is defined as $$\sigma\smile\sigma':=\sum_{j,j'}\lambda_j\mu_{j'}\left(\alpha^{j*}\smile \beta^{j'*}\right).$$

For a given space $ X$, the cup product yields a bilinear map $\smile:\bfH^p( X) \times\bfH^q( X)\to\bfH^{p+q}( X)$ of vector spaces.
In particular, it turns the total cohomology vector space $\bfH^*( X):=\bigoplus_{p\in\N}\bfH^p( X) $ into a graded ring $(\bfH^*( X),+,\smile)$.
The \textit{cohomology ring map} $ X\mapsto \bfH^*( X)$ defines a contravariant functor from the category of spaces, $\Top$, to the category of graded rings, $\Ring$
(see \citep[Sec.~3.2]{hatcher2000}).

\begin{definition}
A ring $(R,+,\bullet)$ is called a \textbf{graded ring} if there exists a family of subgroups $\lbrace R_p\rbrace _{p\in  \N}$ of $R$ such that $R=\bigoplus_{p\in  \N} R_p$ (as abelian groups), and $R_a\bullet R_b\subseteq R_{a+b}$ for all $a,b\in \N$.
Let $R$ and $S$ be two graded rings. 
A ring homomorphism $\varphi:R\to  S$ is called a \textbf{graded homomorphism} if it preserves the grading, i.e.~$\varphi(R_p)\subseteq S_p$, for all $p\in \N$. 
\end{definition}

To avoid the difficulty of describing and comparing ring structures (in a computer), we study a computable invariant of the graded cohomology ring, called the \textit{cup-length}. 

\begin{definition}\label{def:cup-length}
The \textbf{length} of a graded ring $R$ is the largest non-negative integer $\ell$ such that there exist homogeneous elements $\eta_1,\dots,\eta_\ell\in R$ with nonzero degrees (i.e. $\eta_1,\dots,\eta_\ell\in \bigcup_{p\geq 1} R_p$), such that $\eta_1 \bullet \dots \bullet \eta_\ell \neq 0$. 
If $\bigcup_{p\geq 1} R_p=\emptyset$, then we declare that the length of $R$ is zero.
We denote the length of a graded ring $R$ by $\len(R)$.
The map 
\[\len:\ob(\Ring)\to\N,\text{ with }R\mapsto \len(R)\] is called the \textbf{length invariant}.

When $R=(\bfH^*( X),+,\smile)$ for some space $ X$,
we denote $\cupprod( X):=\len(\bfH^*( X))$ and call it the \textbf{cup-length of $ X$}.
The map 
\[\cupprod:\ob(\Top)\to\N, \text{ with }X\mapsto \cupprod( X)\] is called the \textbf{cup-length invariant}.
\end{definition}

Here are some properties of the (cup-)length invariant that we will use.

\begin{proposition} \label{prop:len_of_basis}
Let $R$ be a graded ring. Suppose $B=\bigcup_{p\geq 1} B_p$, where each $B_p$ generates $R_p$ under addition. Then $\len(R)=\sup\left\{\ell\geq 1\mid B^{ \ell}\neq \{0\}\right\}.$ In the case of cohomology ring, let $B_p$ be a linear basis for $\bfH^p( X) $ for each $p\geq 1$, and let $B:=\bigcup_{p\geq 1} B_p$. Then, $\cupprod( X)=\sup\left\{\ell\geq 1\mid B^{ \ell}\neq \{0\}\right\}$. 
\end{proposition}

\begin{proof} 
It follows from the definition that $\len(R)=\sup\big\{\ell\geq 1\mid(\bigcup_{p\geq 1} R_p )^\ell\neq \{0\}\big\}.$ We claim that $\left(\bigcup_{p\geq 1} R_p \right)^\ell\neq \{0\}$ iff $B^{ \ell}\neq \{0\}$. Indeed, whenever $\eta_1\bullet\dots\bullet\eta_\ell\neq 0$, where each $\eta_i\in \bigcup_{p\geq 1} R_p$, we can write every $\eta_i$ as a linear sum of elements in $B$. Thus, $\eta$ can be written as a linear sum of elements in the form of $r_1\bullet\dots\bullet r_\ell$, where each $r_j\in B$. Because $\eta\neq 0$, there must be a summand $r_1\bullet\dots\bullet r_\ell\neq 0$. Therefore, $B^{ \ell}\neq \{0\}$. 
\end{proof}

\subsubsection{Persistent cohomology ring and persistent cup-length invariant}\label{sec:cup_l_func}

We study the persistent cohomology ring of a filtration and 
the associated notion of persistent cup-length invariant.
We examine several examples of this persistent invariant and establish a way to visualize it in the half-plane above the diagonal.

A functor $\Rfunc:(\bbR,\leq)\to\Ring$ is called a \textbf{persistent graded ring}. 
Recall the contravariant cohomology ring functor $\bfH^*:\Top\to\Ring$. 
Given a persistent space $\Xfunc:(\bbR,\leq)\to\Top$, 
the composition $\bfH^*(\Xfunc) :(\bbR,\leq)\to\Ring$ is called the \textbf{persistent cohomology ring of $\Xfunc$}. Due to the contravariance of $\bfH^*$, we consider only contravariant persistent graded rings in this paper.

By \citep[Prop.~38]{contessoto2021persistent}, the length of graded rings is an epi-mono invariant\footnote{The original proof was done for surjective ring homomorphisms and injective ring homomorphisms, but in the category of rings these morphisms can be identified with regular ring epimorphisms and ring monomorphisms, respectively.} and thus a categorical invariant, so for any persistent graded ring $\Rfunc$, $\len(\Rfunc)$ defines a functor from $(\Int,\subseteq)$ to $(\N,\leq)^{\mathrm{op}}$. We lift the length invariant to a persistent invariant as:

\begin{definition}\label{def:cup_l_func}
Given a persistent graded ring $\Rfunc:(\bbR,\leq)\to\Ring^{\mathrm{op}}$
we define the \textbf{persistent length invariant of $\Rfunc$} as the functor
\[\len(\Rfunc):(\Int,\subseteq)\to(\N,\leq)^{\mathrm{op}} \text{ with } \linterval  t,s \rinterval \mapsto\len\left(\image( R_s\to  R_t) \right).\] 

If $\Rfunc=\bfH^*(\Xfunc) $ is the persistent cohomology ring of a given persistent space $\Xfunc:(\bbR,\leq)\to\Top$,
then we will call the functor 
\[\len(\bfH^*(\Xfunc) ):(\Int,\subseteq)\to(\N,\leq)^{\mathrm{op}} \text{ with }\linterval  t,s \rinterval \mapsto\len\left(\image(\bfH^*( X_s)\to \bfH^*( X_t))\right),\] the \textbf{persistent cup-length invariant of $\Xfunc$}, and we will denote it by $\cupprod (\Xfunc):(\Int,\subseteq)\to(\N,\leq)^{\mathrm{op}}$.
\end{definition}

Prop.~\ref{prop:generators-of-ring} below allows us to compute the cohomology images of a persistent cohomology ring from representative cocycles (see \citep[Defn. 3]{contessoto_et_al:LIPIcs.SoCG.2022.31}), which is applied to establish Thm.~1 of \citep{contessoto_et_al:LIPIcs.SoCG.2022.31} and to compute persistent cup-length invariants. 
Prop.~\ref{prop:prod-coprod-pers-cup} allows us to simplify the calculation of persistent cup-length invariants in certain cases, such as the Vietoris-Rips filtration of products or wedge sums of metric spaces, e.g. Ex.~\ref{ex:S1-Tn}. 

\begin{proposition}
\label{prop:generators-of-ring}
Let $\Xfunc=\{ X_t\}_{t\in\bbR}$ be a filtration, together with a family of representative cocycles $\bsigma=\{\sigma_I\}_{I\in \barc\!\left(\bfH^*(\Xfunc)\right)}$ for $\bfH^*(\Xfunc) $. Let $t\leq s$ in $\bbR$. Then
$\image(\bfH^*( X_s) \to \bfH^*( X_t) )=\langle  [\sigma_I]_t: [t,s]\subseteq I\in\barc\!\left(\bfH^*(\Xfunc)\right) \rangle ,$
generated as a graded ring. 
\end{proposition}

\begin{proof}
First, let us recall the following: Given a space $ X$, the cohomology ring $\bfH^*( X)\in \Ring$ is a graded ring generated by the graded cohomology vector space $\bfH^*( X)\in \Vect$, under the operation of cup products. It is clear that any linear basis of $\bfH^*( X)$ also generates the ring $\bfH^*( X)$, under the cup product. Given an inclusion of spaces $ X\xhookrightarrow{\iota} Y$, let $f:\bfH^*( Y) \to \bfH^*( X)$ denotes the induced cohomology ring morphism. Let $A$ be a linear basis for $\bfH^*( Y) $. Since $A$ also generates $\bfH^*( Y) $ as a ring, the image $f(A)$ generates $f(\bfH^*( Y) )$ as a ring. 

Now, let $\bfH^*(\iota_t^s ):\bfH^*( X_s) \to \bfH^*( X_t) $ denote the cohomology map induced by the inclusion $\iota_t^s: X_t\hookrightarrow  X_s$. Notice that the set $A:=\{[\sigma_I]_s: s\in I\in\barc\!\left(\bfH^*(\Xfunc)\right)\}$ forms a linear basis for $\bfH^*( X_s) $, and thus $\bfH^*(\iota_t^s )(A)$ generates $\image(\bfH^*(\iota_t^s ))$ as a ring. On the other hand, for each representative cocycle and any $t\leq s$, $\bfH^*(\iota_t^s )([\sigma_I]_s)=\left[\sigma_I\vert_{C_p( X_t)}\right]\neq 0\iff [t,s]\subseteq I$.
It follows that 
\begin{align*}
    \bfH^*(\iota_t^s )(A)&=\left\{\bfH^*(\iota_t^s  )\left([\sigma_I]_s\right): [t,s]\subseteq I\in\barc\!\left(\bfH^*(\Xfunc)\right)\right\}\\
    &=\left\{[\sigma_I]_t: [t,s]\subseteq I\in\barc\!\left(\bfH^*(\Xfunc)\right)\right\}.
\end{align*}
\end{proof}

\begin{proposition}
\label{prop:prod-coprod-pers-cup}
Let $\Xfunc,\Yfunc:(\bbR,\leq)\to\Top$ be two persistent spaces. Then:
\begin{itemize}
    \item $\cupprod\left(\Xfunc\times\Yfunc\right)=\cupprod(\Xfunc)+\cupprod(\Yfunc),\text{  }$
    \item $\cupprod\left(\Xfunc\amalg\Yfunc\right)=\max\{\cupprod(\Xfunc),\cupprod(\Yfunc)\},\text{ and }$
    \item $\cupprod\left(\Xfunc\vee\Yfunc\right)=\max\{\cupprod(\Xfunc),\cupprod(\Yfunc)\}$.
\end{itemize}
Here $\times,\amalg$ and $\vee$ denote point-wise product, disjoint union, and wedge sum, respectively. For the first item, we additionally require the spaces in $\Xfunc$ and $\Yfunc$ to have torsion-free cohomology rings.
\end{proposition}
\begin{proof}
By functoriality of products , disjoint unions, and wedge sums, we can define the persistent spaces: $\Xfunc\times\Yfunc:=(\{ X_t\times  Y_t\}_{t\in\bbR},\{f_{t}^{s}\times g_{t}^{s}\})$, $\Xfunc\amalg\Yfunc:=(\{ X_t\amalg  Y_t\}_{t\in\bbR},\{f_{t}^{s}\amalg g_{t}^{s}\})$, and $\Xfunc\vee\Yfunc:=(\{ X_t\vee  Y_t\}_{t\in\bbR},\{f_{t}^{s}\vee g_{t}^{s}\})$. Let $[a,b]$ be any interval in $\Int$. Utilizing the contravariance property of the cohomology ring functor $\bfH^*$, we obtain:
\begin{align*}
    \cupprod\left(\Xfunc\times\Yfunc \right)([a,b])&=\len\left(\bfH^*(f_a^b\times g_{a}^{b})\right)\\
    &=\len\left(\bfH^*(f_a^b)\otimes \bfH^*(g_a^b)\right)\\
    &=\len\left(\bfH^*(f_{a}^{b})\right)+\len\left(\bfH^*(g_a^b)\right)\\
    &=\cupprod(\Xfunc)([a,b])+\cupprod(\Yfunc)([a,b]),
    \\
    \cupprod\left(\Xfunc\amalg \Yfunc\right)([a,b])&=\len\left(\bfH^*(f_{a}^{b}\amalg g_{a}^{b})\right)\\
    &=\len\left(\bfH^*(f_{a}^{b})\times \bfH^*(g_{a}^{b})\right)\\
    &=\max\left\{\len\left(\bfH^*(f_{a}^{b})\right),\len\left(\bfH^*(g_{a}^{b})\right)\right\}\\
    &=\max\left\{\cupprod(\Xfunc)([a,b]),\cupprod(\Yfunc)([a,b])\right\}\text{, and}
    \\
        \cupprod\left(\Xfunc\vee \Yfunc\right)([a,b])&=\len\left(\bfH^*(f_a^b\vee g_{a}^{b})\right)\\
    &=\len\left(\bfH^*(f_{a}^{b})\times \bfH^*(g_{a}^{b})\right)\\
    &=\max\left\{\len\left(\bfH^*(f_{a}^{b})\right),\len\left(\bfH^*(g_{a}^{b})\right)\right\}\\
    &=\max\left\{\cupprod(\Xfunc)([a,b]),\cupprod(\Yfunc)([a,b])\right\}.
\end{align*}
\end{proof}

\noindent\textbf{Visualization of persistent cup-length invariant.} Each interval $[a,b]$ in $\Int$ is visualized as a point $(a,b)$ in the half-plane above the diagonal (see Fig.~\ref{figure:interval-triangle}).
To visualize the persistent cup-length invariant of a filtration $\Xfunc$, we assign to each point $(a,b)$ the integer value $\cupprod(\Xfunc)([a,b])$, if it is positive. If $\cupprod(\Xfunc)([a,b])=0$ we do not assign any value. We present an example to demonstrate how persistent cup-length invariants are visualized in the upper-diagonal plane (see Fig.~\ref{fig:torus_cup_len_func_and_dgm}).

\begin{figure}[H]
\centering
\begin{tikzpicture}[scale=0.4]
    \begin{axis} [ 
    axis y line=middle, 
    axis x line=middle,
    ticks=none,
    xmin=-.5, xmax=3.5,
    ymin=-.5, ymax=3.5,]
    \addplot [mark=none] coordinates {(-.5,-.5) (3.5,3.5)};
    \addplot [mark=none,dashed] coordinates {(1,0) (1,1)};
    \addplot [mark=none,dashed] coordinates {(3,0) (3,3)};
    \addplot [thick,color=red!60!white,fill=red!50!white, 
                    fill opacity=0.45]coordinates {
            (1,3)
            (1,1)
            (3,3)  
            (1,3)};
    \node[mark=none] at (axis cs:1,-0.3){\LARGE $a$};
    \node[mark=none] at (axis cs:3,-0.3){\LARGE $b$};
    \node[mark=none] at (axis cs:1,3.3){\LARGE $(a,b)$};
    \end{axis}
    \end{tikzpicture}
 \caption{The interval $[a,b]$ in $\Int$ corresponds to the point $(a,b)$ in $\bbR^2$.  }
 \label{figure:interval-triangle}
\end{figure}
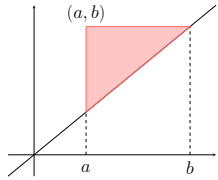

\begin{example} [$\bbS^1$ and $\bbT^d$: visualization of $\cupprod(\cdot)$] \label{ex:S1-Tn}
Let $\bbS^1$ be the geodesic circle with radius $1$, and consider the Vietoris-Rips filtration $\VR_\bullet(\bbS^1)$. In \citep{adamaszek2017vietoris}, the authors computed the homotopy types of Vietoris-Rips complexes of $\bbS^1$ at all scale parameters. Following from their results, the persistent graded ring $\bfH^*(\VR_\bullet(\bbS^1))$ is given by
\[
\bfH^*(\VR_r(\bbS^1))\cong 
      \begin{cases}
       \bfH^*(\bbS^{2l+1}) , &\parbox[t]{.5\textwidth}{ if $[a,b] \subseteq \left(\tfrac{l}{2l+1}2\pi ,\tfrac{l+1}{2l+3}2\pi \right)$, for some $ l=0,1,\dots$}\\
       0 , &\mbox{otherwise,}
     \end{cases}
\] 
where the map $\bfH^*(\VR_s(\bbS^1))\to \bfH^*(\VR_r(\bbS^1))$ is an isomorphism if $\tfrac{l}{2l+1}2\pi<r\leq s<\tfrac{l+1}{2l+3}2\pi$, and is $0$ otherwise. We compute the persistent cup-length invariant of $\VR_\bullet(\bbS^1)$ and obtain: for any $a\leq b$,
\begin{equation*}\label{eq:VR_S1}
  \cupprod(\VR_\bullet(\bbS^1))([a,b]) = 
      \begin{cases}
       1, &\parbox[t]{.5\textwidth}{ if $[a,b] \subseteq \left(\tfrac{l}{2l+1}2\pi ,\tfrac{l+1}{2l+3}2\pi \right)$, for some $ l=0,1,\dots$}\\
       0, &\mbox{ otherwise,}
      \end{cases}
\end{equation*}
which is equal to the rank of $\bfH^*(\VR_b(\bbS^1))\to \bfH^*(\VR_a(\bbS^1))$ (viewed as a linear map).

As an application of Prop.~\ref{prop:prod-coprod-pers-cup}, we also study the persistent cup-length invariant of the Vietoris-Rips filtration of the $d$-torus 
$\bbT^d :=\underbrace{\bbS^1\times \bbS^1\times\dots\times \bbS^1}_{d-\text{times}}$, for some integer $d\geq2$. Here $\bbT^d$ is the $\ell_\infty$-product of the $d$ unit geodesic circles. For any $[a,b]\in\Int$, by \citep[Prop.~10.2]{adamaszek2017vietoris} and Prop.~\ref{prop:prod-coprod-pers-cup} we have
\begin{align*}
    &\cupprod\left(\VR_\bullet(\bbT^d )\right)([a,b])\\
=&\cupprod\left(\VR_\bullet(\bbS^1)\times\dots\times\VR_\bullet(\bbS^1))\right)([a,b])\\
=&d\cdot \cupprod(\VR_\bullet(\bbS^1))([a,b]).
\end{align*}
We draw visualizations for both $\cupprod\left(\VR_\bullet(\bbS^1 )\right)$ and $\cupprod\left(\VR_\bullet(\bbT^d )\right)$ in Fig.~\ref{ex:fig_S1-Tn}.

\begin{figure}[H]
\centering
    \begin{tikzpicture}[scale=0.6]
    \begin{axis} [ 
    axis y line=middle, 
    axis x line=middle,
    ytick={0.5,1},
    yticklabels={$\tfrac{\pi}{2}$,$\pi$},
    xtick={0.5,1},
    xticklabels={$\tfrac{\pi}{2}$,$\pi$},
    xmin=0, xmax=1.1,
    ymin=0, ymax=1.1,]
    \addplot [mark=none,color=dgmcolor!20!white] coordinates {(0,0) (1,1)};
    \addplot [thick,color=dgmcolor!20!white,fill=dgmcolor!20!white, 
                    fill opacity=0.45]coordinates {
            (0,.67) 
            (0,0)
            (.67,.67)
            (0,.67)};
    \addplot [thick,color=dgmcolor!20!white,fill=dgmcolor!20!white, 
                    fill opacity=0.45]coordinates {
            (0.67,.8) 
            (.67,.67)
            (.8,.8)
            (0.67,.8)};
    \addplot [thick,color=dgmcolor!20!white,fill=dgmcolor!20!white, 
                    fill opacity=0.45]coordinates {
            (0.67,.8) 
            (.67,.67)
            (.8,.8)
            (0.67,.8)};
    \addplot [thick,color=dgmcolor!20!white,fill=dgmcolor!20!white, 
                    fill opacity=0.45]coordinates {
            (0.8,.857) 
            (.8,.8)
            (.857,.857)
            (0.8,.857)};
    \addplot [thick,color=dgmcolor!20!white,fill=dgmcolor!20!white, 
                    fill opacity=0.45]coordinates {
            (0.857,.89) 
            (.857,.857)
            (.89,.89)
            (0.857,.89) };
    \addplot [thick,color=dgmcolor!20!white,fill=dgmcolor!20!white, 
                    fill opacity=0.45]coordinates {
            (0.89,.95) 
            (.89,.89)
            (.95,.95)
            (0.89,.95) };
    \node[mark=none] at (axis cs:.74,.76){\tiny{\textsf{1}}};
    \node[mark=none] at (axis cs:.25,.45){\textsf{1}};
    \node[mark=none] at (axis cs:.7,.3){$\cupprod\left(\VR_\bullet(\bbS^1 )\right)$};
    \end{axis}
    \end{tikzpicture}
    \hspace{1.5cm}
    \begin{tikzpicture}[scale=0.6]
    \begin{axis} [ 
    axis y line=middle, 
    axis x line=middle,
    ytick={0.5,1},
    yticklabels={$\tfrac{\pi}{2}$,$\pi$},
    xtick={0.5,1},
    xticklabels={$\tfrac{\pi}{2}$,$\pi$},
    xmin=0, xmax=1.1,
    ymin=0, ymax=1.1,]
    \addplot [mark=none,color=dgmcolor!50!white] coordinates {(0,0) (1,1)};
    \addplot [thick,color=dgmcolor!50!white,fill=dgmcolor!50!white, 
                    fill opacity=0.45]coordinates {
            (0,.67) 
            (0,0)
            (.67,.67)
            (0,.67)};
    \addplot [thick,color=dgmcolor!50!white,fill=dgmcolor!50!white, 
                    fill opacity=0.45]coordinates {
            (0.67,.8) 
            (.67,.67)
            (.8,.8)
            (0.67,.8)};
    \addplot [thick,color=dgmcolor!50!white,fill=dgmcolor!50!white, 
                    fill opacity=0.45]coordinates {
            (0.67,.8) 
            (.67,.67)
            (.8,.8)
            (0.67,.8)};
    \addplot [thick,color=dgmcolor!50!white,fill=dgmcolor!50!white, 
                    fill opacity=0.45]coordinates {
            (0.8,.857) 
            (.8,.8)
            (.857,.857)
            (0.8,.857)};
    \addplot [thick,color=dgmcolor!50!white,fill=dgmcolor!50!white, 
                    fill opacity=0.45]coordinates {
            (0.857,.89) 
            (.857,.857)
            (.89,.89)
            (0.857,.89) };
    \addplot [thick,color=dgmcolor!50!white,fill=dgmcolor!50!white, 
                    fill opacity=0.45]coordinates {
            (0.89,.95) 
            (.89,.89)
            (.95,.95)
            (0.89,.95) };
    \node[mark=none] at (axis cs:.74,.76){\tiny{\textsf{d}}};
    \node[mark=none] at (axis cs:.25,.45){\textsf{d}};
    \node[mark=none] at (axis cs:.7,.3){$\cupprod\left(\VR_\bullet(\bbT^d )\right)$};
    \end{axis}
    \end{tikzpicture}
    \caption{Persistent cup-length invariants $\cupprod\left(\VR_\bullet(\bbS^1 )\right)$ and $\cupprod\left(\VR_\bullet(\bbT^d )\right)$. See Ex.~\ref{ex:S1-Tn}}
\label{ex:fig_S1-Tn}
\end{figure}
\end{example}

\subsubsection{Persistent cup-length diagram and computation of the persistent cup-length invariant}

In this section, we recall from \citep[Sec.~3]{contessoto_et_al:LIPIcs.SoCG.2022.31} the notion of the
\textit{persistent cup-length diagram} of a filtration, defined by using a family of representative cocycles, and recall that the persistent cup-length invariant can be retrieved from the persistent cup-length diagram (cf. Thm.~\ref{thm:tropical_mobius}).

\begin{definition} [Support of $\ell$-fold products]
\label{def:ell-product}
Let $\bsigma$ be a family of representative cocycles for $\bfH^*(\Xfunc) $. Let $\ell\in\N^+$ and let $I_1,\dots,I_\ell$ be a sequence of elements in $\barc\!\left(\bfH^*(\Xfunc)\right)$ with representative cocycles $\sigma_{I_1},\dots,\sigma_{I_\ell}\in\bsigma$, respectively. 
Consider the $\ell$-fold product $\sigma_{I_1} \smile\dots\smile \sigma_{I_\ell}$.
We define the \textbf{support} of $\sigma_{I_1} \smile\dots\smile \sigma_{I_\ell}$
to be  
\begin{equation}\label{eq:support}
\operatorname{supp}(\sigma_{I_1} \smile\dots\smile \sigma_{I_\ell}):= \left\{t\in \bbR\mid [ \sigma_{I_1}]_t \smile\dots\smile [\sigma_{I_\ell}]_t \neq [0]_t\right\}.
\end{equation}
\end{definition}

\begin{proposition}
\label{prop:support}
With the same assumption and notation in Defn.~\ref{def:ell-product}, let $I:=\operatorname{supp}(\sigma_{I_1} \smile\dots\smile \sigma_{I_\ell})$. If $I\neq \emptyset$, then $I$ is an interval $\linterval  b,d\rinterval $, where $b\leq d$ are such that $d$ is the right end of $\cap_{1\leq i\leq \ell}I_i$ and $b$ is the left end of some $I'\in \barc\!\left(\bfH^*(\Xfunc)\right)$ ($I'$ is not necessarily one of the $I_i$).
\end{proposition}
\begin{proof} We prove in the case of closed intervals. For the other types of intervals, the statement follows from a similar discussion. 

Let $d$ be the right end of $\cap_{1\leq i\leq \ell}I_i$. Clearly, any $t>d$ is not in $I$, because there is some $I_i$ such that $[\sigma_{I_i}]_t=[0]_t$. To show $d$ is the right end of $I$, it suffices to show that $d$ is in $I$. 
If $d\notin I$, then it follows from $[ \sigma_{I_1}]_d \smile\dots\smile [\sigma_{I_\ell}]_d = [0]_d$ that $[ \sigma_{I_1}]_t \smile\dots\smile [\sigma_{I_\ell}]_t = [0]_t$ for all $t\leq d$. Thus, $I=\emptyset$, which gives a contradiction. Therefore, $d$ is the right end of $I$.

We show that $I$ in an interval, i.e. for any $t\in I$ and $s\in [t,d]$, we have $s\in I$. This is true because $[ \sigma_{I_1}]_s \smile\dots\smile [\sigma_{I_\ell}]_s$, as the preimage of a non-zero element $[ \sigma_{I_1}]_t \smile\dots\smile [\sigma_{I_\ell}]_t$, cannot be zero.

Assume the left end of $I$ is $b$. Then $[ \sigma_{I_1}]_b \smile\dots\smile [\sigma_{I_\ell}]_b\neq 0$ but $[ \sigma_{I_1}]_{b-\epsilon} \smile\dots\smile [\sigma_{I_\ell}]_{b-\epsilon}=0$ for any $\epsilon>0$. Notice that we can write the cup product $[\sigma_{I_1}]_b \smile\dots\smile [\sigma_{I_\ell}]_b=\sum \lambda_{I'} [\sigma_{I'}]_b$ for some coefficients $\lambda_{I'}$ and distinct representative cocycles $\sigma_{I'}$ with $[\sigma_{I'}]_b\neq 0$, where $I'\in \barc\!\left(\bfH^*(\Xfunc)\right)$. For any $\epsilon>0$, it follows from $[ \sigma_{I_1}]_{b-\epsilon} \smile\dots\smile [\sigma_{I_\ell}]_{b-\epsilon}=0$ and the linear independence of $[\sigma_{I'}]_{b-\epsilon}$ that $[\sigma_{I'}]_{b-\epsilon}=0$ for every $I'$. Thus, these $I'$ are bars with left end equal to $b$. 
\end{proof}

\begin{example}[$\operatorname{supp}(\alpha\smile\beta)\neq I_\alpha\cap I_\beta$]
\label{ex:supp v.s. intersection}
Consider the filtration $\Xfunc=\{ X_t\}_{t\geq0}$ of a pinched $2$-torus and its total barcode, as shown in Fig.~\ref{fig:pinched_torus}. Here $\alpha$ is the $1$-cocycle born at $t=1$; $\beta$ is the $1$-cocycles born at $t=2$; $v$ and $\gamma$ be the $0$-cocycle and $2$-cocycle, respectively. Notice that $I_\alpha\cap I_\beta = [1,3)$, while $\operatorname{supp}(\alpha\smile\beta) = [2,3).$ 

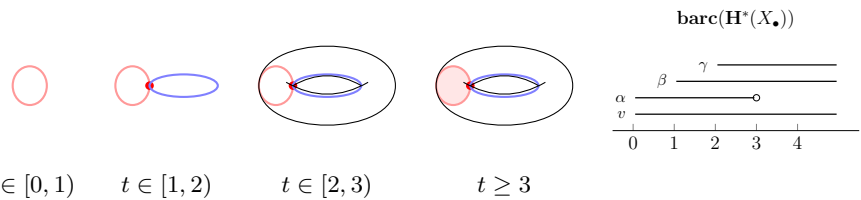
\begin{figure}[H]
\centering
\begin{tikzcd}[column sep=tiny,row sep=0pt]
    \begin{tikzpicture}[scale=0.3]
    \filldraw[color=red!40, fill=none, thick](0.75,2) ellipse (0.75 and .86);
    \end{tikzpicture}
	& 
	\begin{tikzpicture}[scale=0.3]
	   \filldraw[color=red!40, fill=none, thick](0.75,2) ellipse (0.75 and .86);
    \node[fill=red,circle,inner sep=.4mm] at (1.5,2)  {};
    \filldraw[color=blue!50, fill=none, thick](3,2) ellipse (1.5 and .5);
    \end{tikzpicture}
	& 
	\begin{tikzpicture}[scale=0.3]
	   \filldraw[color=red!40, fill=none, thick](0.75,2) ellipse (0.75 and .86);
    \node[fill=red,circle,inner sep=.4mm] at (1.5,2)  {};
    \filldraw[color=blue!50, fill=none, thick](3,2) ellipse (1.5 and .5);
    \Coordinate{left}{0,2}
    \Coordinate{right}{6,2}
    \Coordinate{leftinner}{1.5,2}
    \Coordinate{rightinner}{4.5,2}
        \draw (left) to[out=90,in=90,looseness=1] 
        (right) to [out=270,in=270,looseness=1] (left);
        \draw (leftinner) to[out=30,in=150,looseness=1] 
        (rightinner); 
        \draw (1.2,2.15) to[out=330,in=210,looseness=1] 
        (4.8,2.15); 
    \end{tikzpicture}
& \begin{tikzpicture}[scale=0.3]
	   \filldraw[color=red!40, fill=red!10, thick](0.75,2) ellipse (0.75 and .86);
    \node[fill=red,circle,inner sep=.4mm] at (1.5,2)  {};
    \filldraw[color=blue!50, fill=none, thick](3,2) ellipse (1.5 and .5);
    \Coordinate{left}{0,2}
    \Coordinate{right}{6,2}
    \Coordinate{leftinner}{1.5,2}
    \Coordinate{rightinner}{4.5,2}
        \draw (left) to[out=90,in=90,looseness=1] 
        (right) to [out=270,in=270,looseness=1] (left);
        \draw (leftinner) to[out=30,in=150,looseness=1] 
        (rightinner); 
        \draw (1.2,2.15) to[out=330,in=210,looseness=1] 
        (4.8,2.15); 
    \end{tikzpicture}
&
\begin{tikzpicture}[scale=.6] 
    \begin{axis} [ 
    ticklabel style = {font=\large},
    height=4.5cm,
    width=7cm,
    hide y axis,
    axis x line*=bottom,
    xtick={1,2,3,4,5},
    xticklabels={$0$, $1$, $2$, $3$, $4$},
    xmin=.5, xmax=6.5,
    ymin=0, ymax=1.6,]
    \addplot [mark=none,thick] coordinates {(1.05,.2) (5.95,.2)};
    \addplot [mark=none,thick] coordinates {(1.05,.4) (3.95,.4)};
    \addplot [mark=none,thick] coordinates {(2.05,.6) (5.95,.6)};
    \addplot [mark=none,thick] coordinates {(3.05,.8) (5.95,.8)};
    \addplot [mark=none] coordinates {(1,.2)}  node[left] {$v$};
    \addplot [mark=none] coordinates {(2,.6)}  node[left] {$\beta$};
    \addplot [mark=none] coordinates {(1,.4)}  node[left] {$\alpha$};
    \addplot [mark=none] coordinates {(3,.8)}  node[left] {$\gamma$};
    \addplot [mark=o] coordinates {(4,.4)};
    \node[mark=none] at (axis cs:3.5,1.3){\large$\barc\!\left(\bfH^*(\Xfunc)\right)$};
    \end{axis}
    \end{tikzpicture}
\\
\text{\small{$t\in [0,1)$}} 
& \text{\small{$t\in [1,2)$}} 
& \text{\small{$t\in [2,3)$}} 
& \text{\small{$t\geq 3$}}
&
	\end{tikzcd}
\caption{A filtration $\Xfunc$ of a pinched $2$-torus and its total barcode. See Ex.~\ref{ex:supp v.s. intersection}.} 
\label{fig:pinched_torus}
\end{figure}
\end{example}

Because the cup product operation commutes up to a scalar: for any pair $\alpha,\beta$ of cochains, $\alpha\smile\beta=(-1)^s\beta\smile\alpha\text{, for some integer }s$, we immediately have the following proposition.

\begin{proposition} \label{prop:propertiy_of_*_sigma}
Let $I_1,\dots,I_\ell$, be as in Defn.~\ref{def:ell-product}. 
The support $\operatorname{supp}(\sigma_{{I}_1} \smile\dots\smile \sigma_{{I}_\ell})$ 
is symmetric, i.e.~for any permutation $\rho$ of $\{1,2,\dots,\ell\}$, we have
\[\operatorname{supp}(\sigma_{{I}_1} \smile\dots\smile \sigma_{{I}_\ell})= \operatorname{supp}(\sigma_{{I}_{\rho(1)}} \smile\dots\smile \sigma_{{I}_{\rho(\ell)}}).\]
\end{proposition}

Let $\barc\!\left(\bfH^+(\Xfunc)\right)$ consist of the positive-degree bars in the barcode of $\Xfunc$. 

\begin{definition}[Persistent cup-length diagram]
\label{def:cup_dgm}
Let $\bsigma=\{\sigma_I\}_{I\in\barc\!\left(\bfH^+(\Xfunc)\right)}$ be a family of representative cocycles for $\bfH^{ + }(\Xfunc) $.
The \textbf{persistent cup-length diagram of $\Xfunc$ (associated to $\bsigma$)} is defined to be the map $\sigmadgmX:\Int\to\N$, given by
\footnote{For notational consistency, in this paper we use $\dgm\left(\cupprod(\cdot),\bsigma\right)$ to denote persistent cup-length diagrams, which is slightly different from the notation  $\dgm_{\bsigma}^\smile(\cdot)$ used in \citep{contessoto_et_al:LIPIcs.SoCG.2022.31}.}:
\begin{align*}
I\mapsto \max\left\{\ell\in\N^+\mid I=\operatorname{supp}(\sigma_{{I}_1} \smile\dots\smile \sigma_{{I}_\ell})\text{, for some }I_i\in\barc\!\left(\bfH^+(\Xfunc)\right)\right\},
\end{align*}
with the convention that $\max\emptyset=0.$
\end{definition}

Recall from \citep[Ex.~18]{contessoto_et_al:LIPIcs.SoCG.2022.31} that the persistent cup-length diagram depends on the choice of the representative cocycles $\bsigma$. However, the persistent cup-length diagram can always be used to compute the persistent cup-length invariant (regardless of the choice of $\bsigma$), through the following theorem.\footnote{This theorem was stated without proof as Thm.~1 of the conference paper \citep{contessoto_et_al:LIPIcs.SoCG.2022.31}.} 
\begin{theorem}
\label{thm:tropical_mobius}
Let $\Xfunc$ be a filtration, and let $\bsigma$ be a family of representative cocycles for the barcode of $\Xfunc$.
The persistent cup-length invariant $\cupprod(\Xfunc)$ can be retrieved from the persistent cup-length diagram $\sigmadgmX$: for any $[a,b]\in \Int$,
\begin{equation}\label{eq:tropical_mobuis}
    \cupprod(\Xfunc)([a,b])=  \max_{[c,d]\supseteq   [a,b]}\sigmadgmX([c,d]). 
\end{equation}
\end{theorem}
\begin{remark}
\label{rem:tropical-Mobius-Inversion}
The persistent cup-length invariant is analogous to the rank invariant: In standard persistence theory, for each interval $[a,b]$ the rank invariant $\rank(\Mfunc)$ of a persistence module $\Mfunc$ 
counts the \textbf{sum} of the multiplicities of the intervals in the barcode $\barc(\Mfunc)$ of $\Mfunc$ that contain $[a,b]$ (see \citep[pg.~106]{cohen2007stability}), i.e. 
$$\rank(\Mfunc)([a,b])=\sum_{[c,d]\supset[a,b]}\dgm(\Mfunc)([c,d]),$$ 
where the multiplicity function $\dgm(\Mfunc)(\cdot)$ is  the persistent diagram of $\Mfunc$.
 Eqn. (\ref{eq:tropical_mobuis}) expresses the fact that a similar relation exists in the case of the persistent cup-length invariant $\cupprod(\Xfunc)$ but with the difference that the cup-length counts the \textbf{maximum} number (instead of the sum) of non-zero cup products of cocycles. That is, we prove that by switching the `sum' with `max' operation (which resembles a `tropical' M\"obius inversion formula) $\cupprod(\Xfunc)$ can be recovered from  $\sigmadgmX$. 
\end{remark}

\begin{proof}[Proof of Thm.~\ref{thm:tropical_mobius}]
\label{pf:tropical_mobius} 
Let $I:=[a,b]$ be a closed interval. We first consider the case when $\cupprod(\Xfunc)([a,b])=0$, in which case the image ring $\image \left(\bfH^*( X_b) \to \bfH^*( X_a) \right)$ is trivial in positive dimensions. We claim that for any $[c,d]\supseteq [a,b]$, $\sigmadgmX([c,d])=0$. Assume not, then $\sigmadgmX([c,d])>0$ for some $[c,d]\supseteq [a,b]$, which necessarily means that there is a bar associated with a positive-degree cocycle that contains $[a,b]$. This contradicts the fact that $\image \left(\bfH^{ + }( X_b)\to \bfH^{ + }( X_a)\right)=0.$ 
\medskip

We now assume $\cupprod(\Xfunc)([a,b])\neq0$ and define 
\[B:=\{[\sigma_{I'}]_a\mid \barc\!\left(\bfH^+(\Xfunc)\right)\ni {I'}\supseteq   [a,b]\}.\]
Recall that for ${I'}=[c,d]$ in the barcode, $\sigma_{I'}$ is a cocycle in $ X_d$ and $[\sigma_{I'}]_a$ is the cohomology class of the restriction $\sigma_{I'}\vert_{C_p( X_a)}$, if the dimension of $\sigma_{I'}$ is $p$.
Then,
\begin{align}
\cupprod(\Xfunc)([a,b]) 
=\,& \len\left(\image \left(\bfH^*( X_b) \to \bfH^*( X_a) \right)\right) \label{eq1}  \\
=\,& \len\left(\linterval  B \rinterval \right) \label{eq2}  \\
=\,& \max \left\{\ell\in\N^+\mid B^{\smile\ell}\neq \{0\}\right\}. \label{eq3} 
\end{align}
Eqn.~(\ref{eq1}) follows from the definition of the persistent cup-length invariant, and Eqn.~(\ref{eq2}) is a direct application of Prop.~\ref{prop:generators-of-ring}, where $\linterval \cdot\rinterval $ denotes the generating set of a ring. 
Because $B$ linearly spans the image $\image(\bfH^{ + }( X_b)\to \bfH^{ + }( X_a))$ in each dimension, the assumption of Prop.~\ref{prop:len_of_basis} is satisfied and thus Eqn.~(\ref{eq3}) follows. 

\medskip

Given $I_1',\dots,I_\ell'\in \barc\!\left(\bfH^+(\Xfunc)\right)$ such that ${I'}_i\supseteq  [a,b]$ for each $i$, we claim that
\[[\sigma_{I_1'} ]_a\smile\dots\smile[\sigma_{I_\ell'} ]_a\neq 0 \iff \operatorname{supp}(\sigma_{I_1'} \smile\dots\smile \sigma_{I_\ell'})\supseteq  [a,b].\]
The `$\Leftarrow$' is trivial. As for `$\Rightarrow$', recall from Prop.~\ref{prop:support} that in this case the support is a non-empty interval with its right end equal to the right end of $\cap_i {I'}_i\supseteq [a,b]$. It follows that the support, as an interval, contains both $a$ and $b$, and thus containing $[a,b]$.

Therefore, we have Eqn.~(\ref{eq4}) below:
\begin{align}
&\cupprod(\Xfunc)([a,b]) \nonumber\\
=\, & \max \left\{\ell\in\N^+\mid B^{\smile\ell}\neq \{0\}\right\} \nonumber\\
=\,& \max \left\{\ell\in\N^+\mid [\sigma_{I_1'} ]_a\smile\dots\smile[\sigma_{I_\ell'} ]_a\neq 0,\,{I'}_i\supseteq  [a,b],\, {I'}_i\in \barc\!\left(\bfH^+(\Xfunc)\right),\, \forall i \right\} \nonumber\\
=\,& \max \left\{\ell\in\N^+\mid \operatorname{supp}(\sigma_{I_1'}\smile\dots\smile\sigma_{I_\ell'})\supseteq  [a,b],\,{I'}_i\in\barc\!\left(\bfH^+(\Xfunc)\right),\, \forall i \right\} \label{eq4}\\
 =\,& \max_{[c,d]\supseteq  [a,b]} \left\{\max\left\{\ell\in\N^+\mid [c,d]=\operatorname{supp}(\sigma_{I_1'}\smile\dots\smile\sigma_{I_\ell'}) \text{, where }{I'}_i\in\barc\!\left(\bfH^+(\Xfunc)\right)\right\}
 \right\} \label{eq-second-last}\\
 =\,& \max_{[c,d]\supseteq  [a,b]}\sigmadgmX([c,d]). \label{eq-last}
\end{align}
Here Eqn.~(\ref{eq-second-last}) and Eqn.~(\ref{eq-last}) follow from the definition of the support of $\ell$-fold products 
(Defn.~\ref{def:ell-product}) and the definition of the persistent cup-length diagram (Defn.~\ref{def:cup_dgm}), respectively.
\end{proof} 

We compute the persistent cup-length diagrams of some filtrations, and utilize Eqn.~(\ref{eq:tropical_mobuis}) in Thm.~\ref{thm:tropical_mobius} for computing the persistent cup-length invariants of these filtrations. We assume the convention that $[1,\infty],[2,\infty],\ldots$ are intervals in $\Int$.

\begin{example}
\label{ex:cup_len_func_torus} 
Recall the filtration $\Xfunc=\{ X_t\}_{t\geq0}$ of a pinched $2$-torus and its total barcode from Fig.~\ref{fig:pinched_torus}. Also recall from Ex.~\ref{ex:supp v.s. intersection} the $1$-cocycles representatives  $\alpha$ and $\beta$ and the $2$-cocycle representative $\gamma$.
Let $\bsigma:=\{\alpha,\beta,\gamma\}$.

Because $\bfH^*(\Xfunc) $ is non-trivial up to degree $2$, $\sigmadgmX(I)\leq 2$ for any $I$.
The only non-trivial cup product is $\alpha\smile\beta$, whose support is $[2,3)$. Thus, $\sigmadgmX([2,3))=2$. 
Therefore, the persistent cup-length diagram $\sigmadgmX$ is (see the left-most figure in Fig.~\ref{fig:torus_cup_len_func_and_dgm} for its visualization):
\[
\sigmadgmX(I)= 
      \begin{cases}
    1, & \mbox{ if } I=[0,3), [1,\infty) \text { or } [2,\infty)\\
    2, & \mbox{ if } I=[2,3) \\
    0, & \mbox{otherwise.} 
     \end{cases}
\]
Applying Thm.~\ref{thm:tropical_mobius}, we obtain the persistent cup-length invariant $\cupprod(\Xfunc)$, visualized in Fig.~\ref{fig:torus_cup_len_func_and_dgm}:
\[
  \cupprod(\Xfunc)([t,s]) = 
      \begin{cases}
       2 , & \text{if } [t,s)\subseteq [2,3) \\
       1 , & \text {if } 0\leq t\leq 2 \text{ and } s\leq 3 \text{; or } 1\leq t< \infty \text{ and } s\geq 3\\
       0 , & otherwise.
      \end{cases}
\] 
\begin{figure}
\centering
      \begin{tikzpicture}[scale=0.6]
    \begin{axis} [ 
    ticklabel style = {font=\large},
    axis y line=middle, 
    axis x line=middle,
    ytick={1,2,3,4,5.3},
    yticklabels={$1$,$2$,$3$,$4$,$\infty$},
    xticklabels={$1$,$2$,$3$,$4$,$\infty$},
    xtick={1,2,3,4,5.3},
    xmin=0, xmax=5.5,
    ymin=0, ymax=5.5,]
    \addplot [mark=none] coordinates {(0,0) (5.3,5.3)};
  \addplot[dgmcolor!40!white,mark=*] (0,3) circle (3pt) node[right,black]{\textsf{1}};
    \addplot[dgmcolor!40!white,mark=*] (1,5.3) circle (3pt) node[below,black]{\textsf{1}};
  \addplot[dgmcolor!40!white,mark=*] (2,5.3) circle (3pt) node[below,black]{\textsf{1}};
  \addplot[dgmcolor!100!white,mark=*] (2,3) circle (3pt) node[below,black]{\textsf{2}};
    \node[mark=none] at (axis cs:3.5,1.5){$\sigmadgmX$};
    \end{axis}
    \end{tikzpicture}
    \hspace{1.5cm}
    \begin{tikzpicture}[scale=0.6]
    \begin{axis} [ 
    ticklabel style = {font=\large},
    axis y line=middle, 
    axis x line=middle,
    ytick={1,2,3,4,5.3},
    yticklabels={$1$,$2$,$3$,$4$,$\infty$},
    xticklabels={$1$,$2$,$3$,$4$,$\infty$},
    xtick={1,2,3,4,5.3},
    xmin=0, xmax=5.5,
    ymin=0, ymax=5.5,]
    \addplot [mark=none] coordinates {(0,0) (5.3,5.3)};
    \addplot [thick,color=dgmcolor!20!white,fill=dgmcolor!20!white, 
                    fill opacity=0.45]coordinates {
            (0,0)
            (0,3) 
            (1,3)
            (1,5.3)  
            (5.3,5.3)
            (0,0)};
    \addplot [thick,color=dgmcolor!40!white,fill=dgmcolor!40!white, 
                    fill opacity=0.45]coordinates {
            (2,3) 
            (2,2)
            (2,2)
            (3,3)
            (2,3)};
    \node[mark=none] at (axis cs:2,4){\textsf{1}};
    \node[mark=none] at (axis cs:2.3,2.6){\textsf{2}};
    \node[mark=none] at (axis cs:3.5,1.5){$\cupprod(\Xfunc)$};
    \end{axis}
    \end{tikzpicture}
\caption{The persistent cup-length invariant $\cupprod(\Xfunc)$ and its persistent cup-length diagram $\sigmadgmX$ (see Ex.~\ref{ex:cup_len_func_torus}) for $\Xfunc$ a filtration of a pinched $2$-torus given in Fig.~\ref{fig:pinched_torus}.} 
\label{fig:torus_cup_len_func_and_dgm}
\end{figure}
\end{example}

To compute the persistent cup-length invariant, it suffices to compute the persistent cup-length diagram. 
For a finite simplicial filtration $\Xfunc: X_1\hookrightarrow\dots\hookrightarrow  X_N(= X)$, let $\barc\!\left(\bfH^+(\Xfunc)\right)$ be the barcode over positive dimensions and $\bsigma:=\{\sigma_I\}_{I\in \barc\!\left(\bfH^+(\Xfunc)\right)}$ a family of representative cocycles. For any $\ell\geq 1$, let $\Sigma_\ell$ be the collection of all
$\operatorname{supp}(\sigma_{{I}_1} \smile\dots\smile \sigma_{{I}_\ell})$
where each $I_i\in \barc\!\left(\bfH^+(\Xfunc)\right)$. Then the persistent cup-length diagram is obtained by first computing $\left\{\Sigma_\ell \right \}_{\ell\geq 1}$ using: 

\begin{algorithm}
\begin{algorithmic}
\label{alg:main-simple}
\While{$\Sigma_\ell\neq \emptyset$}
    \For{$(I_1, \sigma_1)\in \barc\!\left(\bfH^+(\Xfunc)\right)$ and $(I_2, \sigma_2)\in \Sigma_\ell$}
        \If{$\operatorname{supp}(\sigma_{{I}_1} \smile \sigma_{{I}_2})\neq\emptyset$}
            \State Append $(\operatorname{supp}(\sigma_{{I}_1} \smile \sigma_{{I}_2}),\sigma_1\smile\sigma_2)$ to $\Sigma_{\ell+1}$
        \EndIf
    \EndFor
    \State $\ell\gets \ell+1$
\EndWhile
\end{algorithmic}
\end{algorithm}

See \citep[Sec.~3.4]{contessoto_et_al:LIPIcs.SoCG.2022.31} for the detailed algorithm and a proof that our algorithm runs in polynomial-time in the total number of simplices.

\subsection{Persistent LS-category invariant}
\label{sec:ls-cat}

In this section, we study another example of categorical invariants, the LS-category of topological spaces. 
Then we lift it to a persistent invariant, which we call the \emph{persistent LS-category invariant}. 

One expects the persistent LS-category invariant to be difficult to compute, but it will be seen in Prop.~\ref{prop:cup<cat} that the persistent cup-length invariant serves as a computable lower bound estimate of the persistent LS-category.

The LS-category of a space was introduced by Lyusternik and Schnirelmann for providing lower bounds on the number of critical points for smooth functions on a manifold \citep{lusternik1934methodes}. The \emph{LS-category of a map} was first defined by Fox \citep{fox1941lusternik} and subsequently studied by Berstein and Ganea \citep{berstein1962category}. We recall the definitions of the LS-category of spaces and maps:

\begin{definition}[{\citep[Defn.~1.1]{cornea2003lusternik}}]
\label{def:cat(X)}
Let $ X$ be a topological space. The \textbf{LS-category of $ X$}, denoted by $\cat( X)$, is the  least number $n$ (or $+\infty$) of open sets $U_1,\dots,U_{n+1}$ in $ X$ that cover $ X$ such that each inclusion $U_i\hookrightarrow  X$ is null-homotopic (i.e. $U_i$ is contractible to a point in $ X$).
\end{definition}

\begin{definition}[{\citep[Defn.~1.1]{berstein1962category}}]
\label{def:cat(f)}
The \textbf{LS-category of a continuous map $f: X\to  Y$}, denoted by $\cat(f)$, is the least number $n$ (or $+\infty$) such that $ X$ can be covered by open sets $U_1,\dots,U_{n+1}$ such that each $f\vert_{U_i}$ is null-homotopic (i.e. $f\vert_{U_i}$ is homotopic to a constant map from $U_i$ to $X$).
\end{definition}

We recall the following properties of LS-category from \citep{berstein1962category,cornea2003lusternik}, 
which guarantees that the LS-category yields a categorical invariant (see Defn.~\ref{def:categorical invariant})
of $\Top$ even though it is not an epi-mono invariant (see Ex.~\ref{rmk:cat not epi-mono}). 

\begin{proposition}
\label{prop:property-cat(f)}
Let $f: X\to  Y$ be a map of topological spaces.
\begin{enumerate}
    \item  \label{prop:cat(f) v.s. cat(X)}
    $\cat(f)\leq\min\{\cat( X),\cat(Y)\}$. If $f$ is a homotopy equivalence, then $$\cat(f)=\cat( X)=\cat(Y).$$
    \item \label{prop:cat(composition)}
    $\cat(g\circ f)\leq \min\{\cat(f),\cat(g)\}$, for any pair of  maps $f: X\to  Y$ and $g: Y\to  Z$. 
    \item \label{prop: cat of h.e.}
    $\cat(f_1) =\cat(f_2)$, if $f_1$ and $f_2$ are homotopic
    to each other.
    \item \label{prop:cat v.s. cup}
    $\cupprod(f):=\len (\image(\bfH^*(f)))\leq\cat(f)$, where $\bfH^*(f)$ is the map on cohomology induced by $f$. In particular,
    \[\cupprod( X)=\cupprod(\id_{ X})\leq \cat(\id_{ X})=\cat( X).\]
\end{enumerate}
\end{proposition}

\begin{remark}
\label{rmk:cat not epi-mono}
The invariant $\cat(\cdot)$ is not an epi-mono invariant. 
Consider the embedding $\iota:\bbS^1\hookrightarrow\mathbb{D}^2$. By Prop.~\ref{prop:property-cat(f)} (\ref{prop:cat(f) v.s. cat(X)}), we have \[\cat(\bbS^1)=1>0=\cat(\iota)=\cat(\mathbb{D}^2).\]
In particular, this implies that $\cat(\iota)=0$  is not equal to $\cat(\mathbf{im}\iota)=\cat(\bbS^1)=1$.
\end{remark}

\begin{example} \label{ex:loop space}
The inequality $\cupprod( X)\leq\cat( X)$ can be strict. For a topological space $ X$, let $L(X)$ be its free loop space, i.e. the set of unbased loops equipped with the compact-open topology. By \citep[Thm. 9.3]{cornea2003lusternik}, if $ X$ be any simply-connected space of finite type (all its homology groups are finitely generated) and non-trivial reduced rational homology, then $\cat(L( X))=\infty.$ For instance (cf. \citep[Rmk. 9.10]{cornea2003lusternik}), for the two-dimensional sphere $\mathbb{S}^2$, we have
\[\cupprod\left(L\!\left(\mathbb{S}^2\right)\right)=1<+\infty = \cat\left(L\!\left(\mathbb{S}^2\right)\right).\]
\end{example}

\begin{example} \label{ex:cat v.s. cup} 
The cup-length and LS-category are not necessarily stronger invariants than each other. 
For instance, the spaces $\mathbb{S}^2$ and the free loop space $L\!\left(\mathbb{S}^2\right)$ have identical cup-length but different LS-category: by Ex. \ref{ex:loop space}, we have 
\[\cupprod(\mathbb{S}^2)=1=\cupprod\left(L\!\left(\mathbb{S}^2\right)\right),\,\cat(\mathbb{S}^2)= 1<\infty=\cat\left(L\!\left(\mathbb{S}^2\right)\right).\]

On the other hand, the spaces $L\!\left(\mathbb{S}^2\times \mathbb{S}^2\right)$ 
and $L\!\left(\mathbb{S}^2\right)$ have identical LS-category but different cup-length:
\[\cupprod\left(L\!\left(\mathbb{S}^2\times \mathbb{S}^2\right)\right)=2>1=\cupprod\left(L\!\left(\mathbb{S}^2\right)\right),\,\cat\left(L\!\left(\mathbb{S}^2\times \mathbb{S}^2\right)\right)= \infty=\cat\left(L\!\left(\mathbb{S}^2\right)\right).\]
\end{example}

Because the LS-category is a categorical invariant, we can lift it to a persistent invariant as:

\begin{definition}[Persistent LS-category]
\label{def:pers-ls}
Given a persistent space $\Xfunc:(\bbR,\leq)\to\Top$, the functor \[\cat(\Xfunc):(\Int,\subseteq)\to(\N,\leq)^{\mathrm{op}}, \text{ with }[a,b]\mapsto \cat( X_b\to  X_a)
\]
is called \textbf{the persistent LS-category invariant}. 
\end{definition}

\begin{restatable}{proposition}{cupcat}
\label{prop:cup<cat}
For any persistent space $\Xfunc:(\bbR,\leq)\to\Top$,
$$\cupprod(\Xfunc)(\cdot)\leq \cat(\Xfunc)(\cdot).$$
\end{restatable}

\begin{proof}
The proof follows directly by Prop.~\ref{prop:property-cat(f)} (3) and the definitions of the persistent cup length invariant and the persistent LS-category of a persistent space.
\end{proof}

We see in the following example that using the persistent cup-length invariant and Prop.~\ref{prop:cup<cat} can help us compute the persistent LS-category.

\begin{example}[Example of $\cat(\Xfunc)$] 
\label{ex:cat(two disk)}
Let $\Xfunc=\{ X_t\}_{t\geq0}$ be a filtration of the wedge sum of two $2$-disks, as shown in Fig.~\ref{fig:dcircle-filt}. In order to compute the persistent LS-category of $\Xfunc$ from its definition, one needs to figure out the LS-category of the (non-identity) transition maps in $\Xfunc$. Let us instead compute the persistent cup-length invariant first.

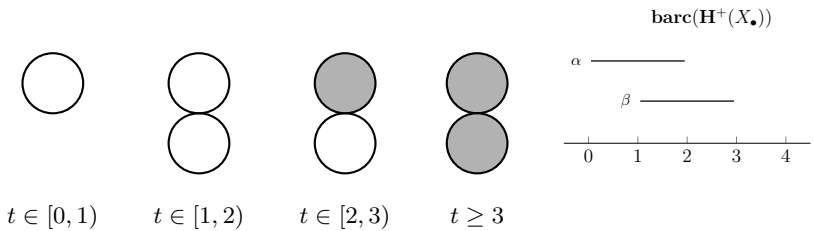
\begin{figure}[H]
\centering
\begin{tikzcd}[column sep=small,row sep=tiny]
\begin{tikzpicture}[scale=0.4]
	   \filldraw[fill=none,thick](0,2) circle (1);
    \end{tikzpicture}
	& 
	\begin{tikzpicture}[scale=0.4]
	   \filldraw[fill=none,thick](0,2) circle (1);
	   \filldraw[fill=none,thick](0,0) circle (1);
    \end{tikzpicture}
	& 
	\begin{tikzpicture}[scale=0.4]
	   \filldraw[fill=black!30,thick](0,2) circle (1);
	   \filldraw[fill=none,thick](0, 0) circle (1);
    \end{tikzpicture}
    & 
	\begin{tikzpicture}[scale=0.4]
	   \filldraw[fill=black!30,thick](0,2) circle (1);
	   \filldraw[fill=black!30,thick](0,0) circle (1);
    \end{tikzpicture}
    &
    	\begin{tikzpicture}[scale=.6] 
    \begin{axis} [ 
   ticklabel style = {font=\large},
    width=7cm,
    hide y axis,
    axis x line*=bottom,
    xtick={1,2,3,4,5},
    xticklabels={$0$, $1$, $2$, $3$, $4$},
    xmin=.5, xmax=5.5,
    ymin=0, ymax=1,]
    \addplot [mark=none,thick] coordinates {(1.05,.41) (2.95,.41)};
    \addplot [mark=none,thick] coordinates {(2.05,.21) (3.95,.21)};
    \addplot [mark=none] coordinates {(1,.41)}  node[left] {$\alpha$};
    \addplot [mark=none] coordinates {(2,.21)}  node[left] {$\beta$};
    \node[mark=none] at (axis cs:3.5,0.6){\large$\barc\!\left(\bfH^+(\Xfunc)\right)$};
    \end{axis}
    \end{tikzpicture}
    \\
        \text{\small{$t\in [0,1)$}} 
    & \text{\small{$t\in [1,2)$}} 
    & \text{\small{$t\in [2,3)$}} 
    & \text{\small{$t\geq 3$}} 
    &
	\end{tikzcd}
\caption{A filtration $\Xfunc$ of the wedge sum of two $2$-disks and its positive-degree barcode, where $\alpha$ and $\beta$ are the $1$-cocycles corresponding to the top and bottom circle, respectively. See Ex.~\ref{ex:non-example}.}
\label{fig:dcircle-filt}
\end{figure}

Let $\bsigma:=\{\alpha,\beta\}$. Because all elements in $\bsigma$ have trivial cup products with each other, we have the persistent cup-length diagram 
as below (see the left-most figure in Fig.~\ref{fig:dcirclerep} for its visualization):
\[
\sigmadgmX(I)= 
      \begin{cases}
    1, & \mbox{ if } I=[0,2)  \text { or } I=[1,3)\\
    0, & \mbox{otherwise.} 
     \end{cases}
\]
Applying Thm.~\ref{thm:tropical_mobius}, we obtain the persistent cup-length invariant $\cupprod(\Xfunc)$, visualized in Fig.~\ref{fig:dcirclerep}:
\[
  \cupprod(\Xfunc)([t,s]) = 
      \begin{cases}
       1 , & \text{if }[t,s] \subseteq [0,2) \text{ or } [1,3) \\
       0 , & otherwise.
      \end{cases}
\] 

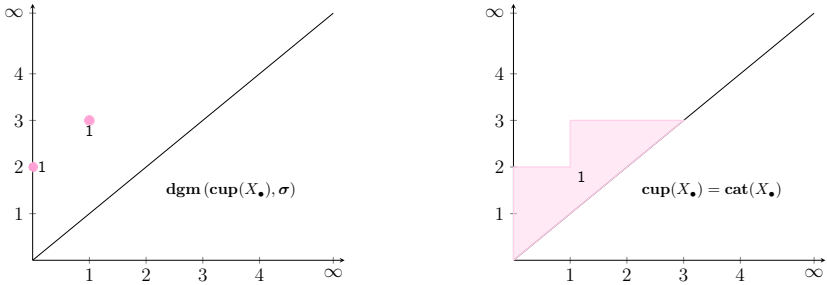
\begin{figure}[H]
\centering
      \begin{tikzpicture}[scale=0.6]
    \begin{axis} [ 
    ticklabel style = {font=\large},
    axis y line=middle, 
    axis x line=middle,
    ytick={1,2,3,4,5.3},
    yticklabels={$1$,$2$,$3$,$4$,$\infty$},
    xticklabels={$1$,$2$,$3$,$4$,$\infty$},
    xtick={1,2,3,4,5.3},
    xmin=0, xmax=5.5,
    ymin=0, ymax=5.5,]
    \addplot [mark=none] coordinates {(0,0) (5.3,5.3)};
  \addplot[dgmcolor!40!white,mark=*] (0,2) circle (3pt) node[right, black]{\textsf{1}};
    \addplot[dgmcolor!40!white,mark=*] (1,3) circle (3pt) node[below,black]{\textsf{1}};
    \node[mark=none] at (axis cs:3.5,1.5){$\sigmadgmX$};
    \end{axis}
    \end{tikzpicture}
    \hspace{1.5cm}
    \begin{tikzpicture}[scale=0.6]
    \begin{axis} [ 
    ticklabel style = {font=\large},
    axis y line=middle, 
    axis x line=middle,
    ytick={1,2,3,4,5.3},
    xtick={1,2,3,4,5.3},
    xticklabels={$1$, $2$, $3$, $4$, $\infty$}, 
    yticklabels={$1$, $2$, $3$, $4$, $\infty$}, 
    xmin=0, xmax=5.5,
    ymin=0, ymax=5.5,]
    \addplot [mark=none] coordinates {(0,0) (5.3,5.3)};
    \addplot [thick,color=dgmcolor!20!white,fill=dgmcolor!20!white, 
                    fill opacity=0.45]coordinates {
            (1,3) 
            (1,2)
            (0,2)  
            (0,0)
            (3,3)
            (1,3)};
    \node[mark=none] at (axis cs:1.2,1.8){\textsf{1}};
     \node[mark=none] at (axis cs:3.5,1.5){$\cupprod(\Xfunc)=\cat(\Xfunc)$};
    \end{axis}
    \end{tikzpicture}
\caption{The persistent cup-length diagram (left) and the persistent cup-length (or LS-category) invariant (right) of $\Xfunc$, where $\Xfunc$ is the filtration given in Fig.~\ref{fig:dcircle-filt}.}
\label{fig:dcirclerep}
\end{figure}

We now compute the persistent LS-category of $\Xfunc$. For any $[t,s] \subseteq [0,2) \text{ or } [1,3)$, it follows from Prop.~\ref{prop:cup<cat} that $\cat(\Xfunc)([t,s])\geq \cupprod(\Xfunc)([t,s])=1$; it follows from Prop.~\ref{prop:property-cat(f)} (\ref{prop:cat(f) v.s. cat(X)}) that $\cat(\Xfunc)([t,s])\leq \cat( X_s)\leq 1$. Therefore, we have $\cat(\Xfunc)([t,s])\leq \cat( X_s)=1=\cupprod(\Xfunc)([t,s])$.

For $[t,s] $ that is not a subset of $ [0,2) \text{ or } [1,3)$, we show that $\cat(\Xfunc)([t,s])=0=\cupprod(\Xfunc)([t,s])$ by considering different cases:
\begin{itemize}
    \item if $s\in [3,\infty)$, then $\cat(\Xfunc)([t,s])\leq \cat( X_s) = \cat(\mathbb{D}^2\vee \mathbb{D}^2)=0$;
    \item if $s\in [2,3)$ and $t\in [0,1)$, then $$\cat(\Xfunc)([t,s]) = \cat(\mathbb{S}^1\hookrightarrow\mathbb{D}^2\vee \mathbb{S}^1)=\cat(\mathbb{S}^1\hookrightarrow\mathbb{D}^2)\leq \cat(\mathbb{D}^2)=0.$$
\end{itemize}

In summary, we have  proved that $\cat(\Xfunc)([t,s])=\cupprod(\Xfunc)([t,s])$ for any $t\leq s$.
\end{example}

In Sec.~\ref{sec:homo-stab-per-inv}, we will show that the erosion distance between the persistent cup-length (or persistent LS-category) invariants is stable under the homotopy-interleaving distance of persistent spaces, cf. Cor.~\ref{cor:stab-cup} (or Cor.~\ref{cor:stab-cat}). It is worth noticing that even though persistent cup-length serves as a pointwise lower bound of persistent LS-category, the latter is not necessarily a stronger invariant than the former one, nor vice versa. See the example below:
\begin{example} \label{ex:cat v.s. cup p-version}
A constant filtration of $X$ is a filtration $\Xfunc$ such that $X_t=X$ for all $t$ and all transition maps are the identity map on $X$. The phenomenon in the static case that cup-length and LS-category are not necessarily stronger than each other can be easily extended to the persistent setting, by considering the constant filtrations of spaces in Ex.~\ref{ex:cat v.s. cup}.
\end{example}

\subsection{M\"obius inversion of persistent invariants}\label{sec:mobuis}

In this section, we study the Möbius inversion of $\invariant (\Ffunc)$ for a given persistent
object 
$\Ffunc:(\bbR,\leq)\to\catC$ and a categorical invariant $\invariant $.

First, we recall the concept of M\"obius inversion in the sense of Rota \citep{rota1964foundations}.
\begin{definition}[{\citep[Prop.~1 (pg.344)]{rota1964foundations}}]
Let $\mathbcal{Q}=(\mathbcal{Q},\leq)$ be a locally finite poset. We define the \textbf{M\"obius function} $\mu_\mathbcal{Q}:\mathbcal{Q}\times \mathbcal{Q}\to \bbZ$, given recursively by the formula 
\[
\mu_\mathbcal{Q}(p,q) =
\begin{cases} 
      1, & p=q, \\
      -\sum_{p\leq r<q}\mu_\mathbcal{Q}(p,r), & p<q,\\
      0, & otherwise.
   \end{cases}
\]
\end{definition}

We recall the following result of Rota's: 
\begin{proposition}[{\citep[Prop.~2 (pg.344)]{rota1964foundations}}]
    Let $\mathbcal{Q}=(\mathbcal{Q},\leq)$ be a locally finite poset with an initial element $0$ (i.e.~$0\leq q$, for all $q\in \mathbcal{Q}$) and let $\field$ be a  field. Let $f,g:\mathbcal{Q}\to \field$ be a pair of functions. If $f(q)=\sum_{p\leq q}g(q)$ for  $q\in \mathbcal{Q}$, then $g$ is given point-wisely by
    $$g(q)=\sum_{p\leq q}f(p)\mu_\mathbcal{Q}(p,q)\text{, for }q\in \mathbcal{Q}.$$
    The function $g$ will be called the \textbf{M\"obius inversion} of $f$.
\end{proposition}

Following \citep[Defn.~2.2]{patel2018generalized}, we consider a certain constructibility condition on persistent objects, in which case the M\"obius inversion of a persistent invariant associated to such persistent objects exists. Let $\catC$ be a category with an identity object $e$. For a set of real number $\{s_1<\cdots<s_m\}$, a persistent object $\Ffunc:(\bbR,\leq)\to \catC$  is said to be \textit{$\{s_1<\cdots<s_m\}$-constructible}, if $F_t\to F_s$ is an isomorphism when $[t,s]\subseteq[s_i,s_{i+1}]$ for some $i$ or $[t,s]\subseteq[s_m,\infty)$, and $F_t\to F_s$ is the identity on $e$ when $[t,s]\subseteq (-\infty,s_1).$ \label{para:constructibility}

\begin{definition}\label{def:I-dgm}
Let $\Ffunc$ be an $\{s_1<\cdots<s_m\}$-constructible persistent object. 
Given any persistence $\invariant $-invariant $\invariant (\Ffunc):\Int\to \N$ (viewed as a function), 
we define the \textbf{persistence $\invariant $-diagram}\footnote{We only consider what Patel called the type $\mathfrak{A}$ persistence diagram \citep[Defn.~7.1]{patel2018generalized}. In \citep[Defn.~7.2]{patel2018generalized}, the author also considered another notion of persistence diagram when the underlying category is abelian. 
In this paper, we face categories that are not abelian, such as the category of rings and the category of flags (see Rmk.~\ref{rmk:non-abelian}).
} (associated to $\Ffunc$) 
$\dgm(\invariant(\Ffunc))(\cdot):\Int\to \bbZ$ point-wisely as   
\begin{equation}\label{eq:mobius-I}
\begin{split}
    \dgm(\invariant(\Ffunc))(\linterval  s_i,s_j\rinterval ):= & \invariant(\Ffunc) (\linterval  s_i,s_j\rinterval ) - \invariant(\Ffunc) (\linterval  s_{i-1},s_j\rinterval ) \\
    &- \invariant(\Ffunc) (\linterval  s_i,s_{j+1}\rinterval ) + \invariant(\Ffunc) (\linterval  s_{i-1},s_{j+1}\rinterval ),
\end{split}
\end{equation}
$    \dgm(\invariant(\Ffunc))(\linterval  s_i,\infty\rinterval ) := \invariant(\Ffunc) (\linterval  s_i,\infty\rinterval ) - \invariant(\Ffunc) (\linterval  s_{i-1},\infty\rinterval )$, and $\dgm(\invariant(\Ffunc))(I):=0$ otherwise.
\end{definition}

\begin{proposition}
The persistence $\invariant $-diagram of $\dgm(\invariant(\Ffunc))$ in Defn.~\ref{def:I-dgm} agrees with the M\"obius inversion of $\invariant (\Ffunc):\Int\to \field$ in the sense of Rota, i.e. 
\begin{equation}\label{eq:mobius I = sum dgm}
    \invariant(\Ffunc)([a,b]) = \sum_{[c,d]\supseteq  [a,b]} \dgm(\invariant(\Ffunc))([c,d]).
\end{equation}
\end{proposition}
\begin{proof}
    The proof is formally the same as in \citep[Thm.~4.1]{patel2018generalized} and omitted. 
\end{proof}

\begin{example}\label{ex:rk and dgm} 
Recall from Sec.~\ref{sec:pers-theo} the notion of the standard persistence module and its associated persistence diagram. 
Given a $\{s_1<\cdots<s_m\}$-constructible persistence module $\Mfunc:(\bbR,\leq)\to\Vect$, we can recover the persistence diagram $\dgm(\Mfunc)$ of $\Mfunc$ element-wise from the rank invariant $\rank(\Mfunc):\Int\to\N$ of $\Mfunc$, via Eqn.~(\ref{eq:mobius-I}). 
Reversely, $\rank(\Mfunc)$ can be obtained from $\dgm(\Mfunc)$ via Eqn.~(\ref{eq:mobius I = sum dgm}).
\end{example}

Finally, note that the Möbius inversion of the rank invariant of a constructible persistence module $\Mfunc$ is always non-negative, is due to the fact that $\Mfunc$ is interval decomposable (see \citep[Thm.~1.1]{crawley2015decomposition}).

\begin{example}[M\"obius inversion of $\cupprod(\Xfunc)$ or $\cat(\Xfunc)$ can be negative]
\label{ex:non-example}
Let $\Xfunc$ be the filtration of the wedge sum of two $2$-disks given in Fig.~\ref{fig:dcircle-filt}, and recall from Ex.~\ref{ex:cat(two disk)} that $\cupprod(\cdot)=\cat(\cdot)=: \invariant.$
If we consider the singleton interval $[1,1]=\{1\}$, then the M\"obius inversion of $ \invariant$ applied to 
$[1,1]$ is negative, i.e.
\begin{align*}
 \dgm(\invariant(\Xfunc))([1,1])&= \invariant([1,1])- \invariant([0,1])- \invariant([1,2]))+ \invariant([0,2])\\
&=1-1-1+0=-1<0.
\end{align*}
\end{example}

\section{Stability of persistent invariants} \label{sec:all-stability}

In Sec.~\ref{sec:cat-stab-per-inv}, we recall the notions of the interleaving distance between persistent objects (see Defn.~\ref{def:interleaving-dist}) and the erosion distance $d_{\mathrm{E}}$ between persistent invariants (see Defn.~\ref{def:de}). We show the following categorical stability for any persistent invariant:

\distab*

In Sec.~\ref{sec:homo-stab-per-inv}, we show that the erosion distance $d_{\mathrm{E}}$ between persistent invariants that arise from weak homotopy invariants is stable under the homotopy interleaving $d_{\mathrm{HI}}$ (see Defn.~\ref{def:dhi}) between persistent spaces. Subsequently, for persistent spaces arising from Vietoris-Rips filtrations, we establish the stability of persistent invariants under the Gromov-Hausdorff distance $d_{\mathrm{GH}}$ between metric spaces. 

\homostab*

By checking that the persistent cup-length invariant of persistent spaces and the persistent LS-category of persistent CW complexes satisfy the assumptions in the above theorem, we obtain the following two corollaries:

\homostabcup*

\homostabcat*

\subsection{Categorical stability of persistent invariants}
\label{sec:cat-stab-per-inv}
We first recall the definition of the interleaving distance between persistent objects and the notion of erosion distance between persistent invariants. 

\begin{definition}[Interleaving distance, {\citep[Defn.~3.20]{bubenik2015metrics}}]
\label{def:interleaving-dist}
Let $\catC$ be any category. 
Let $\Ffunc,\Gfunc:(\mathbb{R},\leq)\to \catC$ be a pair of persistent objects. $\Ffunc,\Gfunc$ are said to be $\epsilon$-interleaved if there exists a pair of natural transformations $\varphi=(\varphi_t: F_t\to  G_{t+\epsilon})_{t\in\mathbb{R}}$ and $\psi=(\psi_t: G_t\to  F_{t+\epsilon})_{t\in\mathbb{R}}$, i.e. the diagrams 
\[
\xymatrix
{
 F_a\ar[dr]_{\varphi_a}\ar[rr]^{f_{a}^{b}} &   &  F_{b}\ar[dr]_{\varphi_b} &  &  G_a\ar[dr]_{\psi_a}\ar[rr]^{g_{a}^{b}} &   &  G_{b} \ar[dr]_{\psi_b}\\
	 &  G_{a+\epsilon} \ar[rr]_{g_{a+\epsilon}^{b+\epsilon}} & &
	  G_{b+\epsilon}
	 & &  F_{a+\epsilon} \ar[rr]_{f_{a+\epsilon}^{b+\epsilon}}
	 & &
	 F_{b+\epsilon}
	 \\
}
\]
commute for all $a\leq b$ in $\bbR$; and such that the diagrams
\[
\xymatrix
{
 F_t\ar[dr]_{\varphi_t}\ar[rr]^{f_{t}^{t+2\epsilon}}  &  &  F_{t+2\epsilon} & & &  G_t\ar[dr]_{\psi_t}\ar[rr]^{g_{t}^{t+2\epsilon}}  &  &  G_{t+2\epsilon} \\
	 &  G_{t+\epsilon} \ar[ur]_{\psi_{t+\epsilon}} & & & & &  F_{t+\epsilon} \ar[ur]_{\varphi_{t+\epsilon}}\\
}
\]
commute for all $t\in\mathbb{R}$.
The interleaving distance between $\Ffunc$ and $\Gfunc$ is
$$\di^{\catC}(\Ffunc,\Gfunc):= \inf\lbrace \epsilon\geq 0\mid\text{ there is an }\epsilon\text{-interleaving between }\Ffunc\text{ and }\Gfunc\rbrace.$$
\end{definition}

\begin{definition}[Erosion distance,
{\citep[Defn.~5.3]{patel2018generalized}}]
\label{def:de}
Let $\bfJ_1,\bfJ_2: \Int\to (\catN,\leq)^{\mathrm{op}}$ be two functors. 
$\bfJ_1,\bfJ_2$ are said to be \textit{$\epsilon$-eroded} if $\bfJ_1([a,b])\geq \bfJ_2([a-\epsilon,b+\epsilon])$ and $\bfJ_2([a,b])\geq \bfJ_1([a-\epsilon,b+\epsilon])$, for all $[a,b]\in\Int$.
The \textit{erosion distance of $\bfJ_1,\bfJ_2$} is 
$$d_{\mathrm{E}}(\bfJ_1,\bfJ_2):=\inf \lbrace \epsilon\geq0
\mid \bfJ_1,\bfJ_2\text{ are }\epsilon\text{-eroded}\},
$$
with the convention that $d_{\mathrm E}(\bfJ_1,\bfJ_2)=\infty$ if an $\epsilon$ satisfying the condition above does not exist.
\end{definition}

\begin{proof}[Proof of Thm.~\ref{thm:stab-per-inv}]
Denote by $f_a^b: F_a\to  F_b$ and $g_a^b: G_a\to  G_b$, $a\leq b$, the associated morphisms from $\Ffunc$ to $\Gfunc$.
Assume that $\Ffunc,\Gfunc$ are $\epsilon$-interleaved.
Then, there exist two $\bbR$-indexed families of morphisms $\varphi_t: F_t\to  G_{t+\epsilon}$ and $\psi_t: G_t\to   F_{t+\epsilon}$, which are natural for all $t\in\bbR$, such that $\psi_{t+\epsilon}\circ \varphi_t=f_t^{t+2\epsilon}$ and $\varphi_{t+\epsilon}\circ \psi_t=g_t^{t+2\epsilon}$, for all $t\in\bbR$.
Let $[a,b]\in\Int$. We claim that $\invariant (g_{a-\epsilon}^{b+\epsilon})\leq \invariant \left(f_a^b\right)$. If we show this, then similarly we can show the symmetric inequality, and therefore obtain that $\invariant (\Ffunc),\invariant (\Gfunc)$ are $\epsilon$-eroded.
Indeed, the claim is true because
\begin{align*}
    \invariant (g_{a-\epsilon}^{b+\epsilon})&=\invariant (g_{b-\epsilon}^{b+\epsilon}\circ g_{a-\epsilon}^{b-\epsilon})\\
    &=\invariant (\varphi_{b}\circ \psi_{b-\epsilon}\circ g_{a-\epsilon}^{b-\epsilon})\\
    &\leq \invariant (\psi_{b-\epsilon}\circ g_{a-\epsilon}^{b-\epsilon})\hspace{1em}(\text{by  condition (ii) of Defn.~\ref{def:categorical invariant}})\\
    &=\invariant (f_{a}^{b}\circ \psi_{a-\epsilon} )\hspace{1em}(\text{by naturality of }\psi)\\
    &\leq \invariant (f_{a}^{b}) \hspace{1em}(\text{by  condition (ii) of Defn.~\ref{def:categorical invariant}}).
\end{align*}
\end{proof}

\subsection{Homotopical stability of persistent invariants}
\label{sec:homo-stab-per-inv}
To prove 
Thm.~\ref{thm:main-stability}, we first recall the definition of the homotopy-interleaving distance, together with certain results from persistence theory. 

Following the terminology in \citep[Defn.~1.7]{blumberg2017universality}, a pair of persistent spaces $\Xfunc,\Yfunc:(\bbR,\leq)\to\Top$ are called \textbf{weakly equivalent},
denoted by $\Xfunc\simeq\Yfunc$, if there exists a persistent space $\Zfunc:(\bbR,\leq)\to\Top$ and a pair of natural transformations
$\varphi:\Zfunc\Rightarrow\Xfunc$ and $\psi:\Zfunc\Rightarrow\Yfunc$
such that for each $t\in\bbR$, the maps $\varphi_t: Z_t\to X_t$ and $\psi_t: Z_t\to  Y_t$ are weak homotopy equivalences, i.e., they induce isomorphisms on all homotopy groups. 

\begin{definition}[The homotopy interleaving distance, {\citep[Defn.~3.6]{blumberg2017universality}}]
\label{def:dhi}
Let $\Xfunc,\Yfunc:(\mathbb{R},\leq)\to \Top$ be a pair of persistent spaces.
The \textit{homotopy interleaving distance of $\Xfunc,\Yfunc$} is
$$d_{\mathrm{HI}}(\Xfunc,\Yfunc)=\inf\left\{\di^{\Top}(\Xfunc',\Yfunc')\hspace{0.5em}\mid\Xfunc'\simeq \Xfunc\text{ and }\Yfunc'\simeq \Yfunc\right\}.$$
\end{definition}

\begin{proposition}[{\citep[Prop. 1.9 \& Sec.~6.1]{blumberg2017universality}}]
\label{prop:dh-dgh}
For compact metric spaces $X$ and $Y$, 
\[\dhi\left( \VR_\bullet(X),\VR_\bullet(Y)\right)\leq 2\cdot \dgh(X,Y).\]
\end{proposition}

To prove Thm.~\ref{thm:main-stability}, we establish the following lemma:

\begin{lemma}
\label{lem:w.h.e.-per-inv}
Let $\invariant $ be a categorical invariant satisfying the condition that for any maps $ X\xrightarrow{f} Y\xrightarrow{g} Z\xrightarrow{h}W$ where $g$ is a weak homotopy equivalence, $\invariant (g\circ f)=\invariant (f)$ and $\invariant (h\circ g)=\invariant (h)$. If $\Xfunc\simeq \Xfunc'$, then $\invariant (\Xfunc)=\invariant (\Xfunc').$
\end{lemma}

We apply the Lem.~\ref{lem:w.h.e.-per-inv} and Thm.~\ref{thm:stab-per-inv} to prove Thm.~\ref{thm:main-stability}, which states that certain categorical weak homotopy invariant is stable under the homotopy-interleaving distance between persistent spaces.

\begin{proof}[Proof of Thm.~\ref{thm:main-stability}] 
Let $\Xfunc,\Yfunc:(\bbR,\leq)\to\Top$ be two persistent spaces. For any pair $\Xfunc',\Yfunc':(\bbR,\leq)\to\Top$ of persistent spaces such that $\Xfunc'\simeq \Xfunc$ and  $\Yfunc'\simeq \Yfunc$, we have
$$ d_{\mathrm{E}}(\invariant (\Xfunc),\invariant (\Yfunc))=d_{\mathrm{E}}\left(\invariant (\Xfunc'),\invariant (\Yfunc')\right)\leq \di^{\Top}(\Xfunc',\Yfunc'),$$
where the leftmost equality follows from Lem.~\ref{lem:w.h.e.-per-inv} and the rightmost inequality follows from Thm.~\ref{thm:stab-per-inv}.
Thus, Eqn.~(\ref{eq:dE-dHI}) follows.

For the case of Vietoris-Rips filtrations of metric spaces, the statement follows from Prop.~\ref{prop:dh-dgh}.
\end{proof}

\begin{proof}[Proof of Lem.~\ref{lem:w.h.e.-per-inv}]
Since $\Xfunc'\simeq \Xfunc$, there exists a persistent space $\Zfunc:(\bbR,\leq)\to\Top$ and natural transformations $\Xfunc\xLeftarrow{\varphi}\Zfunc\xRightarrow{\psi}\Xfunc'$, such that for each $t\in\bbR$, the maps $\varphi_t: Z_t\to X_t$ and $\psi_t: Z_t\to X'_t$ are weak homotopy equivalences. We claim that $\invariant (\Xfunc)=\invariant (\Zfunc)$, i.e. for any $t\leq s$, $\invariant (\Xfunc)([t,s])=\invariant (\Zfunc)([t,s]).$ Indeed, for the following commutative diagram:
	\begin{center}
		\begin{tikzcd}[column sep = 4em]
		 Z_t \ar[r,"g_t^s"]
		 \ar[d,"\simeq" right, "\varphi_t" left]
        &
         Z_s
        \ar[d,"\simeq" left, "\varphi_s" right]
        \\
		  X_t \ar[r,"f_t^s" below]
        &
         X_s,
		\end{tikzcd}
	\end{center} 
because $\varphi_t$ and $\varphi_s$ are weak homotopy equivalence, we have
\[\invariant (f_t^s)=\invariant (f_t^s\circ\varphi_t)=\invariant (g_t^s\circ\varphi_s)=\invariant (g_t^s).\]
\end{proof}

We prove Cor.~\ref{cor:stab-cup} and Cor.~\ref{cor:stab-cat}:

\begin{proof}[Proof of Cor.~\ref{cor:stab-cup}] 
It suffices to show that weak homotopy equivalence preserves cohomology algebras. Indeed, let $g: Y\to  Z$ be a weak homotopy equivalence. By \citep[Prop.~4.21]{hatcher2000}, the map $g$ induces a graded linear isomorphism $\bfH^*(g):\bfH^*( Z) \to\bfH^*( Y) $. On the other hand, the induced map $\bfH^*(g)$ preserves the cup product operation. Thus, $\bfH^*(g)$ is a graded algebra isomorphism, and it follows that
\[\cupprod(g\circ f) = \len(\image(\bfH^*(g)\circ \bfH^*(f))) =\len(\image( \bfH^*(f))) =\cupprod(f),\]
and similarly $\cupprod(h\circ g)=\cupprod(h).$
\end{proof}

\begin{proof}[Proof of Cor.~\ref{cor:stab-cat}] 
The Whitehead theorem states that for CW complexes weak homotopy equivalences are homotopy equivalences. It follows from Prop.~\ref{prop:property-cat(f)} (\ref{prop: cat of h.e.}) that the LS-category is an invariant satisfying the required condition. Indeed, for any maps $ X\xrightarrow{f} Y\xrightarrow{g} Z\xrightarrow{h}W$ where $g$ is a weak homotopy equivalence of CW complexes (and thus a homotopy equivalence), we have 
\[g\circ f\sim f\implies \cat(g\circ f) = \cat(f)\]
and similarly $h\circ g\sim h\implies \cat(h\circ g) = \cat(h).$
\end{proof}

The strength of the persistent cup-length invariant at discriminating filtrations has been demonstrated by several examples in \citep{contessoto_et_al:LIPIcs.SoCG.2022.31}. 
In particular, the following example was stated there without proof. Here, we provide detailed proofs for the case of the persistent cup-length invariant, and also extend it to the case of the persistent LS-category invariant. In Rmk.~\ref{rmk:inter-torus-wedge}, we compute the interleaving distance between the persistent homology of these two spaces and see that the persistent cup-length (or LS-category) invariant provides a better approximation of the Gromov-Hausdorff distance of the two spaces than persistent homology. 

The wedge sum $X\vee_{x_0\sim y_0} Y$ (in short, $X\vee Y$) of two $(X,x_0)$ and $(Y,y_0)$ is the quotient space of the disjoint union of $X$ and $Y$ by the identification of basepoints $x_0\sim y_0$. 
Recall from \citep{burago2001course,adamaszek2020homotopy} that the \emph{gluing metric} on $X\vee Y$ is given by
\label{para:gluing}
	$$d_{X\vee Y}(x,y):= d_X(x,x_0)+d_Y(y,y_0),\forall x\in X, y\in Y$$ 
and $d_{X\vee Y}\vert_{X\times X}=d_X,d_{X\vee Y}\vert_{Y\times Y}=d_Y$.

\begin{example}[$\VR(\bbT^2)$ v.s. $\VR(\bbS^1\vee\bbS^2\vee\bbS^1)$]
\label{ex:T2-wedge}
Let the $2$-torus $\bbT^2=\bbS^1\times \bbS^1$ be the $\ell_\infty$-product of two unit geodesic circles.
Let $\bbS^2$ be the unit $2$-sphere, equipped with the geodesic distance, and denote by $\bbS^1\vee \bbS^2\vee \bbS^1$ the wedge sum equipped with the gluing metric.
Using the characterization of Vietoris-Rips complex of $\bbS^1$ given by \citep{adamaszek2017vietoris}, we obtain the persistent cup-length invariant of $\VR(\bbS^1)$. Combined with Prop.~\ref{prop:prod-coprod-pers-cup}, we obtain the persistent cup-length invariant of $\VR(\bbT^2)$ (see Fig.~\ref{fig:torus_sphere_VR}): for any interval $[a,b]$,
\begin{align*}
    \cupprod\left(\VR(\bbT^2)\right)([a,b])=
      \begin{cases}
      2, &\mbox{ if $[a,b] \subset \left(\tfrac{l}{2l+1}2\pi ,\tfrac{l+1}{2l+3}2\pi \right)$ for some $ l=0,1,\dots$}\\
      0, &\mbox{ otherwise.}
      \end{cases}
\end{align*}
Via a similar discussion, we obtain the persistent LS-category invariant of $\VR(\bbT^2)$ which turns out to be the same as $\cupprod\left(\VR(\bbT^2)\right)$. Indeed, $\VR(\bbT^2)$ consists of the homotopy types of even-dimensional torus, and for any $n$, the LS-category of an $n$-dimensional torus is $2$ the same as its cup-length.

\begin{figure}[H]
\centering
    \begin{tikzpicture}[scale=0.6]
    \begin{axis} [ 
    ticklabel style = {font=\Large},
    axis y line=middle, 
    axis x line=middle,
    ytick={0.5,0.67,0.95},
    yticklabels={$\tfrac{\pi}{2}$, $\tfrac{2\pi}{3}$,$\pi$},
    xtick={0.5,0.67,0.95},
    xticklabels={$\tfrac{\pi}{2}$,$\tfrac{2\pi}{3}$, $\pi$},
    xmin=0, xmax=1.1,
    ymin=0, ymax=1.1,]
   \addplot [mark=none,color=dgmcolor!40!white] coordinates {(0,0) (1,1)};
    \addplot [thick,color=dgmcolor!40!white,fill=dgmcolor!40!white, 
                    fill opacity=0.45]coordinates {
            (0,.67) 
            (0,0)
            (.67,.67)
            (0,.67)};
    \addplot [thick,color=dgmcolor!40!white,fill=dgmcolor!40!white, 
                    fill opacity=0.45]coordinates {
            (0.67,.8) 
            (.67,.67)
            (.8,.8)
            (0.67,.8)};
    \addplot [thick,color=dgmcolor!40!white,fill=dgmcolor!40!white, 
                    fill opacity=0.45]coordinates {
            (0.67,.8) 
            (.67,.67)
            (.8,.8)
            (0.67,.8)};
    \addplot [thick,color=dgmcolor!40!white,fill=dgmcolor!40!white, 
                    fill opacity=0.45]coordinates {
            (0.8,.857) 
            (.8,.8)
            (.857,.857)
            (0.8,.857)};
    \addplot [thick,color=dgmcolor!40!white,fill=dgmcolor!40!white, 
                    fill opacity=0.45]coordinates {
            (0.857,.89) 
            (.857,.857)
            (.89,.89)
            (0.857,.89) };
    \addplot [thick,color=dgmcolor!40!white,fill=dgmcolor!40!white, 
                    fill opacity=0.45]coordinates {
            (0.89,.95) 
            (.89,.89)
            (.95,.95)
            (0.89,.95) };
    \node[mark=none] at (axis cs:.74,.76){\tiny{\textsf{2}}};
    \node[mark=none] at (axis cs:.25,.45){\textsf{2}};
    \end{axis}
    \end{tikzpicture}
    \hspace{1.5cm}
    \begin{tikzpicture}[scale=0.6]
    \begin{axis} [ 
    ticklabel style = {font=\Large},
    axis y line=middle, 
    axis x line=middle,
    ytick={0.5,0.67,0.95},
    yticklabels={$\tfrac{\pi}{2}$,$\tfrac{2\pi}{3}$,$\pi$},
    xtick={0.5,0.6,0.67,0.95},
    xticklabels={$\tfrac{\pi}{2}$,$\zeta_2$, $\tfrac{2\pi}{3}$, $\pi$},
    xmin=0, xmax=1.1,
    ymin=0, ymax=1.1,]
    \addplot [mark=none] coordinates {(0,0) (1,1)};
    \addplot [thick,color=dgmcolor!20!white,fill=dgmcolor!20!white, 
                    fill opacity=0.45]coordinates {
            (0,0.6)
            (0,0)
            (0.6,0.6)
            (0,0.6)};
    \addplot [thick,color=black!10!white,fill=black!10!white, 
                    fill opacity=0.4]coordinates {
            (0.6,0.95)
            (0.6,0.6)
            (0.95,0.95)
            (0.6,0.95)};         
    \node[mark=none] at (axis cs:.25,.45){\textsf{1}};
    \end{axis}
    \end{tikzpicture}
\caption{The persistent invariants $\invariant(\VR(\bbT^2))$ (left) and $\invariant(\VR(\bbS^1\vee\bbS^2\vee\bbS^1))\vert_{(0,\zeta_2)}$ (right), respectively, for $\invariant=\cupprod$ or $\cat$. Here, $\zeta_2=\arccos(-\tfrac{1}3)\approx 0.61\pi$.} 
\label{fig:torus_sphere_VR}
\end{figure}

For the persistent cup-length invariant of $\VR(\bbS^1\vee \bbS^2\vee \bbS^1)$, recall \citep[Prop.~3.7]{adamaszek2020homotopy}: the Vietoris-Rips complex of a metric gluing is the wedge sum of Vietoris-Rips complexes. Applying Prop.~\ref{prop:prod-coprod-pers-cup}, we have for any interval $[a,b]$,
\begin{align*}
    \cupprod\left(\VR(\bbS^1\vee \bbS^2\vee \bbS^1)\right)([a,b])
    =\max \left\{ \cupprod(\VR(\bbS^1))([a,b]),\cupprod(\VR(\bbS^2))([a,b])\right\} 
\end{align*}
We now compute $\cupprod(\VR(\bbS^2))$. 
For any $r\geq \pi=\diam(\bbS^2)$, $\VR_r(\bbS^2)$ is contractible. For any $r\in (0,\zeta_2)$, where $\zeta_2:=\arccos(-\tfrac{1}3)\approx 0.61\pi$, it follows from \citep[Thm.~10]{lim2020vietoris} that $\VR_r(\bbS^2)$ is homotopy equivalent to $\bbS^2$. Thus, $\cupprod(\VR(\bbS^2))([a,b])=1,\forall [a,b]\subset(0,\zeta_2)$, implying
\[\cupprod(\VR(\bbS^1\vee \bbS^2\vee \bbS^1))([a,b])=1,\forall [a,b]\subset(0,\zeta_2).\]
Again, it is not difficult to see that $\cat(\VR(\bbS^1\vee \bbS^2\vee \bbS^1))=\cupprod(\VR(\bbS^1\vee \bbS^2\vee \bbS^1))$ when restricted to the interval $(0,\zeta_2).$

Due to the current lack of knowledge about the homotopy types of $\VR_r(\bbS^2)$ for $r$ close to $\pi$, we are not able to completely characterize the  function $\cupprod(\VR(\bbS^2))$, nor $\cupprod(\VR(\bbS^1\vee \bbS^2\vee \bbS^1))$. However, despite this, we are still able to exactly evaluate the erosion distance of $\cupprod(\VR(\bbS^1\vee \bbS^2\vee \bbS^1))$ and $\cupprod(\VR(\bbT^2))$, as in Prop.~\ref{prop:erosion-comp}. 

\end{example}

\begin{restatable}{proposition}{torusvswedge} 
\label{prop:erosion-comp} Let $\invariant=\cupprod$ or $\cat$.
For the $2$-torus $\bbT^2$ and the wedge sum space $\bbS^1\vee \bbS^2\vee \bbS^1$, 
\[\tfrac{\pi}{3}=d_{\mathrm{E}}\left(\invariant(\VR(\bbT^2)),\invariant(\VR(\bbS^1\vee\bbS^2\vee\bbS^1))\right)\leq 2\cdot d_{\mathrm{GH}}\left(\bbT^2,\bbS^1\vee\bbS^2\vee\bbS^1\right).\]
\end{restatable}

\begin{proof}
\label{pf:erosion-comp}
For  simplicity of notation, we denote 
\[\bfJ_{\times}:=\invariant(\VR(\bbT^2))\text{ and }\bfJ_{\vee}:=\invariant(\VR(\bbS^1\vee\bbS^2\vee\bbS^1)).\] 
For an interval $I=[a,b]$ and $\epsilon>0$, we denote $I^{\epsilon}:=[a-\epsilon,b+\epsilon]$.

Suppose that $\bfJ_{\times}$ and $\bfJ_{\vee}$ are $\epsilon$-eroded, which means $\bfJ_{\times}(I)\geq \bfJ_{\vee}(I^\epsilon)$ and $\bfJ_{\vee}(I)\geq \bfJ_{\times}(I^\epsilon)$, for all $I\in\Int$. We take $I_0:=[\tfrac{\pi}{3}-\delta,\tfrac{\pi}{3}+\delta]$ for $\delta$ sufficiently small, so that the associated point of $I_0$ in the upper-diagonal half plane is very close to the point $(\tfrac{\pi}{3},\tfrac{\pi}{3})$. Then, we have 
\[\bfJ_{\vee}(I_0)= 1\leq \bfJ_{\times}(I_0^t)=2\text{, for any }t<\tfrac{\pi}{3}-2\delta.\]
Therefore, in order for the inequality $\bfJ_{\vee}(I_0)\geq \bfJ_{\times}(I_0^\epsilon)$ to hold, it must be true that $\epsilon\geq \tfrac{\pi}{3},$ implying that $d_{\mathrm{E}}(\bfJ_{\times},\bfJ_{\vee})\geq\tfrac{\pi}{3}.$

Next, we prove the inverse inequality $d_{\mathrm{E}}(\bfJ_{\times},\bfJ_{\vee})\leq\tfrac{\pi}{3}$. Fix an arbitrary $\epsilon>\tfrac{\pi}{3}$. We claim that $\bfJ_{\vee}(I^\epsilon)=0$ for all $I\in \Int$. As before, let $\zeta_2:=\arccos\left(-\tfrac{1}{3}\right)\leq\tfrac{2\pi}{3}$. Notice that the longest possible bar in the barcode for $\VR(\bbS^1\vee\bbS^2\vee\bbS^1)$ is $(0,\pi)$, and any bar $I'$ in the barcode, except for the bar $(0,\zeta_2)$, is a sub-interval of $[\zeta_2,\pi)$. Thus, the length of $I'$ is less than or equal to $(\pi-\zeta_2)<\zeta_2<2\epsilon$. For any $I\in \Int$, because the interval $I^\epsilon$ has length larger than $2\epsilon$, it cannot be contained in a bar from the barcode. Thus, $\bfJ_{\vee}(I^\epsilon)=0$. We can directly check that a similar claim holds for $\bfJ_{\times}$ as well, i.e. $\bfJ_{\times}(I^\epsilon)=0$ for all $I\in \Int$. Therefore, for any $I\in \Int$,
\[ \bfJ_{\times}(I)\geq \bfJ_{\vee}(I^\epsilon)=0\text{  and  }\bfJ_{\vee}(I)\geq \bfJ_{\times}(I^\epsilon)=0.\]
In other words, $\bfJ_\times$ and $\bfJ_\vee$ are $\epsilon$-eroded, for any $\epsilon>\tfrac{\pi}{3}$. Thus, $d_{\mathrm{E}}(\bfJ_{\times},\bfJ_{\vee})\leq\tfrac{\pi}{3}.$
\end{proof}

\begin{remark} \label{rmk:inter-torus-wedge}
Denote by $\bfH _*\left( \cdot\right)$ the persistent homology functor in all dimensions. Then,
\begin{itemize}
    \item $\bfH _*\left( \VR(\bbT^2) \right)\big\vert_{(0,\zeta_2)}\cong \bfH _*\left( \VR(\bbS^1\vee \bbS^2\vee \bbS^1)\right)\big\vert_{(0,\zeta_2)}$, and
    \item $\bfH _*\left( \VR(\bbT^2) \right)\big\vert_{(\pi,\infty)}= \bfH _*\left( \VR(\bbS^1\vee \bbS^2\vee \bbS^1)\right)\big\vert_{(\pi,\infty)}$ is trivial
\end{itemize}
Thus, the interleaving distance $\di$ between persistent homology of $\bbT^2$ and $\bbS^1\vee \bbS^2\vee \bbS^1$ in any dimension $p$ satisfies
\[ \di\left( \bfH _p\left( \VR(\bbT^2) \right), \bfH _p\left( \VR(\bbS^1\vee \bbS^2\vee \bbS^1)\right) \right)\leq \tfrac{\pi-\zeta_2}{2}<\tfrac{\pi}{3}.\]

By providing a better bound for the Gromov-Hausdorff distance than the one given by persistent homology, the persistent cup-length (or LS-category) invariant demonstrates its  strength in terms of discriminating spaces and capturing additional important topological information.
\end{remark}

\section{Persistent cup modules: \texorpdfstring{$\ell$}{ell}-fold product of persistent cohomology}\label{sec:cup module}

For the purpose of extracting more information from the cohomology ring structure, we study the $\ell$-fold product of the persistent cohomology algebra for any positive integer $\ell\geq 1$, and prove in Prop.~\ref{prop:dgm_p,l} that the persistent cup-length invariant can be retrieved from the persistence diagram of the $\ell$-fold products. Also, we establish the stability of $\ell$-fold products of persistent cohomology (see Thm.~\ref{thm:stability dE_dB_dGH}).

Given any graded algebra $( A,+,\bullet)$, let $( A^{\ell},+,\bullet)$ denote the graded subalgebra of $ A$ generated by the elements $\{a_1\bullet\cdots\bullet a_\ell:a_i\in  A\}$.
For a graded algebra morphism $f: A\to  B$, let $f^{\ell}: A^{\ell}\to  B^{\ell}$ be the morphism such that $a_1\bullet\cdots\bullet a_\ell\mapsto f(a_1)\bullet\cdots\bullet f(a_\ell),$ which is indeed the restriction $f\vert_{ A^{\ell}}$. Let $\Afunc$ be a persistent graded algebra, i.e. a functor from the poset category $(\bbR,\leq)$ to $\Alg^{\mathrm{op}}$ the opposite category graded algebras. 
We define the \emph{$\ell$-fold product functor} to be
$(\cdot)^{\ell}:\Alg^{\mathrm{op}}\to \Alg^{\mathrm{op}}$ with $ A\mapsto  A^{\ell}$ and $f\mapsto f^{\ell}.$ Notice that when $ A$ has the multiplication identity, $ A= A^{\ell}$ for any $\ell$. To obtain more interesting objects, we will consider the subalgebra $ A^{ + }$ of $ A$ which consists of positive-degree elements in $ A$, and study the $\ell$-fold product of $ A^{ + }$, instead. 
\label{para:graded algebra}

Given $p\in \N$, let $\deg_p(\cdot)$ be the degree-$p$ component of a graded vector space. For instance, $$\deg_p(\bfH^*( X))=\bfH^p( X)$$ is the $p$-th cohomology of a topological space $ X$. For any $\ell\geq 1$ and a graded algebra $A$, we will write $\deg_p\left(\left(A^{ + }\right)^\ell\right)$ to extract the degree-$p$ component of the $\ell$-fold product of $A^{ + }$.  

Let $\bfH^{ + }:\Top\to \Alg^{\mathrm{op}}$ denote the positive-degree cohomology algebra functor, i.e.  $$\bfH^{ + }:=\bigoplus_{p>0}\bfH^p.$$

Throughout this section, we will assume that persistent spaces  have $\mathrm{q}$-tame persistent (co)homology (see pg.~\pageref{para:q-tame}). Examples of such persistent spaces include Vietoris-Rips filtrations of totally bounded metric spaces, cf. \citep[Prop.~5.1]{chazal2014persistence}.

In Sec.~\ref{sec:flag cup module}, we first study the category of (graded) flags and a complete invariant for it, which we call the \emph{rank invariant} and denote as $\rank(\cdot)$, as well a, the M\"obius inversion of $\rank(\cdot)$.  In Sec.~\ref{sec:p.cup module}, we introduce the notion of \emph{persistent cup modules} $\cupmodule(\cdot)$ (see Defn. \ref{def:cup module}) and \emph{persistent $\ell$-cup modules} $\cupmodule^\ell(\cdot)$ (see Defn. \ref{def:l-cup module}), which are persistent graded flags and persistent graded vector spaces, respectively. Let $\de$ be the erosion distance between persistent invariants (see Definition \ref{def:de}), and $\db$ be the \emph{bottleneck distance} between barcodes (see \citep[Defn.~3.1]{cohen2007stability}). We establish the following stability result for persistent cup modules and persistent $\ell$-cup modules:

\cupmodulestab* 

\ellpbarcode*

We provide a concrete example in Sec.~\ref{sec:flag>cup} to compare the performance of different persistent invariants, including persistent homology, persistent cup-length invariant, persistent LS-category, persistent cup modules and persistent $\ell$-cup modules. In particular, we show that persistent ($\ell$)-cup modules sometimes have stronger distinguishing powers than other invariants.

In Sec.~\ref{sec:2d cup module}, we see that persistent cup modules also have the structure of 2-dimensional persistence modules. 

\subsection{Persistent cup module as a persistent graded flag}\label{sec:flag cup module}

In Sec.~\ref{sec:flag}, we recall the notion of flags of vector spaces over the base field $\field$ (see Defn.~\ref{def:flag}) and study the decomposition of flags. 
Let $\Flag$ denote the category of finite-depth flags, and let $\Flag_{\operatorname{fin}}$ be the full subcategory of $\Flag$ consisting of flags of finite-dimensional vector spaces which we will refer to as finite dimensional flags for simplicity.
Let $\mathfrak{J}(\Flag_{\operatorname{fin}})$ denote the commutative monoid (under the direct sum) of isomorphism classes of elements in $\Flag_{\operatorname{fin}}$, and let $\mathfrak{A}(\Flag_{\operatorname{fin}})$ be the \emph{Grothendieck group} of $\Flag_{\operatorname{fin}}$, defined as the group completion of $\mathfrak{J}(\Flag_{\operatorname{fin}})$.

\medskip
Below, we will establish:

\begin{proposition}
\label{prop:decomposition of Flag}
The category $\Flag_{\operatorname{fin}}$ is Krull-Schmidt, and in particular,
\[\mathfrak{J}(\Flag_{\operatorname{fin}})
=\{(m_1,\dots,m_n,0,\dots)\in \mathbb{N}^\infty: m_1\geq \cdots\geq m_n\geq 0\}.\]
\end{proposition}

Since $\Flag_{\operatorname{fin}}$ is additive and Krull-Schmidt, its Grothendieck group $\mathfrak{A}(\Flag_{\operatorname{fin}})$ is the free abelian group generated by the set of isomorphism classes of indecomposables, see \citep[page 10]{patel2018generalized}. By Prop.~\ref{prop:decomposition of Flag}, we have the following corollary:
\begin{corollary}\label{cor:grothendieck group of flag}
The Grothendieck group of $\Flag_{\operatorname{fin}}$ is $\mathfrak{A}(\Flag_{\operatorname{fin}})\cong \bbZ^\infty.$
\end{corollary}

In Sec.~\ref{sec:mobius of rank}, we define the rank invariant of a persistent flag; see Defn.~\ref{def:rank}, and study its generalized persistent diagram (cf. \citep[Defn.~7.1]{patel2018generalized}).

\subsubsection{The category of finite-dimensional flag}\label{sec:flag}

\begin{definition}[Flag and morphism of flags] 
\label{def:flag} 
A \textbf{flag} $\Vflag $ is a non-increasing filtration of vector spaces:
$$\Vflag :=V_1 \supseteq  V_2 \supseteq  \cdots \supseteq  V_\ell\supseteq  \cdots .$$
The flag $\Vflag $ is said to have \textbf{finite depth} if there exists $n$ such that $V_n=0$ and it is said to be \textbf{finite-dimensional} if $\dimension(V_1)<\infty$ (as a consequence, $\dimension(V_n)<\infty$ for all $n$).

A morphism $f:\Vflag \to \Wflag $ of flags is a linear map 
\begin{center}
    $f:V_1\to W_1$ such that $f(V_\ell)\subseteq W_\ell$ for any $\ell$.
\end{center}
The morphism $f$ is said to be \textbf{strict} if $f(V_\ell)=f(V_1)\cap W_\ell$.
\end{definition}

The category $\Flag_{\operatorname{fin}}$ is an additive category with zero object, kenerls, cokernels, images and coimages, but it is not abelian (see Rmk.~\ref{rmk:non-abelian}). 

\begin{remark}
\label{rmk:non-abelian}
Consider two filtrations on $\field$, $\field_\star^1:=\field\supseteq  0\supseteq  \cdots$ and $\field_\star^2:=\field\supseteq  \field\supseteq  0\supseteq \cdots$. The morphism $f:\field_\star^1\to \field_\star^2$ corresponding to the identity map $\id_\field$ on $\field$ has trivial kernel and cokernel (and thus is monic and epic), but is not an isomorphism. In an abelian category, if a morphism is monic and epic, then it is an isomorphism. Therefore, $\Flag_{\operatorname{fin}}$ is not abelian.
\end{remark}

Below, we introduce a complete invariant for objects in $\Flag_{\operatorname{fin}}$. See Prop \ref{prop:dim_complete} for the proof of its completeness.

\begin{definition}[Dimension of a flag] \label{def:dim}
For any finite dimensional flag $\Vflag $, we define the \textbf{dimension of $\Vflag $}, denoted as $\fdim(\Vflag)$, to be the non-increasing sequences of integers $(m_1,m_2,\dots)\in 
\N^\infty$, where $m_\ell:=\dimension(V_\ell)$ for each $l\in\N^+.$
\end{definition}

For any $n\geq 1$, we define a finite-depth flag $\field_\star^n$ such that $\field_\ell^n$ is $\field$ for $1\leq l\leq n$ and $0$ for $l>n$. Notice that the dimension of $\field_\star^n$ is
\[(\underbrace{1,\dots,1}_{n},0,\dots)\in \mathbb{N}^\infty.\]

\begin{proof}[Proof of Prop.~\ref{prop:decomposition of Flag}]
It suffices to prove that each $\field_\star^n$ is decomposable, and every finite-depth flag decomposes into a finite direct sum of isomorphism classes of elements in $B:=\{\field_\star^n\}_{n\geq 1}.$ 

We first prove that each $\field_\star^n$ is indecomposable. Suppose $\field_\star^n=\Vflag \oplus \Wflag $. Then $V_1\oplus W_1=\field$ implies one of $V$ and $W$ must be $0$. Thus, $\field_\star^n$ cannot be decomposed into a direct sum of non-zero objects in $\Flag_{\operatorname{fin}}$. 

Let $\Vflag $ be a finite filtration such that $V_\ell\neq 0$ iff $l\leq k$ for some integer $k$. Let $m_\ell:=\operatorname{dim}(V_\ell)$ for any $\ell$. Then 
\begin{align*}
    \Vflag &\cong(\field^{m_1}, \field^{m_1}\supseteq \cdots\supseteq  \field^{m_k}\supseteq  0\supseteq \cdots)\\
    &= (\field^{m_1-1}, \field^{m_1-1}\supseteq \cdots\supseteq  \field^{m_k-1}\supseteq  0\supseteq \cdots)\oplus \field_\star^n,
\end{align*}
where $n$ is the smallest $\ell$ s.t. $m_\ell=1.$ Repeat this process for finitely many steps to decompose $\Vflag $ into a direct sum of elements in $B.$ 
\end{proof}

\subsubsection{The rank invariant of a persistent flag and its generalized persistence diagram}\label{sec:mobius of rank}

The dimension of flags induces a categorical invariant (see Defn.~\ref{def:categorical invariant}):

\begin{definition}[Rank of a flag morphism]
For any flag morphism  $f:\Vflag \to \Wflag $  of flags such that $\image(f)$ is finite dimensional, we define the \textbf{rank of $f$} as $$\rank(f):=
\fdim(\image(f)).$$
\end{definition}

We say that a persistent flag $\pflag$ is \emph{$\mathrm{q}$-tame} if for every interval $I$, $\image(\pflag(I))$ is finite dimensional. Note that if a persistent space $\Xfunc$ has $\mathrm{q}$-tame persistent (co)homology, then its persistent cup-module $\Phi(\Xfunc)$ is a $\mathrm{q}$-tame persistent flag and so is the persistent $\ell$-cup module $\Phi^\ell(\Xfunc)$ for any $\ell.$

\begin{definition} \label{def:rank}
The \textbf{rank invariant} of a $\mathrm{q}$-tame persistent flag $\pflag$ is defined as the functor $\rank(\pflag):(\Int,\subseteq)\to (\N^\infty,\leq)^{\mathrm{op}}$, given by
\[\rank(\pflag) (I) :=
\fdim(\image(\pflag(I))).
\]
\end{definition}

In other words, the rank invariant is a persistent invariant because it arises from $\dimension$ which is categorical invariant, cf.  Defn.~\ref{def:p_invariant}.

Suppose $\pflag$ is $S=\{s_1<\cdots<s_m\}$-constructible (see pg.~\pageref{para:constructibility}). We define the \textbf{persistent $\rank$-diagram} $\dgm\left(\rank\left(\pflag\right)\right):\Int\to \mathfrak{A}(\Flag_{\operatorname{fin}})\cong \bbZ^\infty$ of $\pflag$ to be the persistent $\rank$-diagram of the persistent rank invariant $\rank(\pflag)$ associated to $\pflag$ (see Defn.~\ref{def:I-dgm}). 
In other words, $\dgm\left(\rank\left(\pflag\right)\right)$ is the M\"obius inversion of $\rank(\pflag)$, where for every $s_i\leq s_j$, 
\begin{align*}
    \dgm\left(\rank\left(\pflag\right)\right)(\linterval  s_i,s_j\rinterval ) &:= \rank(\pflag) (\linterval  s_i,s_j\rinterval ) - \rank(\pflag) (\linterval  s_{i-1},s_j\rinterval ) \\
    &- \rank(\pflag) (\linterval  s_i,s_{j+1}\rinterval ) + \rank(\pflag) (\linterval  s_{i-1},s_{j+1}\rinterval ),
\end{align*}
$    \dgm\left(\rank\left(\pflag\right)\right)(\linterval  s_i,\infty\rinterval ) := \rank(\pflag) (\linterval  s_i,\infty\rinterval ) - \rank(\pflag) (\linterval  s_{i-1},\infty\rinterval )$, and $\dgm\left(\rank\left(\pflag\right)\right)(I):=0$ otherwise.

For each $l\in\N^+$, denote by $V_{\ell,\bullet}$ the persistence module such that $t\mapsto V_{\ell,t}$ the $\ell$-th vector space in the flag $V_{\star,t}$, and $(t\leq s)\mapsto (V_{\ell,t}\leftarrow V_{\ell,s}).$

\begin{proposition}[M\"obius inversion is depth-wise]
\label{prop:depth-wise}
The rank invariant of a finitely constructible persistent flag $\pflag$ and its M\"obius inversion are both depth-wise. In other words, $$\rank(\pflag)=\left(\rank(V_{1,\bullet}), \rank(V_{2,\bullet}),\dots\right) :\Int\to \left(\N^\infty,\leq\right)^{\mathrm{op}},$$
and similarly for $\dgm\left(\rank\left(\pflag\right)\right).$ Namely:
$$\dgm\left(\rank\left(\pflag\right)\right)=\left(\dgm(V_{1,\bullet}), \dgm(V_{2,\bullet}),\dots\right) :\Int\to \bbZ^\infty.$$
\end{proposition}

\begin{proof}  
We first show that the rank invariant $\rank(\pflag)$ is depth-wise: for any $I \in \Int$, we have
\begin{align*}
    \rank(\pflag) (I ) 
    &= \fdim\left(\image(\pflag(I ))\right)\\
    &= \big(\dimension (\image(V_{1,\bullet}(I ))),\dimension(\image(V_{2,\bullet}(I ))),\dots\big)\\
    &= \big(\rank(V_{1,\bullet})(I ),\rank(V_{2,\bullet})(I ),\dots\big).
\end{align*}

By the above and the definition of $\dgm\left(\rank\left(\pflag\right)\right)$, $\dgm\left(\rank\left(\pflag\right)\right)$ is depth-wise.
\end{proof}

Notice that the values of the rank invariant are monotonic: they are non-increasing sequences of non-negative integers. However, the M\"obius inversion does not preserve such monotonicity, see Ex.~\ref{ex:monotonicity lost} below.

\begin{example}[M\"obius inversion does not preserve the monotonicity of ranks] 
\label{ex:monotonicity lost}
As in Rmk.~\ref{rmk:non-abelian}, consider the morphism $f:\field_\star^1\to \field_\star^2$ corresponding to the identity map $\id_\field$ on $\field$. Define a persistent flag $\pflag:(\mathbb{R}_{\geq 0},\leq)\to \Flag_{\operatorname{fin}}$ such that
\[V_{\star,t}=\begin{cases}
\field_\star^1,\mbox{ for $t\in [0,1)$,}\\
\field_\star^2,\mbox{ for $t\in [1,2)$,}\\
0,\mbox{ for $t\geq 2$,}\\
\end{cases}
\text{ and }
(V_{\ell,t}\leftarrow V_{\ell,s})=\begin{cases}
\id_{\field_\star^1},\mbox{ for $[t,s]\subseteq [0,1)$,}\\
\id_{\field_\star^2},\mbox{ for $[t,s]\subseteq [1,2)$,}\\
f,\mbox{ for $t<1\leq s<2$,}\\
0,\mbox{ otherwise.}\\
\end{cases}
\] Then,  
\begin{align*}
\dgm\left(\rank\left(\pflag\right)\right)([1,1])&=\rank(\pflag)([1,1])-\rank(\pflag)([0,1])-\rank(\pflag)([1,2])+\rank(\pflag)([0,2])\\
&=(1,1,0,\ldots) - (1,0,\ldots) - (0,0,\ldots) +(0,0,\ldots) = (0,1,0,\ldots). 
\end{align*} 
This shows that the M\"obius inversion does not preserve the monotonicity property of ranks.
\end{example}

\noindent\textbf{Graded flags.} All the above results for flags can be generalized to graded flags. \label{para:graded flag}
A \textbf{graded flag} is a degree-wise non-increasing filtration of graded vector spaces:
$$\bigoplus_{p\geq 0} V^{p}_1 \supseteq  \bigoplus_{p\geq 0} V^{p}_2 \supseteq  \cdots \supseteq  \bigoplus_{p\geq 0} V^{p}_\ell\supseteq  \cdots,$$
which will be denoted by $V^{\circ}_\star$. Here $\circ$ and $\star$ represent the parameter for degree and depth (in flags), respectively.
Let $\GVect$ be the category of graded vector spaces.
As before, we can see that the category $\grFlag_{\operatorname{fin}}$ of finite-depth graded flags of finite-dimensional vector spaces is an additive category with zero object, kenerls, cokernels, images and coimages, but it is not abelian.

We define the \textbf{dimension of a graded flag $V^{\circ}_\star$}, denoted as $\fdim(V^{\circ}_\star)$, to be a matrix such that each row of it is the dimension of the flag in the corresponding degree. As before, we can check that the dimension is a complete invariant for finite-depth graded flags, and every graded flag can be uniquely decomposed into the direct sum of indecomposables whose dimension is of the following form:  
\[
\begin{pmatrix}
0 & \cdots & 0 & 0 & \cdots \\
\cdots  & \cdots  & \cdots  & \cdots & \cdots \\
0 & \cdots & 0 & 0 & \cdots \\
1 & \cdots & 1 & 0 & \cdots \\
0 & \cdots & 0 & 0 & \cdots \\
\cdots  & \cdots  & \cdots  & \cdots & \cdots \\
0 & \cdots & 0 & 0 & \cdots
\end{pmatrix}\in \mathbb{N}^{\infty,\infty},
\]
where for some $p\geq 0$ and $n_p\geq 0$, the first $n_p$ entries in the $p$-th row are $1$ and all other entries are $0$.

For a persistent graded flag $V^\circ_{\star,\bullet}$, we define its \textbf{rank invariant} of a persistent flag $\pflag$ to be the functor $\rank\left(V^\circ_{\star,\bullet}\right):(\Int,\subseteq)\to (\N^\infty,\leq)^{\mathrm{op}}$, given by 
\[\rank\left(V^\circ_{\star,\bullet}\right) (I) :=
\fdim\left(\image \left(V^\circ_{\star,\bullet}(I)\right)\right).
\]
It is straightforward to check that  the rank invariant of persistent graded flags and its M\"obius inversion are both given degree-wise, i.e. \label{para:deg-wise}
\begin{equation}\label{eq:rank of g.flag}
\rank\left(V^\circ_{\star,\bullet}\right)=
\left(\rank\left(V_{\ell,\bullet}^p\right)\right)_{p\in\N^+,\ell\in \N^+} :\Int\to \left(\N^{\infty,\infty},\leq\right)^{\mathrm{op}},
\end{equation}
and
\begin{equation}\label{eq:dgm of g.flag}
\dgm\left(\rank\left(V^\circ_{\star,\bullet}\right)\right)=\left(\dgm\left(V_{\ell,\bullet}^p\right)\right)_{p\in\N^+,\ell\in \N^+} :\Int\to \bbZ^{\infty,\infty}.
\end{equation}

Eqn.~(\ref{eq:dgm of g.flag}) suggests that the persistence diagram of a persistent graded flag $V^\circ_{\star,\bullet}$ can be obtained by stacking the standard persistence diagrams of all $V_{\ell,\bullet}^p$, for all $\ell$ and $p$; see Fig.~\ref{fig:dgm of p. cup module} for an example.

\subsubsection{Persistent cup module as a persistent graded flag and its stability}\label{sec:p.cup module}

For a topological space $ X$ and any $\ell\geq 1$, let $(\bfH^{ + }( X))^{\ell}$ be the $\ell$-fold product of the graded algebra $(\bfH^{ + }( X),+,\smile)$. Then the following non-increasing sequence of spaces forms a graded flag:  
\[\cupmodule( X):\, \bfH^{ + }( X)\supseteq  (\bfH^{ + }( X))^{ 2}\supseteq  (\bfH^{ + }( X))^{3}\supseteq \cdots.\]
Any continuous map $f: X\to  Y$ induces a map from $\bfH^+(f):\bfH^+( Y)\to \bfH^+( X)$ that preserves the cup product operation. Therefore, for any $\ell$, $\bfH^+(f)\vert_{(\bfH^{ + }( Y))^{\ell}}$ is a map from $(\bfH^{ + }( Y))^{\ell}$ to $(\bfH^{ + }( X))^{\ell}$. In addition, induces $\bfH^+(f)$ induces a (graded) flag morphism from $\cupmodule( Y)\to \cupmodule( X)$; see Defn. \ref{def:flag}.

\begin{definition}\label{def:cup module}
For a persistent space $\Xfunc$, we define the \textbf{persistent cup module} of $\Xfunc$ to be the persistent graded flag 
\begin{center}
$\cupmodule(\Xfunc):(\bbR,\leq)\to \grFlag^{\mathrm{op}}$ with $t\mapsto \cupmodule( X_t)$,
\end{center}
and for any $t\leq t'$ the map $\cupmodule(X_{t'})\to \cupmodule(X_t)$ is induced by the map $\bfH^+( X_{t'}\to  X_t).$ In particular, we have the following commutative diagram:
\[
\begin{tikzcd}[column sep=small]
\cupmodule( X_t):
&
\bfH^{ + }( X_t)
\ar[r, phantom, sloped, "\supseteq "]
& (\bfH^{ + }( X_t))^{ 2} 
\ar[r, phantom, sloped, "\supseteq "]
& (\bfH^{ + }( X_t))^{ 3}
\ar[r, phantom, sloped, "\supseteq "]
& \cdots
\\
\cupmodule( X_{t'}):
\ar[u]
&
\bfH^{ + }( X_{t'})
\ar[u]
\ar[r, phantom, sloped, "\supseteq "]
& (\bfH^{ + }( X_{t'}))^{ 2} 
\ar[u]
\ar[r, phantom, sloped, "\supseteq "]
& (\bfH^{ + }( X_{t'}))^{ 3}
\ar[u]
\ar[r, phantom, sloped, "\supseteq "]
& \cdots
\ar[u]
\end{tikzcd}
\]
\end{definition}
 
\begin{definition} \label{def:l-cup module}
For a persistent space $\Xfunc$ and any $\ell\geq 1$, we define the \textbf{persistent $\ell$-cup module} of $\Xfunc$ to be the persistence graded vector space
\[\cupmodule^\ell(\Xfunc):(\bbR,\leq) \to \GVect^{\mathrm{op}}\text{ with }t\mapsto (\bfH^{ + }\left( X_t)\right)^{\ell}\]
\end{definition}

Let 
$$\barc\!\left(\deg_p\left(\cupmodule^\ell(\Xfunc)\right)\right)$$ be the barcode of the degree-$p$ component of the persistent $\ell$-cup module, and let $$\barc\!\left(\cupmodule^\ell(\Xfunc)\right):=\bigsqcup_{p\geq 1}\barc\!\left(\deg_p\left(\cupmodule^\ell(\Xfunc)\right)\right).$$ \label{para:l_cup_dgm_barcode}
 In particular, $\barc\!\left(\cupmodule^1(\Xfunc)\right)$ is the standard barcode of $\Xfunc$ in all positive degrees. 

We now establish the stability of persistent cup modules and persistent $\ell$-cup modules i.e. Thm.~\ref{thm:stability dE_dB_dGH} and Prop.~\ref{prop:stab-l-cup module}.

\noindent\textbf{Stability of persistent ($\ell$-)cup modules} In general, for any persistent graded algebra $\Afunc$ (see pg.~\ref{para:graded algebra}), there is a persistent graded flag structure similar to the one underlying persistent cup modules: for any $t\leq t'$,

\[
\begin{tikzcd}[column sep=small]
\cupmodule( A_t):
&
 A^{ + }_t
\ar[r, phantom, sloped, "\supseteq "]
& ( A^{ + }_{t})^{ 2} 
\ar[r, phantom, sloped, "\supseteq "]
& ( A^{ + }_{t})^{ 3}
\ar[r, phantom, sloped, "\supseteq "]
& \cdots
\\
\cupmodule( A_{t'}):
&
 A^{ + }_{t'}
\ar[u]
\ar[r, phantom, sloped, "\supseteq "]
& ( A^{ + }_{t'})^{ 2} 
\ar[u]
\ar[r, phantom, sloped, "\supseteq "]
& ( A^{ + }_{t'})^{ 3}
\ar[u]
\ar[r, phantom, sloped, "\supseteq "]
& \cdots
\ar[u]
\end{tikzcd}
\]

\begin{lemma} \label{lem:de-de}
For persistent graded algebras $\Afunc$ and $\Bfunc$, 
\[\de(\len(\Afunc), \len(\Bfunc))\leq \de\left(\rank(\cupmodule(\Afunc)),\rank(\cupmodule(\Bfunc))\right).\]
\end{lemma}
\begin{proof}
First notice that for any $t\leq s$, $\left(\image( A^{ + }_{s}\to  A^{ + }_{t})\right)^{\ell}\neq 0$ if and only if $\rank\left(\left( A^{ + }_{s}\right)^{\ell}\to \left( A^{ + }_{t}\right)^{\ell}\right)> 0$. 
Let $\epsilon>\de(\rank\left(\cupmodule(\Afunc)),\rank(\cupmodule(\Bfunc))\right)$. For any $[a,b]$, we have 
\[\rank(\cupmodule( A_b)\to \cupmodule( A_a))\geq \rank(\cupmodule( B_{b+\epsilon})\to  \cupmodule( B_{a-\epsilon})).\]
Thus, 
\begin{align*}
    \len\left( A^{ + }_b\to  A^{ + }_a\right)& = \sup \left\{\ell\mid \left(\image\left( A^{ + }_{b}\to  A^{ + }_{b}\right)\right)^{\ell}\neq 0\right\}\\
    & = \sup \left\{\ell\mid \rank\left(\left( A^{ + }_{b}\right)^{\ell}\to \left( A^{ + }_{a}\right)^{\ell}\right)> 0\right\}\\
    & \geq  \sup \left\{\ell\mid \rank\left(\left( B^{ + }_{b+\epsilon}\right)^{\ell}\to \left( B^{ + }_{a-\epsilon}\right)^{\ell}\right)> 0\right\}\\
    & = \len\left( B^{ + }_{b+\epsilon}\to  B^{ + }_{a-\epsilon}\right).
\end{align*}
By similar discussion, we have $    \len\left( B^{ + }_b\to  B^{ + }_a\right)\geq  \len\left( A^{ + }_{b+\epsilon}\to  A^{ + }_{a-\epsilon}\right).$ 
\end{proof}

\begin{proof}[Proof of Thm.~\ref{thm:stability dE_dB_dGH}] 
By Lem. \ref{lem:de-de}, we have the first inequality below:
\[d_{\mathrm{E}}(\cupprod(\Xfunc),\cupprod(\Yfunc))\leq \de\left( \rank( \cupmodule(\Xfunc)), \rank( \cupmodule(\Yfunc))\right)\leq \dhi(\Xfunc,\Yfunc).\]
To prove the second inequality, we apply Thm.~\ref{thm:stab-per-inv}. 
Recall from the proof of Cor. \ref{cor:stab-cup} that weak homotopy equivalence preserves cohomology algebras. For any maps $ X\xrightarrow{f} Y\xrightarrow{g}( Z)\xrightarrow{h}W$ where $g$ is a weak homotopy equivalence, we have a graded algebra isomorphism $\bfH^*(g):\bfH^*(Z) \to\bfH^*( Y) $. Thus, $\bfH^*(g)$ induces an graded flag isomorphism $ \cupmodule(g): \cupmodule( Z)\to \cupmodule(Y)$, implying that 
\[\rank( \cupmodule(g\circ f) )= \rank( \cupmodule( f)) \text{  and  }\rank( \cupmodule( h\circ g)) = \rank( \cupmodule( h)).\]
Therefore, the second inequality follows immediately from Thm.~\ref{thm:stab-per-inv}.

For the case of Vietoris-Rips filtrations of metric spaces, the statement follows from Prop.~\ref{prop:dh-dgh}.
\end{proof}

\begin{proof} [Proof of Prop.~\ref{prop:stab-l-cup module}]
By considering the following functors for any $p,\ell\geq 1$,
\begin{center}
    \begin{tikzcd} [column sep = small, row sep = small]
    \Alg \ar[r]
    &
    \grFlag\ar[r]
    &
    \GVect\ar[r]
    &
    \Vect
    \\
    A \ar[r, mapsto]
    &
    \cupmodule(\Afunc) \ar[r, mapsto]
    &
    (A^{ + })^{\ell} \ar[r, mapsto]
    &
    \deg_p\left((A^{ + })^{\ell}\right),
    \end{tikzcd}
\end{center}
we obtain that 
\begin{align}
&\max_{p,\ell\geq 1}\,
\db\left( \barc\!\left(\deg_p\left(\cupmodule^\ell(\Xfunc)\right)\right),\barc\!\left(\deg_p(\cupmodule^\ell(\Yfunc))\right)\right) \nonumber\\
\leq &
\max_{p,\ell\geq 1}\,
\di^{\Vect } \left(\deg_p\big((\bfH^{ + }(\Xfunc))^{\ell}\big),\deg_p\big((\bfH^{ + }(\Yfunc))^{\ell}\big) \right) \label{line:alg stab} \\
= & 
\max_{\ell\geq 1}\,
\di^{\GVect} \left( (\bfH^{ + }(\Xfunc))^{\ell},(\bfH^{ + }(\Yfunc))^{\ell}\right) \label{line:functor1}\\
\leq &
\di^{\grFlag} \left(  \cupmodule(\Xfunc), \cupmodule(\Yfunc)\right)\\
\leq &
\di^{\Alg } \left(\bfH^*(\Xfunc) ,\bfH^*(\Yfunc)  \right) \label{line:functor2}\\
\leq &
\dhi \left(\Xfunc,\Yfunc \right), \label{line:dhi}
\end{align}
where Eqn.~(\ref{line:alg stab}) follows from the stability of barcodes \citep[Thm.~4.4]{chazal2009proximity},
Eqn.~(\ref{line:functor1}) - (\ref{line:functor2}) follows from \citep[Prop.~3.6]{bubenik2014categorification}, and Eqn.~(\ref{line:dhi}) follows from \citep[Rmk.~105]{ginot2019multiplicative}.
\end{proof}

\begin{remark} 
\label{rmk:Phi-stability}
One can also apply \citep[Thm. 3.11]{puuska2017erosion} (for the first inequality below) to see that 
\[ 
\de\left( \rank( \cupmodule(\Xfunc)), \rank( \cupmodule(\Yfunc))\right)\leq \di^{\grFlag} \left(  \cupmodule(\Xfunc), \cupmodule(\Yfunc)\right)\leq \dhi(\Xfunc,\Yfunc).\]
\end{remark}

\subsection{Comparative performance of the different invariants}\label{sec:flag>cup}

In this subsection, we use the spaces $\bbT^2\vee\bbS^3$ v.s. $(\bbS^1\times\bbS^2)\vee\bbS^1$ to compare the distinguishing powers of persistent homology, persistent cup-length invariant, persistent LS-category, persistent cup modules and persistent $\ell$-cup modules. In particular, we see that in this example the latter two invariants provide better approximation of the Gromov-Hausdorff distance than other invariants.  

As before, let $\bbT^2=\bbS^1\times \bbS^1$ be the $\ell_\infty$-product of two unit geodesic circles, and similarly let $\bbS^1\times \bbS^2$ be equipped the $\ell_\infty$-product metric. Let $\bbT^2\vee\bbS^3$ and $(\bbS^1\times\bbS^2)\vee\bbS^1$ both be equipped with the gluing metric (see page \pageref{para:gluing}).

By \citep[Thm.~10]{lim2020vietoris}, for any $n>0$, $\VR_r(\bbS^n)$ is homotopy equivalent to $\bbS^n$ for all $r\in \left(0,\zeta_n\right)$ where $\zeta_n:=\arccos(\tfrac{-1}{n+1})$. Note that $\zeta_2=\arccos(-\tfrac{1}3)\approx 0.61\pi$ and $\zeta_3=\arccos(-\tfrac{1}{4})\approx 0.58\pi.$

\medskip\noindent\textbf{Persistent cup-length invariants and persistent LS-category invariants.}
Let $\invariant=\cupprod$ or $\cat$.
For any interval $[a,b]\subset (0,\zeta_3)$, we have
\[\invariant\left(\VR\left(\bbT^2\vee\bbS^3\right)\right)([a,b]) = \invariant\left(\bbT^2\vee\bbS^3\right)=2.\]
For any interval $[a,b]\subset (0,\zeta_2)$, we have
\[\invariant\left(\VR\left((\bbS^1\times\bbS^2)\vee\bbS^1\right)\right)([a,b]) = \invariant\left((\bbS^1\times\bbS^2)\vee\bbS^1\right)=2.\]
Therefore, we have the persistent cup-length (or LS-category) invariants of $\bbT^1\vee\bbS^2$ and $(\bbS^1\times\bbS^2)\vee\bbS^1$ as in Fig.~\ref{fig:torus_sphere_VR_flag}. And notice that 
\begin{equation} \label{eq:dE-cup-cat}
    d_{\mathrm{E}}\left(\invariant(\VR\left(\bbT^2\vee\bbS^3\right)),\invariant(\VR\left((\bbS^1\times\bbS^2)\vee\bbS^1\right))\right)\leq 
    \tfrac{\pi-\zeta_3}{2}\approx 0.21\pi
\end{equation}

\noindent\textbf{Persistence diagrams and barcodes of persistent $\ell$-cup modules.}
Recall that $\Phi^1(X) = \bfH^+(\VR(X))$ for a metric space $X$. 
Applying \citep[Cor.~9.2, Prop.~9.4, Rmk.~9.14 \& Rmk.~9.17]{lim2020vietoris}, we obtain:
\begin{itemize}
    \item $(0,\frac{2\pi}{3}]\in \barc\left(\deg_1(\Phi^1(\bbT^2))\right)$ and there are two copies of it in the barcode;
    \item $(0,\frac{2\pi}{3}]\in \barc\left(\deg_2(\Phi^1(\bbT^2))\right)$;
    \item $(0,\frac{2\pi}{3}]\in \barc\left(\deg_1(\Phi^1((\bbS^1\times\bbS^2))\right)$;
    \item $(0,\zeta_2]\in \barc\left(\deg_2(\Phi^1((\bbS^1\times\bbS^2))\right)$.
\end{itemize}
And the above bars are the only ones that start at $0$ in the corresponding persistence modules.

Applying \citep[Rmk.~6.2]{lim2020vietoris} and \citep[Prop.~3.7]{adamaszek2020homotopy}, we obtain the barcodes of the $\bbT^2\vee\bbS^3$ and $\VR\left((\bbS^1\times\bbS^2)\vee\bbS^1\right)$ in degree $1,2$ and $3$; see Fig.~\ref{fig:ell barcodes}.
   
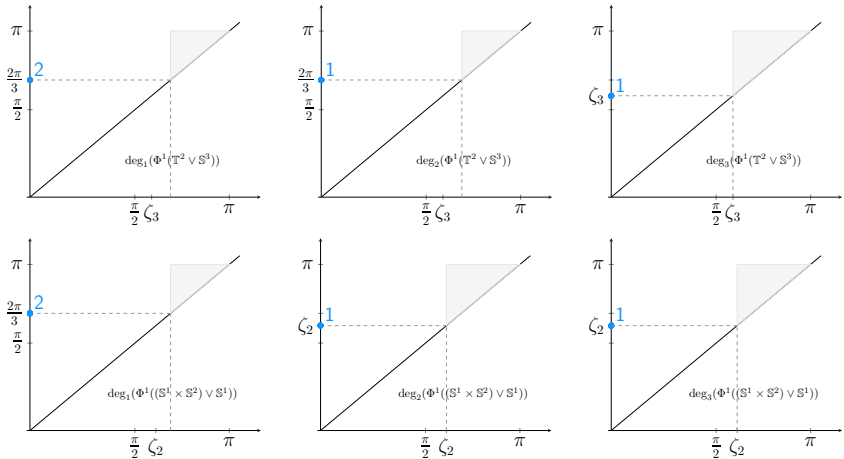
\begin{figure}
    \centering
\begin{tabular}{ c c c  }
\begin{tikzpicture}[scale=0.45]
    \begin{axis} [ 
    ticklabel style = {font=\Large},
    axis y line=middle, 
    axis x line=middle,
    ytick={0.5,0.67,0.95},
    yticklabels={$\tfrac{\pi}{2}$,$\tfrac{2\pi}{3}$,$\pi$},
    xtick={0.5,0.58,0.95},
    xticklabels={$\tfrac{\pi}{2}$,$\zeta_3$, $\pi$},
    xmin=-0.015, xmax=1.1,
    ymin=0, ymax=1.1,]
    \addplot [mark=none] coordinates {(0,0) (1,1)};
    \addplot [thick,color=black!10!white,fill=black!10!white, 
                    fill opacity=0.4]coordinates {
            (0.67,0.95)
            (0.67,0.67)
            (0.95,0.95)
            (0.67,0.95)};
    \addplot [black!40!white,mark=none,dashed, thin] coordinates {(0,0.67) (0.67,0.67)};
    \addplot [black!40!white,mark=none,dashed, thin] coordinates {(0.67,0) (0.67,0.67)};
    \addplot[barccolor,mark=*] (0, 0.67) circle (2pt) node[above right,barccolor]{\Large\textsf{2}};
    \node[mark=none] at (axis cs:0.68,0.21){$\deg_1(\Phi^1(\bbT^2\vee\bbS^3))$};
       \end{axis}
    \end{tikzpicture}
 
&
    \begin{tikzpicture}[scale=0.45]
    \begin{axis} [ 
    ticklabel style = {font=\Large},
    axis y line=middle, 
    axis x line=middle,
    ytick={0.5,0.67,0.95},
    yticklabels={$\tfrac{\pi}{2}$,$\tfrac{2\pi}{3}$,$\pi$},
    xtick={0.5,0.58,0.95},
    xticklabels={$\tfrac{\pi}{2}$,$\zeta_3$, $\pi$},
    xmin=-0.015, xmax=1.1,
    ymin=0, ymax=1.1,]
    \addplot [mark=none] coordinates {(0,0) (1,1)};
    \addplot [thick,color=black!10!white,fill=black!10!white, 
                    fill opacity=0.4]coordinates {
            (0.67,0.95)
            (0.67,0.67)
            (0.95,0.95)
            (0.67,0.95)};
    \addplot [black!40!white,mark=none,dashed, thin] coordinates {(0,0.67) (0.67,0.67)};
    \addplot [black!40!white,mark=none,dashed, thin] coordinates {(0.67,0) (0.67,0.67)};
    \addplot[barccolor,mark=*] (0, 0.67) circle (2pt) node[above right,barccolor]{\Large\textsf{1}};
   \node[mark=none] at (axis cs:0.68,0.21){$\deg_2(\Phi^1(\bbT^2\vee\bbS^3))$};
       \end{axis}
    \end{tikzpicture}
&
    \begin{tikzpicture}[scale=0.45]
    \begin{axis} [ 
    ticklabel style = {font=\Large},
    axis y line=middle, 
    axis x line=middle,
    xtick={0.5,0.58,0.95},
    xticklabels={$\tfrac{\pi}{2}$,$\zeta_3$,$\pi$},
    ytick={0.5,0.58,.67,0.95},
    yticklabels={,$\zeta_3$, , $\pi$},
    xmin=-0.015, xmax=1.1,
    ymin=0, ymax=1.1,]
    \addplot [mark=none] coordinates {(0,0) (1,1)};
    \addplot [thick,color=black!10!white,fill=black!10!white, 
                    fill opacity=0.4]coordinates {
            (0.58,0.95)
            (0.58,0.58)
            (0.95,0.95)
            (0.58,0.95)};
    \addplot [black!40!white,mark=none,dashed, thin] coordinates {(0,0.58) (0.58,0.58)};
    \addplot [black!40!white,mark=none,dashed, thin] coordinates {(0.58,0) (0.58,0.58)};
    \addplot[barccolor,mark=*] (0, 0.58) circle (2pt) node[above right,barccolor]{\Large\textsf{1}};
    \node[mark=none] at (axis cs:0.68,0.21){$\deg_3(\Phi^1(\bbT^2\vee\bbS^3))$};
       \end{axis}
    \end{tikzpicture}
\\
\begin{tikzpicture}[scale=0.45]
    \begin{axis} [ 
    ticklabel style = {font=\Large},
    axis y line=middle, 
    axis x line=middle,
    ytick={0.5,0.67,0.95},
    yticklabels={$\tfrac{\pi}{2}$,$\tfrac{2\pi}{3}$,$\pi$},
    xtick={0.5,0.6,0.95},
    xticklabels={$\tfrac{\pi}{2}$,$\zeta_2$, $\pi$},
    xmin=-0.015, xmax=1.1,
    ymin=0, ymax=1.1,]
    \addplot [mark=none] coordinates {(0,0) (1,1)};
    \addplot [thick,color=black!10!white,fill=black!10!white, 
                    fill opacity=0.4]coordinates {
            (0.67,0.95)
            (0.67,0.67)
            (0.95,0.95)
            (0.67,0.95)};
    \addplot [black!40!white,mark=none,dashed, thin] coordinates {(0,0.67) (0.67,0.67)};
    \addplot [black!40!white,mark=none,dashed, thin] coordinates {(0.67,0) (0.67,0.67)};
    \addplot[barccolor,mark=*] (0, 0.67) circle (2pt) node[above right,barccolor]{\Large\textsf{2}};
    \node[mark=none] at (axis cs:0.68,0.21){$\deg_1(\Phi^1((\bbS^1\times\bbS^2)\vee\bbS^1))$};
       \end{axis}
    \end{tikzpicture}
&
    \begin{tikzpicture}[scale=0.45]
    \begin{axis} [ 
    ticklabel style = {font=\Large},
    axis y line=middle, 
    axis x line=middle,
    ytick={0.5,0.6,0.67,0.95},
    yticklabels={,$\zeta_2$,,$\pi$},
    xtick={0.5,0.6,0.95},
    xticklabels={$\tfrac{\pi}{2}$,$\zeta_2$, $\pi$},
    xmin=-0.015, xmax=1.1,
    ymin=0, ymax=1.1,]
    \addplot [mark=none] coordinates {(0,0) (1,1)};
    \addplot [thick,color=black!10!white,fill=black!10!white, 
                    fill opacity=0.4]coordinates {
            (0.6,0.95)
            (0.6,0.6)
            (0.95,0.95)
            (0.6,0.95)};
    \addplot [black!40!white,mark=none,dashed, thin] coordinates {(0,0.6) (0.6,0.6)};
    \addplot [black!40!white,mark=none,dashed, thin] coordinates {(0.6,0) (0.6,0.6)};
    \addplot[barccolor,mark=*] (0, 0.6) circle (2pt) node[above right,barccolor]{\Large\textsf{1}};
    \node[mark=none] at (axis cs:0.68,0.21){$\deg_2(\Phi^1((\bbS^1\times\bbS^2)\vee\bbS^1))$};
    \end{axis}
    \end{tikzpicture}
&
    \begin{tikzpicture}[scale=0.45]
    \begin{axis} [ 
    ticklabel style = {font=\Large},
    axis y line=middle, 
    axis x line=middle,
    ytick={0.5,0.6,0.67,0.95},
    yticklabels={,$\zeta_2$,,$\pi$},
    xtick={0.5,0.6,0.95},
    xticklabels={$\tfrac{\pi}{2}$,$\zeta_2$, $\pi$},
    xmin=-0.015, xmax=1.1,
    ymin=0, ymax=1.1,]
    \addplot [mark=none] coordinates {(0,0) (1,1)};
    \addplot [thick,color=black!10!white,fill=black!10!white, 
                    fill opacity=0.4]coordinates {
            (0.6,0.95)
            (0.6,0.6)
            (0.95,0.95)
            (0.6,0.95)};
    \addplot [black!40!white,mark=none,dashed, thin] coordinates {(0,0.6) (0.6,0.6)};
    \addplot [black!40!white,mark=none,dashed, thin] coordinates {(0.6,0) (0.6,0.6)};
    \addplot[barccolor,mark=*] (0, 0.6) circle (2pt) node[above right,barccolor]{\Large\textsf{1}};
    \node[mark=none] at (axis cs:0.68,0.21){$\deg_3(\Phi^1((\bbS^1\times\bbS^2)\vee\bbS^1))$};
    \end{axis}
    \end{tikzpicture}
\end{tabular}
    \caption{\emph{Top row:} persistence diagram of persistent cohomology of $\bbT^2\vee\bbS^3$ in degree $1,2$ and $3$. \emph{Bottom row:} persistence diagram of persistent cohomology of $(\bbS^1\times\bbS^2)\vee\bbS^1$ in degree $1,2$ and $3$. In each figure, the gray part is the only region (beside the blue dots) that could attain non-zero value in the persistence diagram. }
    \label{fig:ell barcodes}
\end{figure} 

Each of the persistence diagram in Fig.~\ref{fig:ell barcodes} contains an undetermined region, so we cannot always get the precise value of the bottleneck distance between the barcodes associated to each vertical two diagrams. Instead, we can easily estimate them:
\begin{itemize}
    \item $\db\left( \barc\!\left(\deg_1\left(\cupmodule^1(\bbT^2\vee\bbS^3)\right)\right),\barc\!\left(\deg_1\left(\cupmodule^1((\bbS^1\times\bbS^2)\vee\bbS^1)\right)\right)\right)\leq \frac{1}{2}\left(\pi-\frac{2\pi}{3}\right)=
    \frac{\pi}{6}$.
    \item $\db\left( \barc\!\left(\deg_2\left(\cupmodule^1(\bbT^2\vee\bbS^3)\right)\right),\barc\!\left(\deg_2\left(\cupmodule^1((\bbS^1\times\bbS^2)\vee\bbS^1)\right)\right)\right)\leq \frac{1}{2}\left(\pi-\zeta_2\right)\approx 0.2\pi$.
    \item $\db\left( \barc\!\left(\deg_3\left(\cupmodule^1(\bbT^2\vee\bbS^3)\right)\right),\barc\!\left(\deg_3\left(\cupmodule^1((\bbS^1\times\bbS^2)\vee\bbS^1)\right)\right)\right)\leq \frac{1}{2}\left(\pi-\zeta_2\right)\approx 0.21\pi$.
\end{itemize}

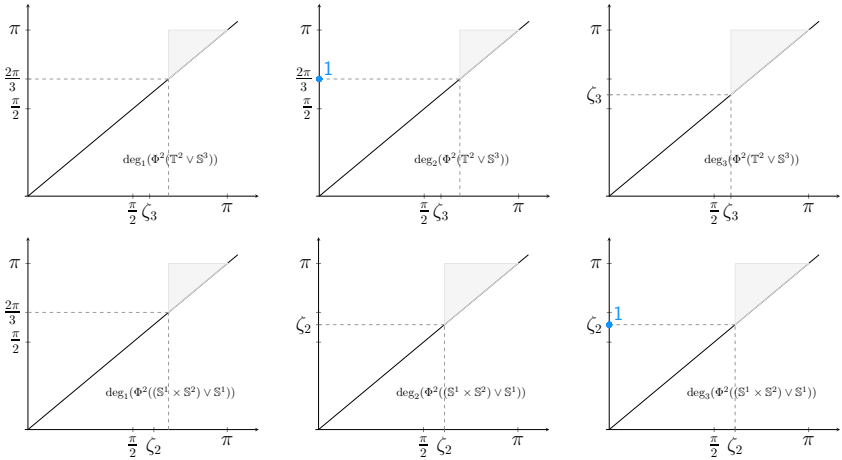
\begin{figure}[H]
    \centering
\begin{tabular}{ c c c  }
\begin{tikzpicture}[scale=0.45]
    \begin{axis} [ 
    ticklabel style = {font=\Large},
    axis y line=middle, 
    axis x line=middle,
    ytick={0.5,0.67,0.95},
    yticklabels={$\tfrac{\pi}{2}$,$\tfrac{2\pi}{3}$,$\pi$},
    xtick={0.5,0.58,0.95},
    xticklabels={$\tfrac{\pi}{2}$,$\zeta_3$, $\pi$},
    xmin=-0.015, xmax=1.1,
    ymin=0, ymax=1.1,]
    \addplot [mark=none] coordinates {(0,0) (1,1)};
    \addplot [thick,color=black!10!white,fill=black!10!white, 
                    fill opacity=0.4]coordinates {
            (0.67,0.95)
            (0.67,0.67)
            (0.95,0.95)
            (0.67,0.95)};
    \addplot [black!40!white,mark=none,dashed, thin] coordinates {(0,0.67) (0.67,0.67)};
    \addplot [black!40!white,mark=none,dashed, thin] coordinates {(0.67,0) (0.67,0.67)};
    \node[mark=none] at (axis cs:0.68,0.21){$\deg_1(\Phi^2(\bbT^2\vee\bbS^3))$};
    \end{axis}
    \end{tikzpicture}
&
    \begin{tikzpicture}[scale=0.45]
    \begin{axis} [ 
    ticklabel style = {font=\Large},
    axis y line=middle, 
    axis x line=middle,
    ytick={0.5,0.67,0.95},
    yticklabels={$\tfrac{\pi}{2}$,$\tfrac{2\pi}{3}$,$\pi$},
    xtick={0.5,0.58,0.95},
    xticklabels={$\tfrac{\pi}{2}$,$\zeta_3$, $\pi$},
    xmin=-0.015, xmax=1.1,
    ymin=0, ymax=1.1,]
    \addplot [mark=none] coordinates {(0,0) (1,1)};
    \addplot [thick,color=black!10!white,fill=black!10!white, 
                    fill opacity=0.4]coordinates {
            (0.67,0.95)
            (0.67,0.67)
            (0.95,0.95)
            (0.67,0.95)};
    \addplot [black!40!white,mark=none,dashed, thin] coordinates {(0,0.67) (0.67,0.67)};
    \addplot [black!40!white,mark=none,dashed, thin] coordinates {(0.67,0) (0.67,0.67)};
    \addplot[barccolor,mark=*] (0, 0.67) circle (2pt) node[above right,barccolor]{\Large\textsf{1}};
   \node[mark=none] at (axis cs:0.68,0.21){$\deg_2(\Phi^2(\bbT^2\vee\bbS^3))$};
    \end{axis}
    \end{tikzpicture}
&
    \begin{tikzpicture}[scale=0.45]
    \begin{axis} [ 
    ticklabel style = {font=\Large},
    axis y line=middle, 
    axis x line=middle,
    xtick={0.5,0.58,0.95},
    xticklabels={$\tfrac{\pi}{2}$,$\zeta_3$,$\pi$},
    ytick={0.5,0.58,.67,0.95},
    yticklabels={,$\zeta_3$, , $\pi$},
    xmin=-0.015, xmax=1.1,
    ymin=0, ymax=1.1,]
    \addplot [mark=none] coordinates {(0,0) (1,1)};
    \addplot [thick,color=black!10!white,fill=black!10!white, 
                    fill opacity=0.4]coordinates {
            (0.58,0.95)
            (0.58,0.58)
            (0.95,0.95)
            (0.58,0.95)};
    \addplot [black!40!white,mark=none,dashed, thin] coordinates {(0,0.58) (0.58,0.58)};
    \addplot [black!40!white,mark=none,dashed, thin] coordinates {(0.58,0) (0.58,0.58)};
    \node[mark=none] at (axis cs:0.68,0.21){$\deg_3(\Phi^2(\bbT^2\vee\bbS^3))$};
    \end{axis}
    \end{tikzpicture}
\\
\begin{tikzpicture}[scale=0.45]
    \begin{axis} [ 
    ticklabel style = {font=\Large},
    axis y line=middle, 
    axis x line=middle,
    ytick={0.5,0.67,0.95},
    yticklabels={$\tfrac{\pi}{2}$,$\tfrac{2\pi}{3}$,$\pi$},
    xtick={0.5,0.6,0.95},
    xticklabels={$\tfrac{\pi}{2}$,$\zeta_2$, $\pi$},
    xmin=-0.015, xmax=1.1,
    ymin=0, ymax=1.1,]
    \addplot [mark=none] coordinates {(0,0) (1,1)};
    \addplot [thick,color=black!10!white,fill=black!10!white, 
                    fill opacity=0.4]coordinates {
            (0.67,0.95)
            (0.67,0.67)
            (0.95,0.95)
            (0.67,0.95)};
    \addplot [black!40!white,mark=none,dashed, thin] coordinates {(0,0.67) (0.67,0.67)};
    \addplot [black!40!white,mark=none,dashed, thin] coordinates {(0.67,0) (0.67,0.67)};
    \node[mark=none] at (axis cs:0.68,0.21){$\deg_1(\Phi^2((\bbS^1\times\bbS^2)\vee\bbS^1))$};
    \end{axis}
    \end{tikzpicture}
&
    \begin{tikzpicture}[scale=0.45]
    \begin{axis} [ 
    ticklabel style = {font=\Large},
    axis y line=middle, 
    axis x line=middle,
    ytick={0.5,0.6,0.67,0.95},
    yticklabels={,$\zeta_2$,,$\pi$},
    xtick={0.5,0.6,0.95},
    xticklabels={$\tfrac{\pi}{2}$,$\zeta_2$, $\pi$},
    xmin=-0.015, xmax=1.1,
    ymin=0, ymax=1.1,]
    \addplot [mark=none] coordinates {(0,0) (1,1)};
    \addplot [thick,color=black!10!white,fill=black!10!white, 
                    fill opacity=0.4]coordinates {
            (0.6,0.95)
            (0.6,0.6)
            (0.95,0.95)
            (0.6,0.95)};
    \addplot [black!40!white,mark=none,dashed, thin] coordinates {(0,0.6) (0.6,0.6)};
    \addplot [black!40!white,mark=none,dashed, thin] coordinates {(0.6,0) (0.6,0.6)};
    \node[mark=none] at (axis cs:0.68,0.21){$\deg_2(\Phi^2((\bbS^1\times\bbS^2)\vee\bbS^1))$};
    \end{axis}
    \end{tikzpicture}
&
    \begin{tikzpicture}[scale=0.45]
    \begin{axis} [ 
    ticklabel style = {font=\Large},
    axis y line=middle, 
    axis x line=middle,
    ytick={0.5,0.6,0.67,0.95},
    yticklabels={,$\zeta_2$,,$\pi$},
    xtick={0.5,0.6,0.95},
    xticklabels={$\tfrac{\pi}{2}$,$\zeta_2$, $\pi$},
    xmin=-0.015, xmax=1.1,
    ymin=0, ymax=1.1,]
    \addplot [mark=none] coordinates {(0,0) (1,1)};
    \addplot [thick,color=black!10!white,fill=black!10!white, 
                    fill opacity=0.4]coordinates {
            (0.6,0.95)
            (0.6,0.6)
            (0.95,0.95)
            (0.6,0.95)};
    \addplot [black!40!white,mark=none,dashed, thin] coordinates {(0,0.6) (0.6,0.6)};
    \addplot [black!40!white,mark=none,dashed, thin] coordinates {(0.6,0) (0.6,0.6)};
    \addplot[barccolor,mark=*] (0, 0.6) circle (2pt) node[above right,barccolor]{\Large\textsf{1}};
    \node[mark=none] at (axis cs:0.68,0.21){$\deg_3(\Phi^2((\bbS^1\times\bbS^2)\vee\bbS^1))$};
    \end{axis}
    \end{tikzpicture}
\end{tabular}
    \caption{\emph{Top row:} persistence diagram of the persistent $2$-cup module of $\bbT^2\vee\bbS^3$ in degree $1,2$ and $3$. \emph{Bottom row:} persistence diagram of the persistent $2$-cup module of $(\bbS^1\times\bbS^2)\vee\bbS^1$ in degree $1,2$ and $3$.}
    \label{fig:ell barcodes 2}
\end{figure} 

For the persistence diagram of persistent $2$-cup modules, see Fig.~\ref{fig:ell barcodes 2} which we use to estimate the bottleneck distance between the barcodes associated to each vertical two diagrams. Fortunately, for the second and third item below, we are able to get the precise value for $\db$. This is because in both cases, matching all points to the diagonal is an optimal matching.
\begin{itemize}
    \item $\db\left( \barc\!\left(\deg_1\left(\cupmodule^2(\bbT^2\vee\bbS^3)\right)\right),\barc\!\left(\deg_1\left(\cupmodule^2((\bbS^1\times\bbS^2)\vee\bbS^1)\right)\right)\right)\leq \frac{1}{2}\left(\pi-\frac{2\pi}{3}\right)=\frac{\pi}{6}$.
    \item $\db\left( \barc\!\left(\deg_2\left(\cupmodule^2(\bbT^2\vee\bbS^3)\right)\right),\barc\!\left(\deg_2\left(\cupmodule^2((\bbS^1\times\bbS^2)\vee\bbS^1)\right)\right)\right)=\frac{1}{2}\cdot\frac{2\pi}{3}=\frac{\pi}{3}$.
    \item $\db\left( \barc\!\left(\deg_3\left(\cupmodule^2(\bbT^2\vee\bbS^3)\right)\right),\barc\!\left(\deg_3\left(\cupmodule^2((\bbS^1\times\bbS^2)\vee\bbS^1)\right)\right)\right)=\frac{\zeta_2}{2}$.
\end{itemize}

Now consider the case $p>3$. Recall that $\VR_r(\bbS^n)$ is homotopy equivalent to $\bbS^n$ for all $r\in \left(0,\zeta_n\right)$. Thus, $\VR_r\left(\bbT^2\vee\bbS^3\right)$ is homotopy equivalent to $\bbT^2\vee\bbS^3$ for all $r\in (0,\zeta_3)$. As a consequence, there are no bars in the degree-$p$ barcode of $\bbT^2\vee\bbS^3$ whose length is larger than $\pi-\zeta_3$.
Via a similar discussion, there are no bars in the degree-$p$ barcode of $(\bbS^1\times\bbS^2)\vee\bbS^1$ whose length is larger than $\pi-\zeta_2$. 
Therefore, for all $\ell$, we have
\[\db\left( \barc\!\left(\deg_p\left(\cupmodule^\ell(\bbT^2\vee\bbS^3)\right)\right),\barc\!\left(\deg_p\left(\cupmodule^\ell((\bbS^1\times\bbS^2)\vee\bbS^1)\right)\right)\right)\leq \frac{\pi-\zeta_3}{2},\]
which holds because the right-hand side is an upper bound for the cost of matching all bars to the diagonal.

Combining the above discussion, we obtain
\begin{equation}\label{eq:ell-p-barcode}
   \max_{\ell,p}\db\left( \barc\!\left(\deg_p\left(\cupmodule^\ell(\bbT^2\vee\bbS^3)\right)\right),\barc\!\left(\deg_p\left(\cupmodule^\ell((\bbS^1\times\bbS^2)\vee\bbS^1)\right)\right)\right)=\frac{\pi}{3}. 
\end{equation}

\noindent\textbf{Rank invariants of persistent cup modules.}
We use Fig.~\ref{fig:ell barcodes}, Fig.~\ref{fig:ell barcodes 2} and Eqn.~(\ref{eq:dgm of g.flag}) to obtain the persistence diagrams of $\bbT^2\vee\bbS^3$ and $(\bbS^1\times\bbS^2)\vee\bbS^1$; see Fig.~\ref{fig:dgm of p. cup module}. Then we use the relationship between rank invariant and persistence diagram, i.e., Eqn.~\ref{eq:mobius I = sum dgm}, to obtain the rank invariant of persistent cup modules of $\bbT^2\vee\bbS^3$ and $(\bbS^1\times\bbS^2)\vee\bbS^1$. 
See Fig.~\ref{fig:torus_sphere_VR_flag_rank}.

\begin{figure}[H]
\centering
     \begin{tikzpicture}[scale=0.6]
    \begin{axis} [ 
    ticklabel style = {font=\large},
    axis y line=middle, 
    axis x line=middle,
    ytick={0.5,0.67,0.95},
    yticklabels={$\tfrac{\pi}{2}$,$\tfrac{2\pi}{3}$,$\pi$},
    xtick={0.5,0.58,0.95},
    xticklabels={$\tfrac{\pi}{2}$,$\zeta_3$, $\pi$},
    xmin=-0.015, xmax=1.1,
    ymin=0, ymax=1.1,]
    \addplot [mark=none] coordinates {(0,0) (1,1)};
    \addplot[barccolor,mark=*] (0, 0.67) circle (2pt) node[right,barccolor]{};
    \addplot[barccolor,mark=none] (0, 0.8) node[right,barccolor]{ \textsf{$\begin{pmatrix}
    2 & 0\\
    1& 1\\
    0& 0
    \end{pmatrix}$}};
    \addplot[barccolor,mark=*] (0, 0.58) circle (2pt) node[right,barccolor]{};
    \addplot[barccolor,mark=none] (0, 0.45) node[right,barccolor]{ \textsf{$\begin{pmatrix}
    2 & 0\\
    1& 1\\
    1& 0
    \end{pmatrix}$}};
    \addplot [black!40!white,mark=none,dashed, thin] coordinates {(0,0.58) (0.58,0.58)};
    \addplot [black!40!white,mark=none,dashed, thin] coordinates {(0.58,0) (0.58,0.58)};
    \addplot [thick,color=black!10!white,fill=black!10!white, 
                    fill opacity=0.4]coordinates {
            (0.58,0.95)
            (0.58,0.58)
            (0.95,0.95)
            (0.58,0.95)};
\node[mark=none] at (axis cs:0.68,0.21){$\dgm(\Phi(\bbT^2\vee\bbS^3))$};
    \end{axis}
    \end{tikzpicture}
    \hspace{1.5cm}
    \begin{tikzpicture}[scale=0.6]
    \begin{axis} [ 
    ticklabel style = {font=\large},
    axis y line=middle, 
    axis x line=middle,
    ytick={0.5,0.67,0.95},
    yticklabels={$\tfrac{\pi}{2}$,$\tfrac{2\pi}{3}$,$\pi$},
    xtick={0.5,0.6,0.95},
    xticklabels={$\tfrac{\pi}{2}$,$\zeta_2$, $\pi$},
    xmin=-0.015, xmax=1.1,
    ymin=0, ymax=1.1,]
    \addplot [mark=none] coordinates {(0,0) (1,1)};
    \addplot[barccolor,mark=*] (0, 0.67) circle (2pt) node[right,barccolor]{};
    \addplot[barccolor,mark=none] (0, 0.8) node[right,barccolor]{ \textsf{$\begin{pmatrix}
    2 & 0\\
    0& 0\\
    0& 0
    \end{pmatrix}$}};
    \addplot[barccolor,mark=*] (0, 0.6) circle (2pt) node[right,barccolor]{};
    \addplot[barccolor,mark=none] (0, 0.45) node[right,barccolor]{ \textsf{$\begin{pmatrix}
    2 & 0\\
    1& 0\\
    1& 1
    \end{pmatrix}$}};
    \addplot [black!40!white,mark=none,dashed, thin] coordinates {(0,0.6) (0.6,0.6)};
    \addplot [black!40!white,mark=none,dashed, thin] coordinates {(0.6,0) (0.6,0.6)};
    \addplot [thick,color=black!10!white,fill=black!10!white, 
                    fill opacity=0.4]coordinates {
            (0.6,0.95)
            (0.6,0.6)
            (0.95,0.95)
            (0.6,0.95)};
\node[mark=none] at (axis cs:0.68,0.21){$\dgm(\Phi((\bbS^1\times\bbS^2)\vee\bbS^1))$};
    \end{axis}
    \end{tikzpicture}
\caption{Persistence diagrams $\dgm\left(\Phi(\VR\left(\bbT^2\vee\bbS^3\right))\right)$ (left) and $\dgm\left(\Phi(\VR\left((\bbS^1\times\bbS^2)\vee\bbS^1\right))\right)$ (right) of persistent cup modules (up to degree $3$) arising from Vietoris-Rips filtrations of $\bbT^2\vee\bbS^3$ and $(\bbS^1\times\bbS^2)\vee\bbS^1$, respectively. These diagrams are obtained by combining Fig.~\ref{fig:ell barcodes} and Fig.~\ref{fig:ell barcodes 2}. See Ex.~\ref{sec:flag>cup}. 
} 
    \label{fig:dgm of p. cup module}
\end{figure}
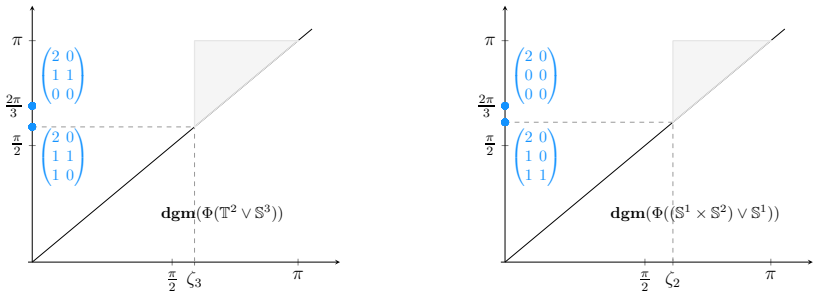

\begin{proposition}
We have
\begin{equation}\label{eq:dE-rank-cup-module}
    d_{\mathrm{E}}\left(\rank\left(\Phi(\VR\left(\bbT^2\vee\bbS^3\right))\right),\rank\left(\Phi(\VR\left((\bbS^1\times\bbS^2)\vee\bbS^1\right))\right)\right)=\frac{\pi}{3}.
\end{equation}
\end{proposition}
\begin{proof}
For notational simplicity, we denote 
\[\bfJ_{1}:=\rank\left(\Phi(\VR\left(\bbT^2\vee\bbS^3\right))\right)\text{ and }\bfJ_{2}:=\rank\left(\Phi(\VR\left((\bbS^1\times\bbS^2)\vee\bbS^1\right))\right).\] 
Suppose that $\bfJ_{1}$ and $\bfJ_{2}$ are $\epsilon$-eroded, which means $\bfJ_{1}(I)\geq \bfJ_{2}(I^\epsilon)$ and $\bfJ_{2}(I)\geq \bfJ_{1}(I^\epsilon)$, for all $I\in\Int$. We take $I_0:=[\frac{\pi}{3}-\delta,\frac{\pi}{3}+\delta]$ for $\delta$ sufficiently small, so that the associated point of $I_0$ in the upper-diagonal half plane is very close to the point $(\frac{\pi}{3},\frac{\pi}{3})$. Then, we have that for any $t<\frac{\pi}{3}-2\delta,$
\[\bfJ_{2}(I_0)= \begin{pmatrix}
    2 & 0\\
    1& 0\\
    1& 1
    \end{pmatrix}\text{ and } \bfJ_{1}(I_0^t)=\begin{pmatrix}
    2 & 0\\
    1& 1\\
    1& 0
    \end{pmatrix}\text{ or }\begin{pmatrix}
    2 & 0\\
    1& 1\\
    0& 0
    \end{pmatrix}\]
are non-comparable.
Therefore, in order for the inequality $\bfJ_{2}(I_0)\geq \bfJ_{1}(I_0^\epsilon)$ to hold, it must be true that $\epsilon\geq \frac{\pi}{3},$ implying that $d_{\mathrm{E}}(\bfJ_{1},\bfJ_{2})\geq\frac{\pi}{3}.$

Via a discussion similar to the proof of Prop.~\ref{prop:erosion-comp}, we see that for any $\epsilon>\frac{\pi}{3}$ and any $I\in \Int$, 
\[\bfJ_{1}(I^\epsilon)=\begin{pmatrix}
    0 & 0\\
    0& 0\\
    0& 0
    \end{pmatrix}\text{  and  }\bfJ_{2}(I^\epsilon)=\begin{pmatrix}
    0 & 0\\
    0& 0\\
    0& 0
    \end{pmatrix}.\] 
Thus, $d_{\mathrm{E}}(\bfJ_1,\bfJ_2)\leq\frac{\pi}{3}.$
\end{proof}

\medskip\noindent\textbf{Persistent homology.}
For the persistent homology of these two spaces $\bbT^2\vee\bbS^3$ and $(\bbS^1\times\bbS^2)\vee\bbS^1$, we have
\begin{itemize}
    \item $\bfH _*\left( \VR\left(\bbT^2\vee\bbS^3\right) \right)\big\vert_{(0,\zeta_3)}\cong \bfH _*\left( \VR\left((\bbS^1\times\bbS^2)\vee\bbS^1\right)\right)\big\vert_{(0,\zeta_3)}$, and
    \item $\bfH _*\left( \VR\left(\bbT^2\vee\bbS^3\right) \right)\big\vert_{(\pi,\infty)}= \bfH _*\left( \VR\left((\bbS^1\times\bbS^2)\vee\bbS^1\right)\right)\big\vert_{(\pi,\infty)}$ is trivial
\end{itemize}
Thus, for any degree $p$, we have
\begin{equation}
\label{eq:homology}
   \di\left( \bfH _p\left( \VR\left(\bbT^2\vee\bbS^3\right) \right), \bfH _p\left( \VR\left((\bbS^1\times\bbS^2)\vee\bbS^1\right)\right) \right)\leq \tfrac{\pi-\zeta_3}{2}. 
\end{equation}

\medskip\noindent\textbf{Comparison of the different invariants.}
In the example of $\bbT^2\vee\bbS^3$ and $(\bbS^1\times\bbS^2)\vee\bbS^1$, using the information we have about the underlying spaces, persistent homology, Eqn.~(\ref{eq:homology}), and persistent cup-length (or LS-category) invariants, Eqn.~(\ref{eq:dE-cup-cat}), have similar discriminating power. But the rank invariant of persistent cup modules, Eqn.~(\ref{eq:dE-rank-cup-module}), and the barcode of persistent $\ell$-cup modules, Eqn.~(\ref{eq:ell-p-barcode}), both provide a better approximation, $\frac{\pi}{6}$, of the Gromov-Hausdorff distance $ d_{\mathrm{GH}}\left(\bbT^2\vee\bbS^3,(\bbS^1\times\bbS^2)\vee\bbS^1\right)$ than the persistent homology and the persistent cup-length (or LS-category) invariants.

\subsection{Further discussion on persistent cup modules}

\medskip\noindent\textbf{Retrieving $\cupprod(\Xfunc)$ from persistent cup modules.} On the other hand, the persistent cup-length invariant can be computed from the barcode of persistent $\ell$-cup modules via the following proposition. 
\begin{proposition}
\label{prop:dgm_p,l}
Let $\Xfunc$ be a persistent space. Then, for any interval $[a,b]$,
\begin{align*}
    \cupprod(\Xfunc)([a,b])
    = & \max_{[c,d]\supseteq  [a,b]} \max \left\{\ell\in\N^+\mid [c,d]\in \barc\!\left(\cupmodule^\ell(\Xfunc)\right)\right\},
\end{align*}
with the convention that $\max\emptyset=0.$
\end{proposition}

\begin{proof}
Let $[a,b]$ be an interval. 
Then, we compute:
\begin{align*}
&\cupprod(\Xfunc)([a,b]) \\
=& \len\left(\image \left(\bfH^*( X_b) \to \bfH^*( X_a) \right)\right)   \\
=& \max \left\{\ell\in\N^+\mid \left(\image \left(\bfH^{ + }( X_b) \to \bfH^{ + }( X_a) \right)\right)^{\ell}\neq 0\right\} \hspace{0.2cm}\text{(by Defn.~\ref{def:cup-length})}\\
=& \max \left\{\ell\in\N^+\mid \image \left((\bfH^{ + }( X_b) )^{\ell}\to (\bfH^{ + }( X_a) )^{\ell}\right)\neq 0\right\} \\
=& \max \left\{\ell\in\N^+\mid \text{there exists an interval }[c,d] \text{ s.t. } [c,d]\supseteq  [a,b] \text{ and } [c,d]\in \barc\!\left(\cupmodule^\ell(\Xfunc)\right)\right\} \\
=& \max_{[c,d]\supseteq  [a,b]} \max \left\{\ell\in\N^+\mid [c,d]\in \barc\!\left(\cupmodule^\ell(\Xfunc)\right)\right\}. 
\end{align*}
Here we applied the fact that for a graded algebra morphism $f:( A,+_ A,\bullet_ A)\to ( B,+_ A,\bullet_ B)$, $(\image(f))^{\ell}=\image(f^{\ell}: A^{\ell}\to  B^{\ell})$ because they are both generated by the set $\{f(a_1)\bullet\cdots\bullet f(a_\ell)\mid a_i\in  A\}.$
\end{proof} 

\begin{remark}
Prop.~\ref{prop:dgm_p,l} shows that the map $\bfJ(\Xfunc):\Int\to\N$ given by
\begin{center}
    $ [c,d]\mapsto \max \left\{\ell\in\N^+\mid [c,d]\in \barc\!\left(\cupmodule^\ell(\Xfunc)\right)\right\}.$
\end{center} recovers the persistent cup-length invariant in the same way as the persistent cup-length diagram, see Thm.~\ref{thm:tropical_mobius}. Unlike the persistent cup-length diagram, $\bfJ(\Xfunc)$ is independent of the choice of representative cocycles. 
\end{remark}

\medskip\noindent\textbf{Computation of persistent $\ell$-cup modules.} The barcode for persistent $\ell$-cup modules can be computed from any given set of representative cocycles for $\bfH^{ + }(\Xfunc)$ as follows. Let us denote by $V=\langle G\rangle$ the sub-vector space of $V$ generated by a set $G\subset V$ of vectors. 
\begin{proposition}
    Let $\sigma:=(\sigma_I)_{I\in \barc\!\left(\bfH^+(\Xfunc)\right)}$ be a set of representative cocycles for $\bfH^{ + }(\Xfunc)$ and let $\ell\geq1$ be any positive integer. Then:
  
  For each $\ell$, the persistence diagram of the persistent $\ell$-cup module $\left(\bfH^{ + }(\Xfunc)\right)^\ell$ can be obtained as the M\"obius inversion of the rank invariant of the persistent $\ell$-cup module given point-wisely, for any $a\leq b$ in $\bbR$, as the dimension of the vector space 
   $$\image\left(\Phi^\ell(X_b)\to\Phi^\ell(X_a)\right)=\left\langle\left[\sigma_{I_{i_1}}\smile\ldots\smile\sigma_{I_{i_\ell}}\right]_a 
   \mid
\mathrm{supp}\left(\sigma_{I_{i_1}}\smile\ldots\smile\sigma_{I_{i_\ell}}\right)\supseteq  [a,b]\right\rangle.$$
      
      Furthermore, for any cup-power $\ell$ and for any dimension $p$,  the persistence diagram $\barc\!\left(\deg_p\left(\cupmodule^\ell(\Xfunc)\right)\right)$ of the degree-$p$ persistent module  $\deg_p\left(\left(\bfH^{ + }(\Xfunc)\right)^\ell\right)$ can be obtained as the M\"obius inversion of the rank invariant of the persistent $\ell$-cup module given point-wisely, for any $a\leq b$ in $\bbR$, as the dimension of the vector space:
   \begin{align*}       &\image\left(\deg_p\left(\Phi^\ell(X_b)\right)\to\deg_p\left(\Phi^\ell(X_a)\right)\right)=\\
&=\left\langle\left[\sigma_{I_{i_1}}\smile\ldots\smile\sigma_{I_{i_\ell}}\right]_a
    \mid
\mathrm{supp}\left(\sigma_{I_{i_1}}\smile\ldots\smile\sigma_{I_{i_\ell}}\right)\supseteq  [a,b],\sum_{j=1}^{\ell} \dim(\sigma_{I_{i_j}})=p\right\rangle.
   \end{align*}
     \end{proposition}

\begin{remark}
We can extract a basis of the vector space $\image\left(\Phi^\ell(X_b)\to\Phi^\ell(X_a)\right)$ from the spanning set given above  by first writing these vectors $\left[\sigma_{I_{i_1}}\smile\ldots\smile\sigma_{I_{i_\ell}}\right]_a$ as linear combinations of the basis elements $\{[\sigma_I]_a\}$, i.e.~as rows, and then row-reducing that matrix. An analogous argument holds for extracting a basis for  $\image\left(\deg_p\left(\Phi^\ell(X_b)\right)\to\deg_p\left(\Phi^\ell(X_a)\right)\right)$.
\end{remark}

\subsection{Persistent cup module as a 2-dimensional persistence module}
\label{sec:2d cup module}
The product of poset categories $(\N^+,\leq)$ and $(\bbR,\leq)$ is defined to be the Cartesian product $\N^+\times \bbR$ equipped with the partial order: $(\ell,t)\leq (\ell',t')$ if and only if $\ell\leq \ell'$ and $t\leq t'$. 

Let $\Xfunc$ be a filtration of topological spaces. The persistent cup module of $\Xfunc$ also has the structure of a 2-dimensional persistence module, because it can be viewed as a functor  
\begin{center}
    $ \cupmodule(\Xfunc):(\N^+,\leq)\times(\bbR,\leq) \to \GVect^{\mathrm{op}}$  with $(\ell,t)\mapsto (\bfH^{ + }\left( X_t)\right)^{\ell}$.
\end{center}
This is due to the factor that for any $\ell\leq \ell'$ and $t\leq t'$, we have the following commutative diagram:
\[
\begin{tikzcd}
(\bfH^{ + }\left( X_t)\right)^{\ell}  
& 
(\bfH^{ + }( X_t))^{\ell'}  
\ar[l,hook']
\\
(\bfH^{ + }( X_{t'}))^{\ell} 
\ar[u]
& (\bfH^{ + }( X_{t'}))^{\ell'},
\ar[u]
\ar[l,hook']
\end{tikzcd}
\]
where the row morphisms are natural linear inclusions and the column morphisms are induced maps of $\bfH^{ + }( X_t)\leftarrow \bfH^{ + }( X_{t'})$.

Unlike the $1$-dimensional case where indecomposable persistence module can be characterized by intervals, the indecomposables of 2-dimensional persistence modules are much more complicated and in most cases not finite \citep{leszczynski1994representation,leszczynski2000tame,bauer2020cotorsion}. A simple type of 2-dimensional persistence modules are those that can be decomposed into rectangle modules, but the persistent cup modules are not necessarily rectangle decomposable, see Ex.~\ref{ex:not rectangle} below.

\begin{example}[Persistent cup modules are not rectangle decomposable] 
\label{ex:not rectangle}
Recall from Fig.~\ref{fig:pinched_torus} the filtration $\Xfunc=\{ X_t\}_{t\geq0}$ of a pinched $2$-torus $\bbT^2$ and its total barcode. We directly compute the persistence module $(\bfH^{ + }(\Xfunc))^{ 2}$, and see that it is only non-zero in degree $2$ and its barcode consists of only one bar $[2,3)$. However, the barcode of $\bfH^2(\Xfunc)$ has a single bar $[2,\infty).$ Thus, the persistent cup module of $\Xfunc$ has an indecomposable (in degree $2$) given by 
\begin{center}
\begin{tikzcd}
\bfH^2(\Xfunc):&
\field
&
\field
\ar[l,"\id_\field" above]
&&
\field
\ar[ll,"\id_\field" above]
\\
(\bfH^{ + }(\Xfunc))^{ 2}:&
\field
\ar[u,"\id_\field" left]
&
\field
\ar[l,"\id_\field" below]
\ar[u,"\id_\field" right]
&&
0
\ar[ll,"0" below]
\ar[u,"0" right]
\end{tikzcd}
\end{center}
\end{example}

The persistent cup module has the special structure that its row maps are inclusions of vector spaces, implying that (1) each column is a persistence submodule of any column left to it and (2) each row is a flag of vector spaces (see Defn.~\ref{def:flag}).
From (1), we see that the barcodes of all the columns are closely related to each other: it follows from \citep[Thm.~4.2]{botnan2020decomposition} that if $\Mfunc$ is a persistence submodule of $\Nfunc$, i.e. there is a monomorphism from $\Mfunc$ to $\Nfunc$, then there is a canonical injection from the barcode $\barc(\Mfunc)$ to the barcode $\barc(\Nfunc)$ sending each bar $\linterval  b,d\rinterval \in \barc(\Mfunc)$ to $\linterval  b,d'\rinterval \in \barc(\Nfunc)$ for some $d'\leq d$. 
From (2), we were inspired to study the persistent cup modules as persistent (graded) flags in the previous section.

\section{Discussion}

\medskip\noindent
\textbf{Understanding the stability of persistent cup modules when represented as 2D modules.} Note that: (i) The right inequality in Rem.~\ref{rmk:Phi-stability} can also be interpreted as the stability of the persistence cup module as a 2D persistence module  $\bbR\times \N \to\Vect^{op}$: indeed we can consider the strict flow (a.k.a.~linear family of translations) $\Omega=(\Omega_\epsilon)_{\epsilon\geq0}$ on the poset $\bbR\times \N$, given by $(t,\ell)\mapsto (t+\epsilon,\ell)$. (ii) A flag is also a functor $\N\to \Vect^{op}$ and thus a persistent flag is also a functor $\bbR\to (\Vect^{op})^{\N}$; in particular the interleaving distance of persistent flags coincides with the restriction of the interleaving distance on $((\Vect^{op})^{\N})^{\bbR}$ on persistent flags. (iii) By construction of $\Omega$, the category with a flow $((\Vect^{op})^{\bbR\times \N},-\cdot\Omega)$ is flow-equivariantly isomorphic to  $(((\Vect^{op})^{\N})^{\bbR},-\cdot\mathbf{S})$ (where for each $\epsilon\geq0$, $\mathbf{S}_\epsilon:t\mapsto t+\epsilon$), in the sense of \citep[Defn.~4.2]{deSilva2018}, and thus the corresponding interleaving distances are isometric \citep[Thm.~4.3]{deSilva2018}.

\medskip
\noindent
\textbf{Extension to more general domain posets.} Note that all of our results in this paper can be generalized from the setting of persistence modules over $\bbR$ to the setting of persistence modules over more general posets $P$, for instance for $P=\bbR^n$. 
If we change $\bbR$ with general posets $P$, then Thm.~\ref{thm:stab-per-inv}'s statement and  proof (as well as the other theorems) can be changed and expanded: intervals $\Int$ and their thickening flow $\nabla$ (which gives rise to the erosion distance) should be replaced by $(\mathbf{con}(P),\mathbf{con}(\Omega))$, where $\mathbf{con}(\Omega)$ should be the thickening flow on the poset of connected subposets $\mathbf{con}(P)$ induced by the given flow $\Omega$ on $P$. We did not develop further the ideas towards this direction since this is not the main point of this paper.

\appendix

\begin{appendices}

\section{The category \texorpdfstring{$\Flag_{\operatorname{fin}}$}{flag}}

The direct sum in $\Flag_{\operatorname{fin}}$ is given by $\Vflag \oplus \Wflag :=(V\oplus W, \mathrm{F} (V\oplus W))$, where $\mathrm{F}_\ell(V\oplus W):=V_\ell\oplus W_\ell$ for each $\ell$. Let $f:\Vflag \to \Wflag $ be a morphism of flags. The kernel of $f$ is the vector space injection $\ker(f)\subseteq V$, where $\ker(f)$ is endowed with the filtation given by $\mathrm{F}_l\ker(f):=\ker(f)\cap V_\ell$. The cokernel of $f$ is the sujection of vector spaces $W\twoheadrightarrow \operatorname{coker}(f)=W/\image(f)$, where  $\operatorname{coker}(f)$ is endowed with the filtration given by $\mathrm{F}_l\operatorname{coker}(f):=\operatorname{coker}(f\vert_{V_\ell}:V_\ell\to W_\ell)=W_\ell/\image(f\vert_{V_\ell}).$

\begin{proposition}[Rank is complete] \label{prop:dim_complete}
Two flags $\Vflag $ and $\Wflag $ are isomorphic, if and only if they have the same dimension.
\end{proposition}
\begin{proof}
If two flags $\Vflag $ and $\Wflag $ are isomorphic, then there are morphisms $f:\Vflag \to \Wflag $ and $g:\Wflag \to \Vflag $ such that $g\circ f=\id_{\Vflag }$ and $f\circ g=\id_{\Wflag }$. 
It immediately follows that for any $l\in\N^+$, $f\vert_{\mathrm{F}_\ell V}$ and $g\vert_{\mathrm{F}_\ell W}$ induce an isomorphism between $\mathrm{F}_\ell V$ and $\mathrm{F}_\ell W$. Thus, $\dimension(V_\ell)=\dimension(W_\ell)$ for any $\ell$. 

Conversely, we first consider the case of finite-depth flags. If two finite-depth flags $\Vflag $ and $\Wflag $ have the same dim, then one can construct an isomorphism between them inductively. Start with the smallest space $\mathrm{F}_n V$ and $\mathrm{F}_n W$ in each flag. Because $\mathrm{F}_n V$ and $\mathrm{F}_n W$ have the same dimension, we can construct an isomorphism between them. Extend this isomorphism to an isomorphism between $\mathrm{F}_{n-1} V$ and $\mathrm{F}_{n-1} W$. In the case of infinite-depth flags, because we are considering flags of finite-dimensional vector spaces, every flag stabilizes in finitely many steps. Thus, we can apply a similar argument as in the case of finite-depth flags.
\end{proof}
\end{appendices}

\section*{Conflict of Interest}
On behalf of all authors, the corresponding author states that there is no conflict of interest.

\nocite{label}
\bibliography{sn-bibliography}

\end{document}